%% file: paper_Dmodule_1_.tex
\documentclass{amsart}
\input{header.tex}

\title[Bernstein--Sato Theory for $\cD$-modules in Positive Characteristic]{Bernstein--Sato Theory for $\cD$-modules in Positive Characteristic}

\author[D.~Takeuchi]{Daichi Takeuchi}
\address{
Department of Mathematics,
Institute of Science Tokyo,
2-12-1 Ookayama, Meguro-ku, Tokyo, 152-8551, Japan
}
\email{
daichi.takeuchi4@gmail.com
}

\begin{document}
%%%%%%%%%%%%%%%%%%%%%%%%%%%%%%%%%%%%
%%%%%%%%%%%%%%%%%%%%%%%%%%%%%%%%%%%%

\begin{abstract}
In this article, we develop a positive characteristic analogue of the Bernstein--Sato theory for holonomic $\cD$-modules in the complex setting. We work with $\cD$-modules $\sM$ on a Noetherian regular $F$-finite $\F_p$-scheme $X$, and define their Bernstein--Sato roots as $p$-adic integers. When $\sM=\cO_X$, this recovers Bitoun's definition. 

When  $\sM$ arises from a locally finitely generated unit $F^e$-module and $X$ is of finite type over an $F$-finite field, we show that the roots are finite and rational, generalizing Bitoun's result. In the course of the proof, we also develop a related theory for Cartier modules.

%We also establish a Kashiwara--Malgrange-type theorem describing the relationship between the Bernstein--Sato roots and the monodromy eigenvalues of the tame vanishing cycles of the perverse \'etale $\F_p$-sheaf  corresponding to $\sM$ via the Riemann--Hilbert correspondence. 
\end{abstract}

\subjclass[2020]{14F10, 13A35}
\keywords{Bernstein--Sato theory, $\cD$-modules in positive characteristic, Nearby cycles}
 
\maketitle

\tableofcontents

%%%%%%%%%%%%%%%%%%%%%%%%%%%%%%%%%%%%%%%%%%%%%%%%%%%%%%%%%%%
\section{Introduction}\label{sec:intro}
%%%%%%%%%%%%%%%%%%%%%%%%%%%%%%%%%%%%%%%%%%%%%%%%%%%%%%%%%%%

Let $X$ be a smooth algebraic variety over the complex numbers $\C$, and let $f\colon X\to \A^1_{\C}$ be a $\C$-morphism. The \emph{Bernstein--Sato polynomial} of $f$, introduced independently by Bernstein and Sato--Shintani, is defined as follows. A key observation is that there exists, locally on $X$, a functional equation of the form 
\[
P\cdot f^{s+1}=b(s)f^s, 
\]
where $b(s)$ is a nonzero polynomial in $\C[s]$ and $P\in\cD_{X/\C}[s]$. Here $\cD_{X/\C}$ denotes the sheaf of algebraic differential operators on $X$. 
The Bernstein--Sato polynomial, denoted by $b_f(s)$, is the unique monic polynomial of minimal degree satisfying such an equation.

The notion of the Bernstein--Sato polynomial extends to holonomic $\cD$-modules $\sM$ on $X$. In \cite{Kas78} (in the analytic setting), Kashiwara proves that, for any local section $u\in \sM$, there exists a nonzero polynomial $b(s)\in \C[s]$ and an operator $P\in \cD_{X/\C}[s]$ such that 
\[
P\cdot (u\otimes f^{s+1})=u\otimes b(s)f^s, 
\]
where both sides are viewed as sections of $\sM\otimes_{\cO_X}\cO_X[f^{-1},s]\cdot f^s$, and the left-hand side is given by the natural action of $P$ on this module. 
The \emph{Bernstein--Sato polynomial of $f$ with respect to $u$} is defined as the monic polynomial of minimal degree satisfying such a functional equation. The case $(\sM,u)=(\cO_X,1)$, where $\cO_X$ is endowed with the natural $\cD_{X/\C}$-module structure, recovers the polynomial $b_f(s)$. 

%In this article, we refer to the roots of Bernstein--Sato polynomials as \emph{Bernstein--Sato roots}. Bernstein--Sato roots are fundamental invariants in the $\cD$-module theory, and related to several invariants arising in algebraic geometry and topology. For instance, in~\cite{ELSV04}, the authors show that the jumping numbers of $f$ appear as Bernstein--Sato roots of $\cO_X$. Another example is provided by the Kashiwara--Malgrange theorem. Assume that $\sM$ is regular holonomic, and every root of $b_u(s)$ for every $u\in\sM$ is rational; in the case $\sM=\cO_X$, this is a theorem of Kashiwara. Under this assumption, the Bernstein--Sato roots determine the eigenvalues of the monodromy acting on the topological nearby and vanishing cycles of ${\rm DR}(\sM)$, where ${\rm DR}$ denotes the de Rham functor in the Riemann--Hilbert correspondence. Furthermore, using the notion of $V$-filtration, which is constr
%For a detailed account of the latter, we refer to~\cite{MM04, CDM25}. 

%In many cases, it is known that Bernstein--Sato roots are rational; in the case $\sM=\cO_X$, this is a theorem of Kashiwara. Assume that this holds for all $u\in\sM$. Then one can construct the \emph{V-filtration} on the $\cD$-module $\gamma_+\sM$, where $\gamma\colon X\to X\times_{\C}\A^1_{\C}$ denotes the graph embedding of $f$, and $\gamma_+$ denotes the $\cD$-module pushforward. 

%The $V$-filtration is a fundamental invariant that connects 

The roots of Bernstein--Sato polynomials, which we refer to as the \emph{Bernstein--Sato roots}, are fundamental invariants connecting various objects in algebraic geometry and topology. We review some examples in the case where $\sM=\cO_X$ and $u = 1$. In~\cite[THEOREM~2.1]{ELSV04}, the authors establish a relationship between the Bernstein--Sato roots and the jumping exponents of the multiplier ideals of $f$. Another example is provided by the Kashiwara--Malgrange theorem, which shows that the Bernstein--Sato roots determine the eigenvalues of the monodromy acting on the topological nearby cycles complex associated with $f$. Moreover, this theorem extends to regular holonomic $\cD_{X/\C}$-modules with rational Bernstein--Sato roots. More precisely, the rationality ensures the existence of the $V$-filtration, which encodes, on the $\cD$-module side, the information corresponding to the monodromy action via the Riemann–Hilbert correspondence; see~\cite{MM04, CDM25} for a detailed account. In this extension, the rationality of the Bernstein–Sato roots plays a crucial role.

\medskip

An attempt to develop a positive characteristic analogue of the above picture was initiated by Musta\c{t}\u{a} \cite{Mus09} 
in the case $\sM=\cO_X$. 
Let $X$ be a Noetherian regular $F$-finite\footnote{Recall that an $\F_p$-scheme is said to be $F$-finite if the Frobenius endomorphism $X\to X$ is finite. } $\F_p$-scheme, and fix 
an element $f\in \Gamma(X,\cO_X)$. 
A starting point of his work is the well-known observation that, in the complex case, there is a canonical isomorphism 
\begin{equation}\label{equation: gamma+ cong fs}
    \gamma_+\cO_X[f^{-1}]\cong \cO_X[f^{-1},s]\cdot f^s, 
\end{equation}
 where $\gamma$ denotes the graph embedding $X\to X\times_{\C}\A^1_{\C}$ associated with $f$, and $\gamma_+$ is the $\cD$-module pushforward. Since the left-hand side is defined purely algebraically, 
 this provides an algebraic approach to defining the Bernstein--Sato polynomials. Guided by this observation, Musta\c{t}\u{a} considers the corresponding $\cD_{\A^1_X}$-module $B_f\coloneqq \gamma_+\cO_X$ in positive characteristic, and defines Bernstein--Sato polynomials $(b_f^{(m)}(s))_{m\ge0}$ as a family indexed by $m\in\Z_{\ge0}$. His definition reflects the fact that, for an $F$-finite $\F_p$-scheme $X$, the ring $\cD_X$ admits a canonical filtration 
 \[
 \cD_X=\bigcup_{m\ge0}\cD_X^{(m)}, 
 \]
where $\cD_X^{(m)}$ denotes the ring of differential operators of level $m$, that is, the ring of $\cO_X^{p^m}$-linear endomorphisms of $\cO_X$. Roughly speaking, the polynomial $b_f^{(m)}(s)$  is defined in terms of the action of level $m$ operators on $B_f$. 

 Since Musta\c{t}\u{a} considers $\cD_X^{(m)}$-module structures separately for each $m$, 
it is natural to ask whether there exists an analogue formulated directly at the level of $\cD_X$, without fixing $m$. Building on Mustaţă's work, Bitoun \cite{Bit18} formulates such a version and defines Bernstein--Sato roots of $f$ in the positive characteristic setting. 
A key insight of Bitoun is that the role of $\C[s]$ in the complex setting is played by $\sC\coloneqq C(\Z_p,\F_p)$, the ring of locally constant functions $\Z_p\to \F_p$, where $\Z_p$ is equipped with the $p$-adic topology. In fact, in~\cite{JNBQG}, the authors show that 
\begin{equation*}\label{equation: Bf fs}  B_f[f^{-1}]\cong(\cO_X[f^{-1}]\otimes_{\F_p}\sC)\cdot f^s, 
\end{equation*}
where $f^s$ is a formal symbol, as an analogue of \eqref{equation: gamma+ cong fs}. 
Using the fact that the topological space $\operatorname{Spec}(\sC)$ can be identified with $\Z_p$ via evaluation, Bitoun defines Bernstein--Sato roots as $p$-adic integers. Moreover, he proves several important properties of these roots: their finiteness and rationality~\cite[Corollary~2.3.6]{Bit18}, which is an analogue of Kashiwara's theorem, and their relationship with Frobenius jumping exponents~\cite[Theorem~2.4.1]{Bit18}, which is an analogue of~\cite[THEOREM~2.1]{ELSV04}.

%Note, however, that there is no satisfactory notion of multiplicity in this setting, since every ideal of $\sC$ is radical. 

Although the approaches of Musta\c{t}\u{a} and Bitoun are closely related, their precise relationship is subtle, partly because the natural morphism from the $\cD_X^{(m)}$-module considered by Musta\c{t}\u{a} to the corresponding $\cD_X$-module in Bitoun's framework is not injective. 

\medskip

Several authors have attempted to generalize the above results beyond the case  $\sM=\cO_X$. In~\cite{Sta12}, Stadnik considers locally finitely generated unit $F^e$-modules (lfgu $F^e$-modules for short) for some $e\in\Z_{\ge1}$. In~\cite{BS16}, Blickle and St\"abler treat Cartier modules satisfying certain conditions. They define Bernstein--Sato polynomials in their setting, in a manner similar to that of Musta\c{t}\u{a}. 

%Their relationship is explained in \cite[\S6]{BS16}. Note that both approaches essentially follow that of 

In this paper, we follow the approach of Bitoun and work with $\cD_X$-modules rather than $\cD_X^{(m)}$-modules. One advantage of this approach is that it allows us to state and prove an analogue of the Kashiwara--Malgrange theorem. 
More precisely, when the $\cD_X$-module under consideration arises from an lfgu $F^e$-module, 
 we can interpret the results in terms of perverse \'etale $\F_{p^e}$-sheaves via the Riemann--Hilbert correspondence for lfgu $F^e$-modules. A precise statement and further discussion will appear in forthcoming joint work with H.~Kato and E.~Quinlan-Gallego.

We now describe the main results of this paper in detail. We generalize Bitoun's definition of Bernstein--Sato roots to pairs $(\sM,M)$, where $\sM$ is a $\cD_X$-module and $M$ is an $\cO_X$-submodule of $\sM$. By definition, these roots are $p$-adic integers, and we denote their set by $\BSR(M,f)$. 
In this formulation, the role of a section $u\in \sM$ is played by the submodule $M$, via the identification $M=\cO_X\cdot u$. The special case $(\sM,M)=(\cO_X,\cO_X)$ recovers the original definition. 

We now explain the definition of $\BSR(M,f)$. For each $\alpha\in\Z_p$, we define a certain $\cD_X$-module denoted by $N_f(M)_\alpha$ in Section~\ref{section: Nearby Cycles via the Graph Embedding}. When $\alpha\notin\Z$, this module can be described in terms of the $\cD_X$-module \[\sL_{f,\alpha}\coloneqq\cO_X[f^{-1}]\cdot f^\alpha\]
associated with $f^\alpha$, as follows. Consider the tensor product $\sM\otimes_{\cO_X}\sL_{f,\alpha}$, which naturally carries a $\cD_X$-module structure. Then 
\[
N_f(M)_{\alpha}\cong \cD_X(M\cdot f^\alpha), 
\]
where $M\cdot f^\alpha$ denotes the $\cO_X$-submodule of $\sM\otimes_{\cO_X}\sL_{f,\alpha}$ generated by $x\otimes f^\alpha$ with $x\in M$. 
See \S\ref{subsubsection: Non-Integral Case} for details. 

Using the construction $N_f(-)_\alpha$, we define 
\[
\BSR(M,f)\coloneqq\{\alpha\in\Z_p\mid N_{f}(fM)_\alpha \subsetneq N_{f}(M)_\alpha\}. 
\]
In general, $\BSR(M,f)$ need not be finite or rational.  Our first main theorem provides a sufficient condition for $\BSR(M,f)$ to satisfy these properties. Recall that a unit $F^e$-module naturally 
carries a $\cD_X$-module structure (see Construction~\ref{construction: Dmodule str} for details). 
\begin{thm}[{Theorem~\ref{thm: finiteness and rationality of BSR}}]\label{thm: introduction finiteness and rationality of BSR}
Let $(\sM,M)$ be a pair as above. 
 Suppose that $\sM$ arises from an lfgu $F^e$-module  for some $e\in\Z_{>0}$ (in the sense recalled above), and that $M$ is a coherent $\cO_X$-submodule of $\sM$. Then the following hold: 
 \begin{enumerate}
     \item Assume that $(\sM,M)$ satisfies Assumption~\ref{assumption: BSR is bound}. 
 Then the set $\BSR(M,f)$ is finite. 
If, in addition, $M$ is a root of $\sM$, then $\BSR(M,f)\subset\Z_{(p)}=\Z_p\cap \Q$. Moreover, if $M$ is the minimal root, then  
\[
\BSR(M,f)\subset \Z_{(p)}\cap[-1,0]. 
\]
\item The pair $(\sM,M)$ satisfies Assumption~\ref{assumption: BSR is bound} if $X$ admits an open covering $X=\bigcup_iU_i$, where each $U_i$ is of finite type over an $F$-finite field. 
 \end{enumerate}
\end{thm}
Let us comment on this theorem. In contrast to Bitoun's theorem  \cite[Theorem~2.4.1]{Bit18} for $\sM=\cO_X$, the value $0$ can occur as a Bernstein--Sato root. This reflects the possibility that $\sM$ may have a Jordan--H\"older  constituent supported on the zero locus $Z\coloneqq\{f=0\}$. More precisely, one can show that 
$0\notin \BSR(M,f)$ if and only if there exists no nonzero morphism from $\sM$ to an lfgu $F^e$-module supported on $Z$. This follows from 
Lemma~\ref{lem: NffiM=} and Proposition~\ref{prop: description of DXM}.

We briefly explain the proof of Theorem~\ref{thm: introduction finiteness and rationality of BSR}. Let $\sM$ be an lfgu $F^e$-module on $X$, and let $M\subset \sM$ be a coherent $\cO_X$-submodule. Set $q\coloneqq p^e$. As is well known, the unit $F^e$-module structure induces an additive morphism $C\colon \sM\to\sM$ 
satisfying 
\[C(a^{q}x)=aC(x)\] for all local sections $a\in\cO_X$ and $x\in\sM$ \footnote{More precisely, to encode such $C$, one needs to twist $\sM$ by an invertible sheaf. See \S\ref{subsection: Unit Frobenius Modules and Cartier Modules} or \Cref{section: Global Interpretation of Cartier Operators} for details.}. We refer to such pairs $(\sM,C)$ as $q$-Cartier modules. The conditions on $M$ in the statement of the theorem can be interpreted in terms of $q$-Cartier modules: if $M$ is a root of $\sM$, then it is a $q$-Cartier submodule, that is, $C(M)\subset M$. If $M$ is the minimal root, then it is $F$-pure, in the sense that $C(M)=M$ (Lemma~\ref{lem: minimal root is Fpure}). 
 
 Using these interpretations, we reduce the proof of Theorem~\ref{thm: introduction finiteness and rationality of BSR} to the corresponding statements for $q$-Cartier modules. Section~\ref{section: Theory for Cartier Modules} is devoted to developing  the theory of $q$-Cartier modules. Among other things, for $q$-Cartier modules, we introduce the notion of \emph{$\nu$-invariants} of finite level in Definition~\ref{defn: nu} and of infinite level in Definition~\ref{defn: nu infty}. The $\nu$-invariants of infinite level correspond to Bernstein--Sato roots,\footnote{The use of $\nu$-invariants to study Bernstein--Sato roots is inspired by~\cite{MTW05} and~\cite[\S6.5]{Qui21}. } and the finiteness and rationality of $\BSR(M,f)$ 
 are thereby reduced to the corresponding properties of these $\nu$-invariants. In Proposition~\ref{BSRfin} and Theorem~\ref{BSRrat}, we prove that the $\nu$-invariants satisfy these properties under Assumption~\ref{assumption: nu invariant is bound}, which corresponds to Assumption~\ref{assumption: BSR is bound}. 
 Statement~(2) of Theorem~\ref{thm: introduction finiteness and rationality of BSR} is proved in Proposition~\ref{consBl} using a trick of Blickle~\cite[Section~4]{Bli13}.  

\medskip

Although this notion is not used elsewhere in this paper, we also study the test modules $\tau(M,f^t)$ of an $F$-pure $q$-Cartier module $M$, and establish a relationship between the jumping exponents of $\tau(M,f^t)$ and the $\nu$-invariants in Theorem~\ref{BSRFJN}. Since the $\nu$-invariants correspond to the Bernstein--Sato roots, this result can be viewed as a generalization of~\cite[Theorem~2.4.1]{Bit18}. 

We caution that our definition of $\tau(M,f^t)$ differs from that of Blickle~\cite{Bli13}, namely  $\tau(M,\cC,\mathfrak{a}^t)$, where  $\cC=\bigoplus_{n\ge0}C^n\cdot\cO_X$ and $\mathfrak{a}=f\cO_X$. 
Blickle's test module is defined as the minimal submodule satisfying certain conditions at the generic points of $\operatorname{Supp}(M)$, whereas ours is defined so as to agree with $M$ on the open subset $\{f\neq0\}$. Our definition is well suited to the context of Bernstein--Sato theory. We note that, by Proposition~\ref{prop: intrinsic description of test module}, which  gives an intrinsic characterization of our $\tau(M,f^t)$, our test module coincides with that of Blickle when $M$ is $F$-regular and $f$ is a nonzero divisor on $M$. This is precisely the setting considered in~\cite{BS16}.

%\medskip

%In collaboration with Hiroki Kato and Eamon Quinlan-Gallego, we plan to study the Bernstein--Sato roots in the case $\sM=\cO_X$ in detail via computations of the tame vanishing cycles $R\phi_f^{\rm t}(\F_p)$.  

\medskip

Let us mention some possible directions for future work. 
In this paper, we do not consider the $V$-filtration in our setting. In~\cite{Sta14, Sta16}, the authors study positive characteristic analogues of the $V$-filtration. It would be interesting to compare their constructions with ours. 

In \cite{JNBQG}, the authors study Bernstein--Sato roots in the case where $X$ is not regular. It would be interesting to investigate the relationship between their results and ours. 

\medskip

Finally, we describe the organization of this paper. In Section~\ref{section: Theory for Cartier Modules}, we study Cartier modules. We introduce the notion of $\nu$-invariants and prove their finiteness and rationality under Assumption~\ref{assumption: nu invariant is bound}. 

After reviewing the basics of differential operators in positive characteristic in Section~\ref{section: Preliminaries on Differential Operators}, 
we define the Bernstein--Sato roots for $(\sM,M)$ in Section~\ref{section: Nearby Cycles via the Graph Embedding} via the graph construction. 

 In Section~\ref{section: The Case of Unit Frobenius Modules}, after reviewing the basics of lfgu $F^e$-modules, 
 we specialize to this setting and combine the results of Sections~\ref{section: Theory for Cartier Modules} and~\ref{section: Nearby Cycles via the Graph Embedding} to deduce various consequences. 
 
% Finally, in Section~\ref{section: A Kashiwara--Malgrange Theorem for Etale Fp-sheaves}, we prove a version of the Kashiwara--Malgrange theorem. Since we work with general Noetherian $F$-finite schemes, we cannot directly apply the results of \cite{EK04book}, where the authors restrict to schemes of finite type over a field. A large part of Section~\ref{section: A Kashiwara--Malgrange Theorem for Etale Fp-sheaves} is devoted to recalling the Riemann--Hilbert correspondence in positive characteristic, following~\cite{BL19,BBLSZ25}, and to reproving the necessary results originally established in the finite type case in~\cite{EK04book}. 

\begin{acknowledgments*}
This work was inspired by discussions with Hiroki Kato and Eamon Quinlan-Gallego on the relationship between Bernstein--Sato roots and \'etale $\F_p$-cohomology. The author is grateful to them for many stimulating discussions and helpful comments, and for introducing the author to this fascinating area. 

This work was supported by JSPS KAKENHI Grant Number 25KJ0122. 
\end{acknowledgments*}

\begin{convention*}
Throughout this paper, we fix a prime number $p$. Let $X$ be an $\F_p$-scheme, and let $q$ be a positive power of $p$. We denote by 
\[
F_{q}\colon X\to X
\]
the $q$-th Frobenius endomorphism, defined as the morphism of ringed spaces that induces the identity on the underlying topological space and the $q$-th power map $a\mapsto a^q$ on the structure sheaf. When $q=p$, we simply write $F\coloneqq F_q$. 
When it is necessary to specify the scheme $X$, we write $F_{X,q}$ instead of $F_q$. 

We use analogous notation for rings. If $R$ is a commutative $\F_p$-algebra, we denote by $F_q$ the $q$-th Frobenius endomorphism of $R$. When it is necessary to specify the ring $R$, we write $F_{R,q}$. 

 We say that an $\F_p$-scheme or an $\F_p$-algebra is \emph{$F$-finite} if $F_p$ is finite. 
\end{convention*}

\section{Theory for Cartier Modules}\label{section: Theory for Cartier Modules}
Throughout this section, let $X$ be a Noetherian $F$-finite scheme over $\F_p$. We fix a quasi-coherent ideal sheaf $\mathcal{I}\subset \mathcal{O}_X$ that is locally generated by a single element. 
Fix a positive power $q$ of $p$. 

\subsection{Cartier Modules}
\emph{A $q$-Cartier module on $X$} is a pair $(M,C)$ consisting of an $\cO_X$-module $M$ together with an $\cO_X$-linear morphism  
\[
C\colon F_{q\ast} M\to M. 
\]
Since $F_q$ acts as the identity on the underlying topological space, the morphism $C$ can equivalently be regarded as an additive endomorphism $C\colon M\to M$ satisfying  
\[
C(a^qx)=a\, C(x)
\]
for all local sections $a\in\cO_X$ and $x\in M$. Such an endomorphism is called a \emph{$q$-Cartier operator} on the $\cO_X$-module $M$.

Following \cite{Bli13}, we say that a   $q$-Cartier module $M$ is  \emph{coherent} if $M$ is coherent as an $\cO_X$-module. We say that a $q$-Cartier module $(M,C)$ is \emph{nilpotent} if $C^m(M)=0$ for some $m\geq1$, and \emph{$F$-pure} if it is coherent and satisfies $M=C(M)$. 

The following result is a key technical ingredient for what follows. 
\begin{prop}[{\cite[Proposition 2.14]{BB11}}]\label{F-pure submodule}
Let $M$ be a coherent $q$-Cartier module on $X$. Then the descending sequence 
\[
M\supset C(M)\supset C^2(M)\supset\cdots
\]
stabilizes. In other words, there exists an integer $m>0$ such that $C^m(M)=C^{m+1}(M)$. 
\end{prop}
\begin{defn}
    Let $M$ be a coherent $q$-Cartier module on $X$. We define the $q$-Cartier submodule $\underline{M}$ by  
\begin{equation}\label{Mbar}
\underline{M}=\bigcap_{m\geq1}C^m(M).
\end{equation} 
\end{defn}

By Proposition~\ref{F-pure submodule}, $\underline{M}$ coincides with $C^m(M)$ for all sufficiently large $m$. The $\cO_X$-module 
$\underline{M}$ is therefore coherent, since it is the image of the $\cO_X$-linear morphism $C^m\colon (F_q^m)_\ast M\to M$ between coherent $\cO_X$-modules. 

\begin{lem}\label{lem: underline M}
    Let $M$ be a coherent $q$-Cartier module on $X$. The $\cO_X$-submodule $\underline{M}$ satisfies the following properties: 
\begin{enumerate}
    \item $\underline{M}=C(\underline{M})$. 
    \item $C^m(M)\subset\underline{M}$ for all sufficiently large $m$. 
\end{enumerate}
Moreover,  $\underline{M}$ is uniquely characterized by these properties. 
\end{lem}
In other words, $\underline{M}$ is the unique $F$-pure submodule such that the quotient $M/\underline{M}$ is nilpotent. 
\begin{proof}
The submodule $\underline{M}$ satisfies the stated properties since $\underline{M}=C^m(M)$ 
for all sufficiently large $m$. 

Let $M'$ be an $\cO_X$-submodule of $M$ satisfying~(1) and~(2). By~(2), we have $C^m(M)\subset M'$ for all sufficiently large $m$, and hence $\underline{M}\subset M'$. 
 On the other hand, by~(1), we have $M'=C^m(M')$ for all $m\ge0$, which shows $M'\subset \underline{M}$. Therefore, $M'=\underline{M}$. 
\end{proof}

\subsection{Review of Gauges}
In this subsection, we briefly review  Anderson's theory of gauges. We follow \cite[Section~4]{Bli13}, where Blickle generalizes the theory to the $F$-finite field case.  

Let $k$ be an $F$-finite field, and let $R$ be a finitely generated $k$-algebra. Fix generators $t_1,\dots,t_n\in R$ as a $k$-algebra. Let $M$ be a finitely generated $R$-module, and fix generators  $x_1,\dots,x_r\in M$ as an $R$-module. Then we have 
\[
M=\sum_{j=1}^r\sum_{i_1,\dots,i_n\in\Z_{\ge0}}k\cdot t_1^{i_1}\cdots t_n^{i_n}\cdot x_j. 
\]

For an integer $d\ge0$, let $M_{\le d}$ denote the $k$-subspace of $M$ generated by the elements $t_1^{i_1}\cdots t_n^{i_n}\cdot x_j$ with $i_1+\cdots+i_n\le d$ and $j=1,\dots,r$. 
For any $x\in M$, let $\delta(x)$ be the integer $d\ge0$ satisfying 
\[
x\in M_{\le d}\setminus M_{\le d-1}. 
\]
Here, we set $M_{\le -1}=\emptyset$. The function $\delta$ is called a \emph{gauge} on $M$. 

We say that an $R$-submodule $N$ of $M$ is \emph{generated in gauge $\le d$} if $N\cap M_{\le d}$ generates $N$ as an $R$-module. Since $N$ is finitely generated and $M=\bigcup_dM_{\le d}$, it follows that $N$ is generated in gauge $\le d$ for some $d$.

Now assume that $M$ is equipped with the structure of a $q$-Cartier module, that is, an additive map 
\[
C\colon M\to M
\]
satisfying $C(a^{q}x)=a\,C(x)$. We recall the following fundamental result from \cite[Section~4]{Bli13}. 
\begin{thm}\label{thm: generated in gauge d}
There exists a constant $K\ge0$ with the following property: for any positive integer $m\ge1$ and for any $R$-submodule $N\subset M$ that is generated in gauge $\le d$, the submodule $C^m(N)$ is generated in gauge 
\[\le \frac{d}{q^m}+\frac{K}{q-1}+1. 
\]
\end{thm}
\begin{proof}
By \cite[Proposition~4.5]{Bli13}, there exists a constant $K\ge0$ such that for all $x\in M$, 
\[
\delta(C(x))\le \frac{\delta(x)}{q}+\frac{K}{q-1}. 
\]
Moreover, by \cite[Corollary~4.6]{Bli13}, the same constant $K$ satisfies, for all $m\ge1$ and all $x\in M$,
\[
\delta(C^m(x))\le \frac{\delta(x)}{q^m}+\frac{K}{q-1}.
\]

Applying \cite[Lemma~4.7]{Bli13}  to $\varphi=C^m$, we obtain the desired statement. 
\end{proof}

\subsection{\texorpdfstring{The $\nu$-invariants}{The nu-invariants}}
Let $\cI$ be an ideal sheaf of $\cO_X$ that is locally generated by a single element. 
In this subsection, we generalize the definition of the $\nu$-invariants with respect to $\cI$ to $\cO_X$-submodules of $q$-Cartier modules. 

Let $(M,C)$ be a $q$-Cartier module on $X$, and let $N\subset M$ be an $\cO_X$-submodule. 

\begin{defn}\label{defn: nu}
For each integer $m\geq1$, we define  
\[
\nu_{\cI}(q^m;N):=\{n\in\Z_{\geq0}\mid C^m(\cI^nN)\neq C^m(\cI^{n+1}N)\}. 
\]
We call $\nu_{\cI}(q^m;N)$ the \emph{set of $\nu$-invariants of level $m$}. 

When $\cI$ is generated by an element $f\in\Gamma(X,\cO_X)$, we also write 
\[
\nu_{f}(q^m;N)\coloneqq\nu_{\cI}(q^m;N). 
\]
\end{defn}

\begin{remark}\label{rem: classical nu invariant}
The relationship between the above definition of $\nu$-invariants and that found in the literature is as follows.  

Let $R$ be a regular $F$-finite $\F_p$-algebra, and let $\mathfrak{a}$ be an ideal of $R$. For a proper ideal $J\subset R$ such that $\mathfrak{a} \subset \sqrt{J}$, the integer $\nu_{\mathfrak{a}}^J(q^m)$ is defined. This definition is given in \cite{MTW05} when $R$ is local, and in \cite{QG21} in the general case.  In \cite[Definition~4.1]{QG21}, the set 
\[\{\nu_{\mathfrak{a}}^J(q^m)\mid \text{$J\subset R$ is as above}\}
\]is called the set of $\nu$-invariants. 

Let $F_{q\ast}R$ denote $R$ viewed as an $R$-algebra via the Frobenius homomorphism $R\xrightarrow{F_q}R$. Note that the $F_{q\ast}R$-module ${\rm Hom}_{R}(F_{q\ast}R,R)$ is invertible (see Lemma~\ref{lem: generator}(1)). 
We assume that it is free of rank one as an $F_{q\ast}R$-module, and fix a generator $C$. Then the pair $(R,C)$ defines a $q$-Cartier module. 

If the ideal $\mathfrak{a}$ is locally principal, then we have the equality 
\[
\{\nu_{\mathfrak{a}}^J(q^m)\mid \text{$J\subset R$ is as above}\}=\nu_{\mathfrak{a}}(q^m;R), 
\]
where the right-hand side is defined by taking $N=R$ as a submodule of the $q$-Cartier module $(R,C)$. This is proved in \cite[Proposition~4.2]{QG21}. 
\end{remark}

\begin{lem}\label{basic_nu}
The following statements hold: 
\begin{enumerate}
    \item The sets $\nu_{\cI}(q^m;N)$ form a descending chain 
    \[
    \nu_{\cI}(q;N)\supset \nu_{\cI}(q^2;N)\supset \nu_{\cI}(q^3;N)\supset\cdots. 
    \]
    \item If $ n\in \nu_{\cI}(q^m;N)$ and $n\ge q^m$, then $n-q^m\in \nu_{\cI}(q^m;N)$. 
\end{enumerate}
\end{lem}
\begin{proof}
(1) This follows immediately from the definition.

    (2) Since the claim can be proved locally, we may assume that $\cI$ is generated by a global section $f\in\Gamma(X,\cO_X)$. 
    The claim then follows from the computation 
    \[
    f\cdot C^m(f^{n-q^m}N)=C^m(f^nN)\neq C^m(f^{n+1}N)=f\cdot C^m(f^{n-q^m+1}N). 
    \]
\end{proof}

We now define the $\nu$-invariants of infinite level. 
\begin{defn}\label{defn: nu infty}
Let $(M,C)$ be a $q$-Cartier module on $X$, and let $N\subset M$ be an $\cO_X$-submodule. 
\begin{enumerate}
    \item A \emph{$\nu$-invariant of $N$ with respect to $\cI$ of infinite level} is a $p$-adic integer $\alpha\in \Z_p$ such that for any integer $m\geq1$, there exists an integer $n\in\nu_{\cI}(q^m;N)$ satisfying 
\[
\alpha-n\in q^m\Z_p. 
\] 
\item We denote by $\nu_\cI(q^\infty;N)$ the set of all $\nu$-invariants of infinite level. 

When $\cI$ is generated by an element $f\in\Gamma(X,\cO_X)$, we also write 
\[\nu_f(q^\infty;N)\coloneqq\nu_\cI(q^\infty;N).
\]
\end{enumerate}
\end{defn}

In other words, $\alpha\in\Z_p$ is a $\nu$-invariant of infinite level if it can be approximated $p$-adically by $\nu$-invariants of finite level. A more precise formulation is given below. 

For a $p$-adic integer $\alpha\in\Z_p$ and an integer $m\geq1$, 
we denote by $\alpha_m$ the unique integer in $[0,q^m)$ satisfying 
\[
\alpha-\alpha_m\in q^m\Z_p. 
\]

\begin{lem}\label{BSRnu}
The following statements hold: 
\begin{enumerate}
    \item The assignment $a\mapsto a_m$ induces a well-defined map 
    \[
    \nu_{\cI}(q^{m+1};N)\cap [0,q^{m+1})\to \nu_{\cI}(q^m;N)\cap [0,q^m). 
    \]
    \item The assignment $\alpha\mapsto (\alpha_m)_m$ induces a bijection 
    \[
   \nu_{\cI}(q^\infty;N)\to \varprojlim_{m\geq1}\Big(\nu_{\cI}(q^m;N)\cap [0,q^m)\Big), 
    \]
    where the transition maps in the projective limit are those defined in~(1).  
\end{enumerate}
\end{lem}
\begin{proof}
    (1) Let $a\in \nu_{\cI}(q^{m+1};N)$. By Lemma~\ref{basic_nu}(1), we have $a\in\nu_{\cI}(q^{m};N)$. The claim then follows from Lemma~\ref{basic_nu}(2), since $a-a_m$ is a multiple of $q^m$.

    (2) Let $\alpha\in\nu_{\cI}(q^\infty;N)$. By definition, for every integer $m\geq1$, there exists $n\in\nu_{\cI}(q^m;N)$ such that $\alpha-n\in q^m\Z_p$. Applying Lemma \ref{basic_nu}(2) repeatedly, we obtain 
    \[\alpha_m\in \nu_{\cI}(q^m;N)\cap [0,q^m).\]
    Hence the map $\alpha\mapsto (\alpha_m)_m$ is well-defined. 

    Injectivity follows from the uniqueness of the $p$-adic expansion. 
    For surjectivity, let $(n_m)_m$ be an element of  
\[\varprojlim_{m\geq1}\Big(\nu_{\cI}(q^m;N)\cap [0,q^m)\Big). \]
By the definition of the transition maps, the sequence $(n_m)_{m\geq1}$ converges $p$-adically to some  element $\alpha\in\Z_p$. 
Since $\alpha-n_m\in q^m\Z_p$ for all $m\ge1$, we have $\alpha\in\nu_{\cI}(q^\infty;N)$, and its image under the above map is $(n_m)_m$. 
\end{proof}

We show that $\nu$-invariants are local invariants. 
\begin{lem}\label{lem: BSR is local}
    Let $X=\bigcup_{i\in I}U_i$ be an open covering. Then 
    \begin{align}
        &\label{nu=cup}\nu_{\cI}(q^m;N)=\bigcup_{i\in I}\nu_{\cI|_{U_i}}(q^m;N|_{U_i})\qquad(m\ge1),\\
        &\label{BSR=cup}\nu_{\cI}(q^\infty;N)=\bigcup_{i\in I}\nu_{\cI|_{U_i}}(q^\infty;N|_{U_i}). 
    \end{align}
\end{lem}
\begin{proof}
    Let $n\ge0$ be an integer. Since $\{U_i\}_i$ is an open covering, the condition 
    \[C^m(\cI^nN)\neq C^m(\cI^{n+1}N)\]
    holds if and only if there exists some $i\in I$ such that 
    \[
    C^m(\cI^nN)|_{U_i}\neq C^m(\cI^{n+1}N)|_{U_i}. 
    \]
    This proves \eqref{nu=cup}. 

    We now prove \eqref{BSR=cup}. By \eqref{nu=cup} and Lemma~\ref{BSRnu}(2), we obtain   
    \[
    \bigcup_{i\in I}\nu_{\cI|_{U_i}}(q^\infty;N|_{U_i})\subset \nu_{\cI}(q^\infty;N). 
    \]
    For the reverse inclusion, since $X$ is quasi-compact, we may choose a finite subset $J\subset I$ such that $\{U_i\}_{i\in J}$ covers $X$. Replacing $I$ by $J$, we may assume that 
    $I$ is finite. 
    
    Let $\alpha\in\nu_{\cI}(q^\infty;N)$. Then for every $m\ge1$, we have 
    \[\alpha_m\in \nu_{\cI}(q^m;N).\]
    By \eqref{nu=cup}, there exists  $i(m)\in I$ such that 
    \[
    \alpha_m\in \nu_{\cI|_{U_{i(m)}}}(q^m;N|_{U_{i(m)}}). 
    \]
    Since $I$ is finite, there exist an index $i\in I$ and an increasing sequence of integers 
    \[
1\le m_1<m_2<\cdots,\qquad\lim_{n\to\infty}m_n=\infty,
    \]
    such that $i=i(m_n)$ for all $n$. By Lemma~\ref{BSRnu}(2), it follows that $\alpha\in \nu_{\cI|_{U_i}}(q^\infty;N|_{U_i})$. This proves~\eqref{BSR=cup}. 
\end{proof}

\medskip

From now on, let  $(M,C)$ be a coherent $q$-Cartier module on $X$, and let $N\subset M$ be a coherent $\cO_X$-submodule. The following assumption on $\big((M,C),N\big)$ will play a crucial role in what follows. 
\begin{assumption}\label{assumption: nu invariant is bound}
    There exists a constant $L\geq0$ such that 
    \[
    \#\Big(\nu_{\cI}(q^m;N)\cap [0,q^m)\Big)\leq L
    \]
    for all $m\geq1$. 
\end{assumption}

The theory of gauge provides a situation in which this assumption is automatically satisfied. 
\begin{prop}\label{consBl}
Suppose that there exists an open covering  
    \[
    X=\bigcup_iU_i, 
    \]
    where each $U_i$ is of finite type over an $F$-finite field. Then 
Assumption~\ref{assumption: nu invariant is bound} holds. 
\end{prop}
\begin{proof}
 By Lemma~\ref{lem: BSR is local}, the statement is local. Hence 
 it suffices to treat the case where $X={\rm Spec}(R)$, for a finitely generated algebra $R$ over  an $F$-finite field $k$. We may further assume that $\cI$ is generated by an element $f\in R$.

By abuse of notation, we write $M$ and $N$ for the $R$-modules $\Gamma(X,M)$ and $\Gamma(X,N)$, respectively. We also write  $C$ for the induced $q$-Cartier operator on $M$. 

Set $t_1=f$, and choose elements $t_2,\dots,t_n\in R$ such that $t_1,\dots,t_n$ generate $R$ as a $k$-algebra. Choose generators of $M$ as an $R$-module, and let $\delta$ denote the induced gauge on $M$. Let $K$ be a constant satisfying the condition in Theorem~\ref{thm: generated in gauge d}. 

Let $d\ge0$ be an integer such that $N$ is generated in gauge $\le d$. For each $a=0,\dots, q^m-1$, 
the module $f^aN$ is generated in gauge 
\[\leq a+d\le q^m-1+d. \]
Therefore, by Theorem~\ref{thm: generated in gauge d}, $C^m(f^aN)$ is generated in gauge 
\[
\leq 1+\frac{d-1}{q^m}+\frac{K}{q-1}+1\leq 2+d+\frac{K}{q-1} \eqqcolon K_1. 
\]
Consequently, the $R$-module $C^m(f^aN)$ is generated by 
\[
C^m(f^aN) \cap M_{\leq K_1}. 
\]
Put 
\[L\coloneqq\dim_kM_{\le K_1}+1. 
\]
Since $M_{\le K_1}$ is finite-dimensional over $k$, this is a finite integer. Then the descending chain 
\[
C^m(N) \cap M_{\leq K_1}\supset C^m(fN) \cap M_{\leq K_1}\supset\cdots
\]
of $k$-subspaces of $M_{\le K_1} $ has length at most $ L$. 
Hence the cardinality of the set 
\[\{C^m(f^aN)\mid a=0,\dots,q^m-1\}
\]
is bounded by $L$. Thus Assumption~\ref{assumption: nu invariant is bound} is satisfied. 
\end{proof}

 The following is an immediate consequence of Assumption~\ref{assumption: nu invariant is bound}. 
\begin{prop}
 \label{BSRfin}
 If Assumption~\ref{assumption: nu invariant is bound} holds, then the set $\nu_{\cI}(q^\infty;N)$ is finite. 
\end{prop}
\begin{proof}
The claim follows immediately from Lemma \ref{BSRnu}(2). 
\end{proof}

\smallskip

Our next goal is to show that $\nu_{\cI}(q^\infty;N)$ is in fact contained in $\Q$ under an additional condition on $N$. Recall that a $p$-adic integer is rational if and only if its $p$-adic expansion is eventually periodic. The following condition on $N$ arises naturally in the proof. 
\begin{defn}
A coherent $\cO_X$-submodule $N$ of $M$ is said to be \emph{$C$-preperiodic} if there exist integers $i,j>0$ such that 
\[
C^i(N)=C^{i+j}(N). 
\]
\end{defn}

We give criterions for $N$ to be $C$-preperiodic. 

\begin{lem}\label{lem: when N isperiodic}
    Let $N$ be a coherent $\cO_X$-submodule of $M$. Then the following hold: 
    \begin{enumerate}
        \item If $X$ is of finite type over $\F_p$, then $N$ is $C$-preperiodic. 
        \item If $N$ satisfies 
        \[
        C(N)\subset N\quad\text{or}\quad N\subset C(N), 
        \]
        then $N$ is $C$-preperiodic. In particular, any coherent $q$-Cartier submodule is $C$-preperiodic. 
    \end{enumerate}
\end{lem}
\begin{proof}
(1) 
Since the assertion is local on $X$, we may assume that $X={\rm Spec}(R)$, where $R$ is finitely generated over $\F_p$. Choose a gauge $\delta$ on $M$, and let $d\ge0$ be  such that $N$ is generated in gauge $\le d$. By Theorem~\ref{thm: generated in gauge d}, for each $i\ge 1$, the module $C^i(N)$ is generated in gauge 
\[
\le \frac{d}{q^i}+\frac{K}{q-1}+1\le  d+1+\frac{K}{q-1}\eqqcolon K_2. 
\]
Since $M_{\le K_2}$ is finite-dimensional over $\F_p$, it is a finite set. Hence there are only finitely many 
$R$-submodules of $M$ generated by elements of $M_{\le K_2}$. Therefore, the set 
\[
\{C^i(N)\mid i\ge1\}
\]
is finite. The claim follows. 

\medskip

\noindent
(2) The case $C(N)\subset N$ follows from Proposition~\ref{F-pure submodule}. 
Assume that $N\subset C(N)$. Then we have an ascending chain 
\[
N\subset C(N)\subset C^2(N)\subset \cdots
\]
of coherent submodules of $M$. Since $M$ is Noetherian, the chain stabilizes, and we have $C^i(N)=C^{i+1}(N)$ for some $i\ge1$. 
\end{proof}

Let $\alpha$ be a $p$-adic integer. There exists a unique integer $i\in[0,q)$ such that $\alpha-i$ is divisible by $q$ in $\Z_p$. We define  
\[
T(\alpha)\coloneqq\frac{\alpha-i}{q}\in \Z_p. 
\]
\begin{lem}\label{Talpha}
The following statements hold: 
\begin{enumerate}
    \item For any  $\alpha\in\nu_\cI(q^\infty;N)$, we have $T(\alpha)\in\nu_\cI(q^\infty;C(N))$. 
    \item The assignment $\alpha\mapsto T(\alpha)$ defines a surjection  
    \[
    \nu_\cI(q^\infty;N)\to \nu_\cI(q^\infty;C(N)). 
    \]
\end{enumerate}
\end{lem}
\begin{proof}
By Lemma~\ref{lem: BSR is local}, we may assume that $\cI$ is generated by a single element $f\in \Gamma(X,\cO_X)$. 

\medskip

\noindent
(1) Let $\alpha\in \nu_f(q^\infty;N)$. Then for any $m\geq1$, we have 
\[
C^{m+1}(f^{\alpha_{m+1}}N)\supsetneq C^{m+1}(f^{\alpha_{m+1}+1}N). 
\]
Write $\alpha=qT(\alpha)+i$ with $i\in[0,q)$. Then for any $m\ge1$, we have
\[
\alpha_{m+1}
=
q\,T(\alpha)_m+i.
\]
Consequently, we have 
\[
q\, T(\alpha)_m\le \alpha_{m+1},\qquad \alpha_{m+1}+1\le q\, (T(\alpha)_m+1). 
\]
Hence,  
\[
C^m(f^{T(\alpha)_m}C(N))\supset C^{m+1}(f^{\alpha_{m+1}}N)\supsetneq C^{m+1}(f^{\alpha_{m+1}+1}N)\supset C^m(f^{T(\alpha)_m+1}C(N)). 
\]
This shows that $T(\alpha)\in \nu_f(q^\infty;C(N))$. 

\medskip

\noindent
(2) Let $\beta\in \nu_f(q^\infty;C(N))$. Then for any $m\geq1$, 
\[
C^{m+1}(f^{q\beta_m}N)=C^m(f^{\beta_m}C(N))\supsetneq C^m(f^{\beta_m+1}C(N))=C^{m+1}(f^{q\beta_m+q}N). 
\]
    Thus for each $m$, there exists an integer $i_m\in[0,q)$ such that
    \[q\beta_m+i_m\in \nu_f(q^m;N).\]
    Since $i_m$ takes only finitely many values, we may choose an increasing 
  sequence $m_1<m_2<\cdots$ of integers such that $i_{m_n}=i$ is constant. Then the sequence $(q\beta_{m_n}+i)_{n\ge1}$ converges $p$-adically to 
    \[\alpha\coloneqq q\beta+i.\]
    By Lemma~\ref{BSRnu}(2), we obtain 
$\alpha\in \nu_f(q^\infty;N)$. 
    By construction, $T(\alpha)=\beta$, 
which proves surjectivity. 
\end{proof}
\begin{lem}\label{lem: nu invariants for Fpure}
Let $(M,C)$ be a coherent $q$-Cartier module on $X$, and let $N\subset M$ be a coherent $\cO_X$-submodule. Suppose that there exists an integer $j\ge1$ such that $N=C^j(N)$.  Then, for every $\alpha\in\nu_{\cI}(q^\infty;N)$ and every integer $n\ge1$, there exists an integer $m_n$ with $0\le m_n\le q^{nj}-1$ such that 
\[
q^{nj}\alpha+m_n\in\nu_{\cI}(q^\infty;N). 
\]
\end{lem}
\begin{proof}
    By the assumption $N=C^j(N)$ and Lemma~\ref{Talpha}(2), $T^{nj}$ induces a surjection 
\[
T^{nj}\colon\nu_{\cI}(q^\infty;N)\to\nu_{\cI}(q^\infty;N). 
\] 
Therefore, for every $\alpha\in \nu_{\cI}(q^\infty;N)$, 
there exists an element $\beta\in\nu_{\cI}(q^\infty;N)$ such that $T^{nj}(\beta)=\alpha$. By the definition of $T$, this means that $\beta$ can be written 
as 
\[
\beta=q^{nj}\alpha+m_n
\]
for some integer $m_n\in[0,q^{nj})$. The claim follows. 
\end{proof}

\begin{thm}\label{BSRrat}
 Let $(M,C)$ be a coherent $q$-Cartier module on $X$, and let  $N\subset M$ be a $C$-preperiodic coherent $\cO_X$-submodule. Assume that Assumption~\ref{assumption: nu invariant is bound} holds. Then the following statements hold:  
  \begin{enumerate}
    \item The set $\nu_\cI(q^\infty;N)$ is contained in 
    \[\Z_{(p)}=\Z_p\cap \Q. \]

\item Assume further that there exists an integer $j>0$ such that $N=C^j(N)$. Then 
\[\nu_\cI(q^\infty;N)\subset \Z_{(p)}\cap [-1,0]. 
\]
\end{enumerate}  
\end{thm}
\begin{proof}
Since the statement is local on $X$, we may assume that $\cI$ is generated by a single element $f\in\Gamma(X,\cO_X)$. 

\medskip

\noindent
(1) Note that $\alpha\in\Z_p$ is rational if and only if $T(\alpha)$ is rational. By Lemma \ref{Talpha}(1), it therefore suffices to show that $\nu_f(q^\infty;C^i(N))$ consists of rational numbers for some $i\geq1$. 

Since $N$ is $C$-preperiodic, there exist integers $i,j>0$ such that 
\[C^i(N)=C^{i+j}(N).\]
Let $\alpha\in\nu_f(q^\infty;C^i(N))$. By Lemma~\ref{Talpha}(1), we have 
 \[
 T^j( \nu_f(q^\infty;C^i(N)))\subset \nu_f(q^\infty;C^i(N)). 
 \]
Thus the sequence
\[
\alpha,\ T^j(\alpha),\ T^{2j}(\alpha),\ \dots
\]
takes values in the set $\nu_f(q^\infty;C^i(N))$, which is finite by Proposition~\ref{BSRfin}. 
Hence, there exist integers $n>m\geq0$ such that 
\[T^{nj}(\alpha)=T^{mj}(\alpha). \]
This implies that $\alpha\in\Q$. 

\medskip

\noindent
(2)  Let $\alpha\in\nu_{\cI}(q^\infty;N)$, and we show that $\alpha\in [-1,0]$. By Lemma~\ref{lem: nu invariants for Fpure}, for every integer $n\ge1$, there exists an integer $m_n\in[0,q^{nj})$ such that 
\[
q^{nj}\alpha+m_n\in \nu_{\cI}(q^\infty;N). 
\]
By Proposition~\ref{BSRfin}, the set $\nu_{\cI}(q^\infty;N)$ is finite. Therefore, there exists a strictly increasing sequence of positive integers 
\[
0<n_1<n_2<\cdots
\]
such that 
\[
q^{n_1j}\alpha+m_{n_1}=q^{n_ij}\alpha+m_{n_i} \qquad\text{for all $i\ge2$}. 
\]
Rewriting this equality, we obtain 
\[
\alpha=\frac{m_{n_1}-m_{n_j}}{q^{n_ij}-q^{n_1j}}. 
\]
Since $0\le m_{n_1}\le q^{n_1j}-1$ and $0\le m_{n_i}\le q^{n_ij}-1$, we have 
\[
-\frac{q^{n_ij}-1}{q^{n_ij}-q^{n_1j}}\le \alpha\le \frac{q^{n_1j}-1}{q^{n_ij}-q^{n_1j}}. 
\]
Letting $i\to\infty$, we obtain $\alpha\in[-1,0]$, as desired. 
\end{proof}

\subsection{Test Module Filtration}

Let $(M,C)$ be an $F$-pure $q$-Cartier module on $X$. 
    
\begin{defn}
Let $t\ge0$ be a real number. We define the \emph{test module $\tau(M,\cI^t)$ of $M$ with respect to $\cI$ at level $t$} by  
\[
\tau(M,\cI^t)\coloneqq\bigcup_{m\geq0}C^m(\cI^{\lceil tq^m\rceil}\cdot M). 
\]
Here, $\lceil u\rceil$ denotes the smallest integer satisfying $u\le \lceil u\rceil$.  

When $\cI=f\cO_X$, we also write 
\[
\tau(M,f^t)\coloneqq \tau(M,\cI^t). 
\]
\end{defn}
\begin{rem}
    Suppose that we are in the situation of Remark~\ref{rem: classical nu invariant}. Fix a generator $C$ as in loc.~cit. Then the test module $\tau(R,\cI^t)$ associated to the $q$-Cartier module $(R,C)$ agrees with the test ideal $\tau(\cI^t)$. This follows from \cite[Definition~2.9 and Proposition~2.22]{BMS08}. 
\end{rem}

At least when $t\notin\Z[\frac{1}{p}]$, the test module $\tau(M,\cI^t)$ admits an intrinsic description. See Proposition~\ref{prop: intrinsic description of test module} for details. 

\medskip

We record some basic properties of test modules. 

\begin{lem}\label{lem: test module is Ce for large e}
Let $M$ be an $F$-pure $q$-Cartier module, and let $t\in\R_{\geq0}$. Then for every integer $m\geq0$, 
\[C^m(\cI^{\lceil tq^m\rceil}M)\subset C^{m+1}(\cI^{\lceil tq^{m+1}\rceil}M). 
\]
Consequently, we have 
\[\tau(M,\cI^t)=C^m(\cI^{\lceil tq^m\rceil}M)
\]
for all sufficiently large $m$, and the test module is a coherent $\cO_X$-submodule of $M$. 
\end{lem}
\begin{proof}
Since the statement is local, we may assume $\cI=f\cO_X$. 

As $M$ is $F$-pure, we have $M=C(M)$. Hence,  
\[
C^m(f^{\lceil tq^m\rceil}M)=C^m(f^{\lceil tq^m\rceil}C(M))=C^{m+1}(f^{q \lceil tq^m\rceil}M). 
\]
Since $q\lceil tq^m\rceil \ge \lceil tq^{m+1}\rceil$, the desired inclusion follows.

The second assertion follows because $M$ is Noetherian and the above inclusions form an ascending chain of coherent $\cO_X$-submodules. 
\end{proof}
\begin{lem}\label{basic_t}
Let $M$ be an $F$-pure $q$-Cartier module. Then the following statements hold: 
\begin{enumerate}

    \item For $t\geq s\geq0$, we have 
    \[
    \tau(M,\cI^t)\subset \tau(M,\cI^s). 
    \]
\item   For $t\geq0$, we have 
\[\tau(M,\cI^{t+1})=\cI\cdot \tau(M,\cI^t).\]

\item For $t\geq0$, we have 
\[C\big(\tau(M,\cI^{qt})\big)=\tau(M,\cI^t).\]

\item Let $n$ and $m$ be positive integers, and let $t=n/q^m$. Then 
\[
\tau(M,\cI^t)=C^m(\cI^nM). 
\]
\end{enumerate}
\end{lem}

\begin{proof}
Since all the statements are local, we may assume $\cI=f\cO_X$.

(1) This is immediate from the definition of test modules. 

(2) For every $m\geq0$, we have 
\[C^m(f^{\lceil (t+1)q^m\rceil}M)=C^m(f^{q^m}f^{\lceil tq^m\rceil}M)=fC^m(f^{\lceil tq^m\rceil}M).
\]
Letting $m\to\infty$, the claim follows.

(3) 
Choose $m\gg0$ such that 
\[
\tau(M,f^{t})=C^{m+1}(f^{\lceil tq^{m+1}\rceil}M)\quad \text{and} \quad\tau(M,f^{qt})=C^{m}(f^{\lceil tq^{m+1}\rceil}M). 
\]
Applying $C$ to the second equality yields the desired result.

(4)  Let $m'\geq0$ be an integer such that 
\[
\tau(M,f^t)=C^{m+m'}(f^{\lceil tq^{m+m'}\rceil}M)=C^{m+m'}(f^{nq^{m'}}M). 
\]
Note that 
\[
C^{m+m'}(f^{nq^{m'}}M)=C^m(f^nC^{m'}(M)). 
\]
Since $M$ is $F$-pure, we have $C^{m'}(M)=M$. The claim follows. 
\end{proof}

\begin{defn}
Let $M$ be an $F$-pure $q$-Cartier module on $X$. 
\begin{enumerate}
    \item A \emph{right $F$-jumping exponent of $M$ with respect to $\cI$} is a real number $t\in\R_{\ge0}$ for which the inclusion 
\[
\tau(M,\cI^{t+\epsilon})\subset \tau(M,\cI^{t})
\]
is strict for all $\epsilon>0$. 

\item A \emph{left $F$-jumping exponent of $M$ with respect to $\cI$} is a real number $t\in\R_{>0}$ for which the inclusion 
\[
\tau(M,\cI^{t})\subset \tau(M,\cI^{t-\epsilon})
\]
is strict for all $0<\epsilon<t$. 
\end{enumerate}
We denote by 
\[\FJN^+(M,\cI)\quad(\text{resp. }\FJN^-(M,\cI))
\]
the set of all right (resp. left) $F$-jumping exponents. 
We also set 
\[\FJN(M,\cI):=\FJN^+(M,\cI)\cup \FJN^-(M,\cI). \]

When $\cI=f\cO_X$, we also write 
\[\FJN(M,f)\coloneqq \FJN(M,\cI). \]
Similar notation applies to $\FJN^+$ and $\FJN^-$. 
\end{defn}

\begin{lem}\label{FJNtrick}
    Let $M$ be an $F$-pure $q$-Cartier module, and fix $\star\in\{+,-,\emptyset\}$. 
    Then the following statements hold: 
    \begin{enumerate}
        \item If $t>1$ belongs to $\FJN^\star(M,\cI)$, then $t-1\in \FJN^\star(M,\cI)$. 
        \item If $t\in\FJN^\star(M,\cI)$, then $qt\in \FJN^\star(M,\cI)$. 
    \end{enumerate}
    Here, we interpret $\FJN^\emptyset(M,\cI)$ as $\FJN(M,\cI)$. 
\end{lem}
\begin{proof}
Both statements follow directly from Lemma~\ref{basic_t}(2) and~(3).
\end{proof}

In contrast to left $F$-jumping exponents, right $F$-jumping exponents take a rather special form, as the following lemma shows. 
\begin{lem}\label{+1/p}
The set $\FJN^+(M,\cI)$ is contained in $\Z\big[\frac{1}{p}\big]$. 
\end{lem}
\begin{proof}
Let $t\in\R_{\ge0}$ be such that $t\notin \Z\big[\frac{1}{p}\big]$. We show that $t\notin \FJN^+(M,\cI)$. 

    Choose an integer $m\gg0$ such that 
    \[
    \tau(M,\cI^t)=C^m(f^{\lceil tq^m\rceil}M). 
    \]
    Since $tq^m$ is not an integer by assumption, there exists $\epsilon>0$ such that 
    \[\lceil tq^m\rceil=\lceil (t+\epsilon)q^m\rceil.\]
    Hence, 
    \[
    \tau(M,\cI^t)=C^m(\cI^{\lceil (t+\epsilon)q^m\rceil}M)\subset \tau(M,\cI^{t+\epsilon}). 
    \]
    Thus the inclusion  $\tau(M,\cI^{t+\epsilon})\subset \tau(M,\cI^t)$ is in fact an equality, which shows $t\notin \FJN^+(M,\cI)$. 
\end{proof}

\begin{rem}
    One may interpret $\FJN^+(M,\cI)$ as capturing some information about the extent to which $M$ admits a nontrivial morphism to an $F$-pure $q$-Cartier module supported on $Z(\cI)$. As evidence for this perspective, see Proposition~\ref{right-conti}. 
\end{rem}

Under Assumption~\ref{assumption: nu invariant is bound}, the discreteness and rationality of Frobenius jumping exponents extend to our setting. 

\begin{thm}\label{thm: discreteness and rationality of FJN}
Let $M$ be an $F$-pure $q$-Cartier module on $X$. Assume that Assumption~\ref{assumption: nu invariant is bound} holds with $N=M$. 
Then the set $\FJN(M,\cI)\cup\{0\}$ is discrete and contained in $\Q$. 
\end{thm}

\begin{proof}
We first prove discreteness. More precisely, we show that $\FJN(M,\cI)\cap (0,1]$ is finite;  the discreteness then follows  from Lemma \ref{FJNtrick}(1).

By assumption, there exists a positive integer $N$ such that 
\begin{equation}
    \label{nuN}
    \#\Big(\nu_\cI(q^m;M)\cap[0,q^m)\Big)\leq N-1
\end{equation}
for all $m\geq1$. We claim that 
\[\#\Big(\FJN(M,\cI)\cap (0,1]\Big)\leq N.\]
Suppose not, and choose elements 
\[
0<t_1<\cdots<t_{N+1}\le1
\]
in $\FJN(M,\cI)$. 
Choose a positive integer $m$ such that 
\[\frac{2}{q^m}<t_{i+1}-t_i \qquad \text{for all $i=1,\dots,N$.} 
\]
For each $i$,  let $n_i$ be the unique integer in $[0,q^{m})$ such that  
\[
\frac{n_i}{q^{m}}< t_i\leq\frac{n_i+1}{q^{m}}. 
\]
Since $t_i\in\FJN(M,\cI)$ and $\frac{n_i}{q^{m}}< t_i<\frac{n_i+2}{q^{m}}$, we have 
\[
C^{m}(\cI^{n_i}M)=\tau(M,\cI^{\frac{n_i}{q^{m}}})\supsetneq \tau(M,\cI^{\frac{n_i+2}{q^{m}}})=C^{m}(\cI^{n_i+2}M), 
\]
where the first and last equalities follow from Lemma \ref{basic_t}(4). 
Consequently, either $n_i$ or $n_i+1$ belongs to $\nu_\cI(q^{m};M)$. 

By our choice of $m$, the intervals 
\[
(\frac{n_i}{q^{m}},\frac{n_i+2}{q^{m}})
\]
are disjoint. Thus we obtain $N$ distinct elements of $\nu_\cI(q^{m};M)\cap [0,q^{m})$, contradicting the bound \eqref{nuN}. This proves discreteness. 

\medskip

Finally, we prove rationality. Let $\alpha\in \FJN(M,\cI)$. We show that $\alpha\in\Q$. It suffices to consider the case where $\alpha\notin\Z\bigl[\frac{1}{p}\bigr]$. By Lemma \ref{FJNtrick}(2), for each $m\ge0$, we have $q^m\alpha\in \FJN(M,\cI)$. By Lemma~\ref{FJNtrick}(1), the fractional part of $q^m\alpha$ belongs to the finite set $\FJN(M,\cI)\cap [0,1)$. Therefore, there exist integers $m'>m\ge0$ such that 
\[
(q^{m'}-q^m)\alpha=q^{m'}\alpha - q^m\alpha \in\Z, 
\]
which shows $\alpha\in\Q$. 
\end{proof}

\subsubsection{Relationship with the $\nu$-invariants}
In this subsection, we study the relationship between the $\nu$-invariants and Frobenius jumping exponents. Fix an $F$-pure $q$-Cartier module $M$ on $X$. 
For simplicity, we write $\FJN=\FJN(M,\cI)$, $\FJN^-=\FJN^-(M,\cI)$, and $\FJN^+=\FJN^+(M,\cI)$. 

\begin{lem}\label{lem: nu invariants produce FJN}
Let $m\in\Z_{\ge1}$ and let $n\in\nu_{\cI}(q^m;M)$. The following statements hold: 
\begin{enumerate}

\item There exists a sequence $\{t_i\}_{i\ge1}$ in $\Q\cap [\frac{n}{q^m},\frac{n+1}{q^m})$ satisfying the following properties: 
\begin{enumerate}
    \item $t_iq^{m+i}\in\nu_{\cI}(q^{m+i};M)$ for all $i\ge1$,  
    \item $t_i\le t_{i+1}<t_i+\frac{1}{q^{m+i}}$ for all $i\ge1$. 
\end{enumerate}
    \item Let $\{t_i\}_{i\ge1}$ be a sequence as in~(1). Then $\{t_i\}_{i\ge1}$ converges to some $t\in \bigl[\frac{n}{q^m},\frac{n+1}{q^m}\bigr]$. Moreover,  
    \[
    t\in\begin{cases}
        \FJN^-&\text{if }t_i<t\text{ for all }i,\\
         \FJN^+&\text{if }t_i=t\text{ for some }i. 
    \end{cases}
    \]
\end{enumerate}
\end{lem}
\begin{proof}
(1) Since $n\in\nu_{\cI}(q^m;M)$, we have $C^m(\cI^nM)\supset C^m(\cI^{n+1}M)$. Since $M=C(M)$, it follows that 
\[
C^{m+1}(\cI^{nq}M)=C^m(\cI^nC(M))\supsetneq C^m(\cI^{n+1}C(M))=C^{m+1}(\cI^{(n+1)q}M). 
\]
Hence, there exists an integer $r_1\in\{0,\dots,q-1\}$ such that $nq+r_1\in\nu_{\cI}(q^{m+1};M)$. Proceeding inductively, we construct a sequence $\{r_i\}_{i\ge1}$ with $r_i\in\{0,\dots,q-1\}$ such that, for each $i\ge1$,  
\[
nq^i+\sum_{j=1}^{i}q^{i-j}r_j\in\nu_{\cI}(q^{m+i};M). 
\]
Then the sequence $\{t_i\}_{i\ge1}$ defined by 
    \[
    t_i\coloneqq  \frac{nq^i+\sum_{j=1}^{i}q^{i-j}r_j}{q^{m+i}}\]
satisfies the required property. 

\medskip 

\noindent
(2) Since the sequence $\{t_i\}_{i\ge1}$ is increasing and satisfies $ t_i\le\frac{n+1}{q^m}$ for all $i$, it converges to some $t\in\R$. It is clear that  $t\in\bigl[\frac{n}{q^m},\frac{n+1}{q^m}\bigr]$. By assumption~$(b)$, we also have 
\begin{equation*}\label{equation: ti t ti+}
  t_i\le t\le t_i+\frac{1}{q^{m+i}} \qquad\text{for all }i.   
\end{equation*}
Hence, 
\[
\tau(M,\cI^{t_i})\supset \tau(M,\cI^t)\supset \tau(M,\cI^{t_i+\frac{1}{q^{m+i}}}). 
\]
Note that, by Lemma~\ref{basic_t}(4), together with assumption~$(a)$, the inclusion $\tau(M,\cI^{t_i})\supset \tau(M,\cI^{t_i+\frac{1}{q^{m+i}}})$ is strict. Since this holds for all $i$, it follows that $t\in\FJN$. 

\medskip 

First assume that $t\notin\Z[\frac{1}{p}]$. In this case, we have $t_i<t$ for all $i$. On the other hand, by Lemma~\ref{+1/p}, we have $t\notin\FJN^+$, and hence $t\in\FJN^-$. This proves the assertion in this case. 

Assume now that $tq^{m+i_0}\in\Z$ for some $i_0\ge1$. If $t_i<t$ for all $i$, then $t=t_i+\frac{1}{q^{m+i}}$ for all $i\ge i_0$. This shows that $\tau(M,\cI^{t_i})\supsetneq \tau(M,\cI^t)$, and hence $t\in\FJN^-$. On the other hand, if $t=t_i$ for some $i$ (and hence for all sufficiently large $i$), then $\tau(M,\cI^{t})\supsetneq \tau(M,\cI^{t+\frac{1}{q^{m+i}}})$. This shows $t\in\FJN^+$. 

This completes the proof. 
\end{proof}

For $t\in\R_{\ge0}$, let $\psi^-(t)$ and $\psi^+(t)$ denote the unique integers satisfying 
\[
\psi^-(t)<t\le \psi^-(t)+1, \qquad \psi^+(t)\le t< \psi^+(t)+1. 
\]
Explicitly, $\psi^-(t)=\lceil t\rceil-1$ and $\psi^+(t)=\lfloor t\rfloor$. 

The following theorem can be compared with \cite[Thoerem~6.7]{Mus09}, in view of Theorem~\ref{thm: nu invariants vs BSR Dmodule}. 
\begin{thm}\label{thm: FJN and nu-invariant}
Let $M$ be an $F$-pure $q$-Cartier module on $X$, and fix an integer $m\ge1$. Then the following hold: 
\begin{enumerate}
\item The map 
\[
\psi^-_m\colon \FJN^-\to \nu_{\cI}(q^m;M),\quad t\mapsto \psi^-(tq^m)
\]
 is well-defined. 
    \item The map 
\[
\psi^+_m\colon \FJN^+\to \nu_{\cI}(q^m;M),\quad t\mapsto \psi^+(tq^m)
\]
is well-defined.  
 \item We have 
 \[
\operatorname{Im}\psi^-_m\cup \operatorname{Im}\psi^+_m=\nu_{\cI}(q^m;M). 
 \]
\end{enumerate}
\end{thm}
\begin{proof}
We first prove~(1) and~(2). Let $t\in\R_{\ge0}$, and let $n\in\Z_{\ge0}$ be such that 
\[
n\le tq^m\le n+1. 
\]
By Lemma~\ref{basic_t}(1) and~(4), we have 
\[
C^m(\cI^nM)=\tau(M,f^{\frac{n}{q^m}})\supset \tau(M,f^t)\supset\tau(M,f^{\frac{n+1}{q^m}})= C^m(\cI^{n+1}M). 
\]
If $t\in\FJN^-$ and  $n<tq^m$, then we have $\tau(M,f^{\frac{n}{q^m}})\supsetneq \tau(M,f^t)$. 
On the other hand, if $t\in\FJN^+$ and $tq^m<n+1$, then $\tau(M,f^t)\supsetneq\tau(M,f^{\frac{n+1}{q^m}})$. In both cases, we have  $C^m(\cI^nM)\supsetneq C^m(\cI^{n+1}M)$. This proves~(1) and~(2).

\medskip 

\noindent
We now prove~(3). Let $n\in\nu_{\cI}(q^m;M)$. Let $\{t_i\}_{i\ge1}$ be a sequence as in Lemma~\ref{lem: nu invariants produce FJN}, and let $t\in\R$ be its limit. We have $n\le tq^m\le n+1$. We divide the proof into two cases, as in Lemma~\ref{lem: nu invariants produce FJN}(2). 

First assume that $t_i<t$ for all $i$. In this case, we have $t\in\FJN^-$ and $n\le t_iq^m<tq^m$. It follows that $\psi_m^-(t)=n$. 

Next assume that $t_i=t$ for some $i$. In this case, $t\in\FJN^+$ and $tq^m=t_iq^m<n+1$. Hence, $\psi^+_m(t)=n$. This proves the claim.  
\end{proof}

In the following lemma, we show that $\FJN^-$ can be recovered from the family  $\bigl(\nu_{\cI}(q^m;M)\bigr)_{m\ge1}$. 
\begin{lem}\label{lem: nu invariants recover FJN-}
Let $t\in\R_{>0}$. Then $t\in \FJN^-$ if and only if 
\[
\psi^-(tq^m)\in\nu_{\cI}(q^m;M)\qquad\text{for all sufficiently large }m\ge1. 
\]
\end{lem}
\begin{proof}
    The ``only if'' direction follows from Theorem~\ref{thm: FJN and nu-invariant}(1). We prove the ``if'' direction. Choose an integer $m\ge1$ such that 
    \[
    \psi^-(tq^{m+i})\in\nu_{\cI}(q^{m+i};M)\qquad\text{for all }i\ge1. 
    \]
Then the sequence $\{t_i\}_{i\ge1}$ defined by $t_i=\psi^-(tq^{m+i})\cdot q^{-(m+i)}$ satisfies the conditions in Lemma~\ref{lem: nu invariants produce FJN}(1). By construction, we have $t_i<t$ for all $i$, and $\{t_i\}_{i\ge1}$ converges to $t$. Hence, $t \in \FJN^-$. 
\end{proof}

The set $\nu_\cI(q^\infty;M)$ can be described in terms of $\bigl(\FJN^-(M,\cI),\FJN^+(M,\cI)\bigr)$ as follows. 
\begin{thm}\label{BSRFJN}
Let $M$ be an $F$-pure $q$-Cartier module on $X$. 
The following statements hold: 
\begin{enumerate}
    \item We have \[
\nu_\cI(q^\infty;M)\cap\bigl(\Z_{(p)}\cap[-1,0)\bigr)=-\Big(\FJN^-(M,\cI)\cap\Z_{(p)}\cap (0,1]\Big). 
\]
\item We have $0\in  \nu_\cI(q^\infty;M)$ if and only if $0\in\FJN^+(M,\cI)$. 
\end{enumerate}
\end{thm}
Assume that Assumption~\ref{assumption: nu invariant is bound} holds with $N=M$. Then, by Theorem~\ref{BSRrat}(2), we have 
\[\nu_{\cI}(q^\infty; M)\subset\Z_{(p)}\cap[-1,0]. \]
Therefore, under this assumption, the theorem gives a complete description of $\nu_\cI(q^\infty;M)$ in terms of test modules. 
\begin{proof}
 We first prove part~(1). It suffices to show that, for any $\alpha\in\Z_{(p)}\cap (0,1]$, 
 \[-\alpha\in\nu_\cI(q^\infty;M)\quad \text{if and only if}\quad  \alpha\in\FJN^-(M,\cI) . \]

We compute an expansion of $-\alpha$ in $\Z_p$ and an expansion of $\alpha$ in $\R$. 
Since $\alpha\in\Z_{(p)}\cap(0,1]$, we can write 
\begin{equation*}\label{exp}
\alpha=\frac{r}{q^m-1}
\end{equation*}
for some integers $m\geq1$ and $1\leq r\leq q^m-1$. Write 
\[
r=\sum_{i=0}^{m-1}r_iq^i,\qquad r_i=0,\dots,q-1. 
\]
Extend the function $i\mapsto r_i$ to all  integers by the rule $r_i=r_{i-m}$. 

With this notation, we obtain the following expansion in $\Z_p$: 
\begin{equation}
    \label{ainZ_p}
    -\alpha=r(1+q^m+q^{2m}+\cdots)=\sum_{i=0}^\infty r_iq^i. 
\end{equation}
On the other hand, in $\R$, we have 
\begin{equation*}\label{ainR}
    \alpha=\Big(\sum_{i=1}^m\frac{r_{-i}}{q^i}\Big)\cdot\Big(1+\frac{1}{q^m}+\cdots\Big)=\sum_{i=1}^\infty\frac{r_{-i}}{q^i}. 
\end{equation*}

By Lemma \ref{BSRnu}(2) and the expansion in \eqref{ainZ_p}, the condition $-\alpha\in\nu_\cI(q^\infty;M)$ is equivalent to 
\[
\sum_{i=0}^{em-1}r_iq^i\in \nu_\cI(q^{em};M)
\]
for all sufficiently large integers $e\geq1$. On the other hand,    
\[
 q^{em}\sum_{i=1}^{em}\frac{r_{-i}}{q^i}=\sum_{i=0}^{em-1}r_iq^i. 
\]
Since $\alpha\neq0$, we have $r_i>0$ for some $i$. Hence, 
\[
q^{em}\sum_{i=1}^{em}\frac{r_{-i}}{q^i}=\psi^-(\alpha q^{em}), 
\]
where $\psi^-(t)=\lceil t\rceil-1$. 

Therefore, it suffices to show that $\alpha\in\FJN^-(M,\cI)$ if and only if 
\[
\psi^-(\alpha q^{em})\in \nu_\cI(q^{em};M)
\]
 for all sufficiently large $e\geq1$. 
This follows from Lemma~\ref{lem: nu invariants recover FJN-} applied with $C$ replaced by $C^m$ and $q$ replaced by $q^m$. This proves~(1). 

\medskip

    It remains to  prove~(2). By Lemma~\ref{BSRnu}, we have $0\in\nu_\cI(q^\infty;M)$ if and only if 
    \[
    M=C^m(M)\supsetneq C^m(\cI M)
    \]
    for all integers $m\geq1$. On the other hand, we have $0\in \FJN^+(M,\cI)$  if and only if 
    \[
    M=\tau(M,\cI^0)\supsetneq \tau(M,\cI^{\frac{1}{q^m}})=C^m(\cI M)
    \]
    for all $m\geq1$. This proves~(2).  
\end{proof}

\begin{rem}
    Theorems~\ref{thm: FJN and nu-invariant} and~\ref{BSRFJN} appear to be closely related. It is not clear to the author whether there exists a statement that unifies these results. 
\end{rem}

\subsection{\texorpdfstring{The Submodule $M_{\cI!}$}{The Submodule MI!}}
In this subsection, we define a certain submodule of an $F$-pure $q$-Cartier module. 
Let ${\cal C}_X$ denote the graded $\cO_X$-algebra 
\[
\mathcal{C}_X:=\bigoplus_{m\geq0}C^m\cdot \cO_X,
\]
constructed as follows. As a right $\cO_X$-module, $\mathcal{C}_X$ is free with basis $\{C^m\}_{m\ge0}$. The multiplication is determined by the relation 
\[
C\cdot a^q=a\cdot C
\]
for all local sections $a\in \cO_X$. Then the notion of a $q$-Cartier module coincides with that of a left $\mathcal{C}_X$-module. 

Let 
 \[
 {\cal C}_{X,+}\coloneqq \bigoplus_{m>0}C^m\cdot\cO_X
 \]denote its positive part. 
Set $U\coloneqq X\setminus Z(\cI)$. 
 
\begin{lem}\label{f-min}
    Let $M$ be an $F$-pure $q$-Cartier module on $X$. Then there exists a unique coherent $q$-Cartier submodule $N\subset M$ satisfying the following properties: 
    \begin{enumerate}
        \item $N|_U=M|_U$. 
        \item If $N'$ is a coherent $q$-Cartier submodule of $M$ such that $N'|_U=M|_U$, then 
        \[
        N\subset N'. 
        \]
    \end{enumerate}
    Explicitly, for any positive integer $n$, we have 
\[
N={\cal C}_X(\cI^nM)={\cal C}_{X,+}(\cI^nM). 
\]
\end{lem}
\begin{proof}
The uniqueness of $N$ is immediate. It is clear that the modules ${\cal C}_X(\cI^nM)$ and ${\cal C}_{X,+}(\cI^nM)$ satisfy condition~(1), and that 
\[
{\cal C}_X(\cI M)\supset{\cal C}_X(\cI^nM)\supset{\cal C}_{X,+}(\cI^nM). 
\]
Thus it suffices to show that ${\cal C}_X(\cI M)$ satisfies condition~(2). 

Let $N'$ be a coherent $q$-Cartier submodule of $M$ such that $N'|_U=M|_U$. Since $M$ is coherent and $X$ is quasi-compact, there exists an integer $m\geq1$ such that 
\[\cI^m\cdot (M/N')=0.\]
On the other hand, the $q$-Cartier module $M/N'$ is $F$-pure. Hence, its annihilator is a radical ideal by \cite[Lemma~3.13]{BB11}. It follows that $\cI\cdot (M/N')=0$, that is, 
\[\cI M\subset N'.\]
Consequently, 
\[
{\cal C}_X(\cI M)\subset N'. 
\]
This proves the claim. 
\end{proof}

\begin{defn}\label{defn: zero extension of Cartier module}
For an $F$-pure $q$-Cartier module $M$ on $X$, we denote by $M_{\cI!} $ 
the submodule $N$ satisfying the properties in Lemma~\ref{f-min}. When $\cI=f\cO_X$, we also write $M_{f!}\coloneqq M_{\cI!}$. 
\end{defn}
The notation $M_{\cI!}$ is motivated by the observation that $M_{\cI!}$ should be the image of the natural morphism 
\[``\ j_!j^\ast M \ " \to M. 
\]
Here, $j\colon U\hookrightarrow X$ denotes the open immersion, and 
$j_!$ denotes the (nonexistent) extension-by-zero  functor for coherent $q$-Cartier modules. 

\begin{lem}\label{lem: right reduced}
Let $M$ be an $F$-pure $q$-Cartier module on $X$. Then the following statements hold: 
\begin{enumerate}
    \item For all sufficiently large integers $m\gg0$, we have 
    \[
    M_{\cI!}=C^m(\cI M). 
    \]
    \item The $q$-Cartier submodule $M_{\cI!}$ is $F$-pure. 
\end{enumerate}
\end{lem}

\begin{proof}
(1) By Lemma~\ref{f-min}, we have 
\[
M_{\cI!}={\cal C}_X(\cI M)=\sum_{m\geq0}C^m(\cI M). \]
Since $M$ is Noetherian, it suffices to show that 
\[C^m(\cI M)\subset C^{m+1}(\cI M)\]
for every integer $m\geq0$, so that the ascending chain stabilizes. As this can be checked locally, we may assume $\cI=f\cO_X$. Since $M=C(M)$, we have 
\[
C^m(fM)=C^m(fC(M))=C^{m+1}(f^pM)\subset C^{m+1}(fM). 
\]
This proves~(1). 

(2) By~(1), we may write $M_{\cI!}=C^m(\cI M)$ for all $m\gg0$. Hence 
\[C(M_{\cI!})=C^{m+1}(\cI M)=M_{\cI!}. 
\]
Thus $M_{\cI!}$ is $F$-pure. 
    \end{proof}

For $t\in\R_{\geq0}$, let $\cI^t{\cal C}_X$ denote the $\cO_X$-subalgebra of $\mathcal{C}_X$ defined by 
\[
\cI^t{\cal C}_X:=\bigoplus_{m\geq0}C^m\cdot \cI^{\lceil t(q^m-1)\rceil}. 
\]
       \begin{lem}\label{lem: test module is over ft}
       Let $M$ be an $F$-pure $q$-Cartier module on $X$. Then $\tau(M,\cI^t)$ is a module over $\cI^t{\cal C}_X$. 
    \end{lem}
    \begin{proof}
    Since the assertion is local, we may assume that $\cI=f\cO_X$. Let $m,m'\geq0$ be integers. Then 
        \[
        C^m\big(f^{\lceil t(q^m-1)\rceil}\cdot C^{m'}(f^{\lceil tq^{m'}\rceil}M)\big)=C^{m+m'}(f^{q^{m'}\lceil t(q^m-1)\rceil+\lceil tq^{m'}\rceil}M)\subset C^{m+m'}(f^{\lceil tq^{m+m'}\rceil}M). 
        \]
    This shows that 
    \[
\cI^t\mathcal{C}_X\cdot\tau(M,\cI^t)\subset\tau(M,\cI^t), 
    \]
    and hence $\tau(M,\cI^t)$ is a module over $\cI^t\mathcal{C}_X$. 
    \end{proof}
    
The test module $\tau(M,\cI^t)$ has the following description as an $\cI^t\mathcal{C}_X$-module. 
\begin{prop}\label{prop: intrinsic description of test module}
    Let $M$ be an $F$-pure $q$-Cartier module, and let $t\in\R_{\geq0}$. Assume that either $M=M_{\cI!}$ or $t\notin\Z[\frac{1}{p}]$. 
    Then the following hold:  
    \begin{enumerate}
        \item For any integer $n\geq t$, we have 
        \[\tau(M,\cI^t)=\cI^t{\cal C}_X\cdot \cI^nM. 
        \]
        \item Let $N$ be an $\cO_X$-coherent $\cI^t{\cal C}_X$-submodule of $M$ such that $N|_U=M|_U$. Then 
    \[
    \tau(M,\cI^t)\subset N. 
    \]
    \end{enumerate}
    Consequently, $\tau(M,\cI^t)$ is the smallest $\cO_X$-coherent $\cI^t\mathcal{C}_X$-submodule of $M$ such that $\tau(M,\cI^t)|_U=M|_U$. 
\end{prop}
\begin{proof} 
Let $N\subset M$ be a submodule satisfying the condition in~(2). 
Since $M/N$ is coherent and $(M/N)|_U=0$, it follows that, for all sufficiently large integers $n\ge1$, 
\[
\cI^nM\subset N. 
\]
Hence, 
\[\cI^t{\cal C}_X\cdot \cI^nM\subset N. \]
Therefore, statement~(2) follows once~(1) is proved. 

\medskip

It remains to prove (1). Since the assertion is local, we may assume that $\cI=f\cO_X$. For any $m\geq0$, we compute 
\begin{align*}
C^m\cI^{\lceil t(q^m-1)\rceil}\cdot \cI^nM=C^m&f^{\lceil t(q^m-1)\rceil+n}M\\
&\subset C^mf^{\lceil tq^m\rceil}M\subset \tau(M,\cI^t). 
\end{align*}
Thus, 
\[\cI^t{\cal C}_X\cdot \cI^nM \subset\tau(M,\cI^t). \]
We now prove the reverse inclusion 
\[\tau(M,\cI^t)\subset \cI^t{\cal C}_X\cdot \cI^nM . \]
It suffices to show that 
\[C^mf^{\lceil tq^m\rceil}M\subset \cI^t{\cal C}_X\cdot f^nM\]
for every $m\geq0$. 

\medskip

First assume that $tq^m$ is not an integer. Then we may write 
\[
\lceil tq^m\rceil=tq^m+\epsilon
\]
for some $\epsilon>0$. 
Choose a positive integer $m'$ satisfying $\epsilon q^{m'}>n-t$. Since $M$ is $F$-pure, we have $M=C^{m'}M$. Hence, 
\begin{align*}
    C^mf^{\lceil tq^m\rceil}M&=
    C^mf^{\lceil tq^m\rceil }C^{m'}M
    = C^{m+m'}f^{\lceil tq^m\rceil q^{m'}}M. 
\end{align*}
By the choice of $m'$, we have 
\[
\lceil tq^m\rceil q^{m'}=tq^{m+m'}+\epsilon q^{m'}>t(q^{m+m'}-1)+n. 
\]
This shows the desired inclusion in this case. 

Now assume that $N\coloneqq tq^m$ is an integer. In this case, we further assume that $M=M_{\cI!}$. 
Then 
\begin{align*}
f^NM&=f^N\sum_{m'\geq0}C^{m'}f^{\lceil n-t\rceil}M\\
&=\sum_{m'\geq0}C^{m'}f^{Nq^{m'}+\lceil n-t\rceil}M=
    \sum_{m'\geq0}C^{m'}f^{\lceil t(q^{m+m'}-1) \rceil}f^nM. 
\end{align*}
Here, the first equality follows from Lemma \ref{f-min}, since $M=M_{\cI!}$. 
Consequently, 
\[C^m f^{tq^m}M\subset \cI^t{\cal C}_X\cdot f^nM,
\]
which completes the proof. 
\end{proof}

We prove that the filtration defined by test modules is right-continuous under the assumption that $M=M_{\cI!}$. 

\begin{prop}\label{right-conti}
    Assume that $M=M_{\cI!}$. Then, for every $t\geq0$, there exists $\epsilon>0$ such that 
    \[
    \tau(M,\cI^s)=\tau(M,\cI^t)\quad\text{for all }s\in [t,t+\epsilon). 
    \]
  Consequently, 
    \[
    \FJN^+(M,\cI)=\emptyset \quad\text{and}\quad  \FJN^-(M,\cI)= \FJN(M,\cI). 
    \]
\end{prop}
\begin{proof}
We show that any $t\in\R_{\ge0}$ does not belong to $\FJN^+(M,\cI)$. By Lemma~\ref{+1/p}, it suffices to treat the case where $t\in\Z[\frac{1}{p}]$. 

Choose an integer $m\ge0$ such that $n:=tq^m\in\Z$. By Lemma~\ref{basic_t}(4), we have 
\[
\tau(M,\cI^t)=C^m(\cI^nM). 
\]
Since $M=M_{\cI!}$, Lemma~\ref{lem: right reduced}(1) implies that 
\[
M= C^i(\cI M)
\]
for some $i>0$. 
Hence, 
\[
\tau(M,\cI^t)=C^m(\cI^nM)= C^{i+m}(\cI^{nq^i+1}M)=\tau(M,\cI^{t+\frac{1}{q^{i+m}}}). 
\]
This shows that $t\notin\FJN^+(M,\cI)$, and the claim follows.  
\end{proof}

\begin{cor}\label{BSRfmin}
Let $M$ be an $F$-pure $q$-Cartier module on $X$. Assume that Assumption~\ref{assumption: nu invariant is bound} holds with $N=M$.  
Assume further that $M=M_{\cI!}$. Then
\[\nu_\cI(q^\infty;M)\subset \Z_{(p)}\cap [-1,0). 
\]
\end{cor}
\begin{proof}
By Theorem~\ref{BSRrat}(2), it remains to show that $0$ is not a $\nu$-invariant. This follows from Theorem~\ref{BSRFJN}(2), since $\FJN^+(M,\cI)=\emptyset$. 
\end{proof}

\section{\texorpdfstring{Preliminaries on Differential Operators}{Preliminaries on Differential Operators}}\label{section: Preliminaries on Differential Operators}

Let $X$ be a Noetherian regular $F$-finite  scheme over $\F_p$, and let $F$ denote the $p$-th Frobenius morphism $X\to X$. In this preliminary section, we recall definitions and results on differential operators in positive characteristic, which will be used in the following sections. 

Since $F$ is the identity on the underlying space,  the morphism $(F^m)^{\ast}\colon (F^m)^{-1}\cO_X\to \cO_X$ can be identified with the inclusion $\cO_X^{p^m}\hookrightarrow \cO_X$. In what follows, we will freely use this identification. 

\subsection{Rings of Differential Operators}\label{subsection: Ring of Differential Operators}

We begin by recalling the notion of differential operators. Let $J$ denote the kernel of the multiplication morphism  
\[
\cO_X\otimes_{\F_p}\cO_X\to\cO_X,\quad a\otimes b\mapsto ab. 
\]
An $\F_p$-linear endomorphism $P\colon \cO_X\to \cO_X$ is called a \emph{differential operator} if the induced morphism  
\[\bar{P}\colon \cO_X\otimes_{\F_p}\cO_X\to \cO_X,\quad \bar{P}(a\otimes b)=aP(b)\] 
satisfies $\bar{P}(J^{n+1})=0$ for some integer $n\geq0$. We denote by $\cD_X$ the sheaf of differential operators, defined as  the subalgebra of $\mathcal{E}{\rm nd}_{\F_p}(\cO_X)$ consisting of differential operators.

%When $X$ is affine ${\rm Spec}(R)$, we also write $D_R\coloneqq\Gamma(X,\cD_X)$. This is the ring of differential operators on $R$. 

Since $X$ is assumed $F$-finite, the ring $\cD_X$ admits the following description. For every integer $m\ge0$, set 
\[
\cD_X^{(m)}\coloneqq \mathcal{E}{nd}_{\cO_X^{p^m}}(\cO_X). 
\]
\begin{prop}[{\cite[Theorem~1.4.9]{Yek92}}]\label{prop: D=cup De}
As subalgebras of $\mathcal{E}{nd}_{\F_p}(\cO_X)$, we have 
    \[
    \cD_X=\bigcup_{m\ge0}\cD_X^{(m)}. 
    \]
\end{prop}
\begin{proof}
For the proof, see \cite[Theorem~1.4.9]{Yek92} or \cite[2.5]{SV97}. 
\end{proof}

\subsubsection{Frobenius Descent}\label{subsubsection: Frobenius descent}

We recall a version of Frobenius descent. For an $\cO_X$-module $M$, its pullback $F^{m\ast}M=\cO_X\otimes_{F^m,\cO_X}M$ naturally carries the structure of a $\mathcal{D}_X^{(m)}$-module, defined by 
\[
P\cdot(a\otimes x)=P(a)\otimes x\qquad (P\in \cD_X^{(m)},\, a\otimes x\in F^{m\ast}M). 
\]

\begin{thm}\label{thm: Frobenius descent}
    The functor $M\mapsto F^{m\ast}M$ defines an equivalence of categories between the category of $\cO_X$-modules and that of left $\cD_X^{(m)}$-modules. 
\end{thm}
\begin{proof}
    For the proof, see \cite[Proposition~2.1]{AMBL05}. 
\end{proof}
A quasi-inverse to $M\mapsto F^{m\ast}M$ can be  constructed as follows. 
\begin{lem}\label{lem: Frobenius descendent}
    Let $N$ be a $\cD_X^{(m)}$-module. Set 
    \[
    N_{(m)}\coloneqq \{x\in N\mid P(ax)=P(a)x\quad (\forall a\in\cO_X, \forall P\in \cD_X^{(m)})\}. 
    \]
This is an $\cO_X^{p^m}$-submodule of $N$, and the induced morphism $\cO_X\otimes_{\cO_X^{p^m}}N_{(m)}\to N$ is an isomorphism of $\cD_X^{(m)}$-modules. 
\end{lem}
\begin{proof}
We show that the morphism $\cO_X\otimes_{\cO_X^{p^m}}N_{(m)}\to N$ is an isomorphism. 
By the definition of $N_{(m)}$, it is clear that this morphism is $\cD_X^{(m)}$-linear. 

By Theorem~\ref{thm: Frobenius descent}, there exists an isomorphism of $\cD_X^{(m)}$-modules 
\[
N\cong F^{m^\ast}N'=\cO_X\otimes_{\cO_X^{p^m}}N'
\]
for some $\cO_X^{p^m}$-module $N'$. Since $F^m$ is faithfully flat, the morphism $N'\to F^{m\ast}N', x\mapsto 1\otimes x$ is injective, and identifies $N'$ with its image. Therefore, it suffices to show the equality 
\[
(F^{m\ast}N')_{(m)}=1\otimes N'. 
\]
The inclusion $\supset$ is clear. We prove the reverse inclusion $\subset$. Since this claim is local, we may assume that the inclusion $\cO_X^{p^m}\hookrightarrow\cO_X$ admits an $\cO_X^{p^m}$-linear splitting, and fix such a splitting $\bar{P}$. Let 
\[
P\colon \cO_X\xrightarrow{\bar{P}}\cO_X^{p^m}\hookrightarrow \cO_X
\]
denote the composite, 
which defines a section of $\cD_{X}^{(m)}$. 
Now let $y\in (F^{m\ast}N')_{(m)}$. Then 
\[
P(y)=P(1)y=y. 
\]
    By the construction of $P$, the left-hand side lies in $1\otimes N'$. This proves the claim. 
\end{proof}

\medskip 

Using the above results, we describe the category of left $\cD_X$-modules as follows. 
\begin{const}
Consider the following category $\mathscr{S}_X$. An object of $\sS_X$ is a system $M_{(\bullet)}=(M_{(m)})_{m\ge0}$ of the form  
\[
\cdots\to M_{(2)}\to M_{(1)}\to M_{(0)}, 
\]
where $M_{(m)}$ is an $\cO_X^{p^m}$-module and $M_{(m+1)}\to M_{(m)}$ is an $\cO_X^{p^{m+1}}$-linear morphism such that the induced morphism 
$\cO_X^{p^m}\otimes_{\cO_X^{p^{m+1}}}M_{(m+1)}\to M_{(m)}$ is an isomorphism. We refer to the morphism $M_{(m+1)}\to M_{(m)}$  as the \emph{transition morphism}. 

For two objects $M_{(\bullet)}$ and $ N_{(\bullet)}$, a morphism $M_{(\bullet)}\to N_{(\bullet)}$ is a system of $\cO_X^{p^m}$-linear morphisms $M_{(m)}\to N_{(m)}$ compatible with the transition morphisms. 

Let $M_{(\bullet)}$ be an object of $\sS_X$. For every $m\ge0$, we have an isomorphism 
$\cO_X\otimes_{\cO_X^{p^m}}M_{(m)}\cong M_{(0)}$. Via this isomorphism, we endow $M_{(0)}$ with a $\cD_X^{(m)}$-module structure. Since these $\cD_X^{(m)}$-module structures on $M_{(0)}$ are compatible with each other, they define a $\cD_X$-module structure on $M_{(0)}$. We denote by $\Phi M_{(\bullet)}$ the resulting $\cD_X$-module. 

\medskip

Conversely, for a $\cD_X$-module $\sM$, we construct an object of $\sS_X$ as follows. For each $m\ge0$, define 
\[
    M_{(m)}\coloneqq \{x\in \sM\mid P(ax)=P(a)x\quad(\forall a\in\cO_X,\forall P\in\cD_X^{(m)})\}. 
    \]
This is an $\cO_X^{p^m}$-module. The family $(M_{(m)})_{m\ge0}$ is equipped with transition morphisms  
given by the inclusions $M_{(m+1)}\hookrightarrow M_{(m)}$. 
\end{const}
\begin{lem}\label{lem: Mm is an object of S}
The system $(M_{(m)})_m$ constructed above defines an object of $\sS_X$. 
\end{lem}
\begin{proof}
It suffices to prove that for every $m\ge0$,  the morphism $\cO_X^{p^m}\otimes_{\cO_X^{p^{m+1}}}M_{(m+1)}\to M_{(m)}$ is an isomorphism. Since $F^m$ is faithfully flat, it suffices to check this after the scalar extension $F^{m\ast}$. Hence, we are reduced to proving that 
\[
F^{m+1\ast}M_{(m+1)}\to F^{m\ast}M_{(m)}
\]
is an isomorphism. By  Lemma~\ref{lem: Frobenius descendent}, this morphism is identified with the identity morphism $\sM\to \sM$. This proves the claim. 
\end{proof}
\begin{prop}\label{prop: Frobenius descent for Dmodule}
The assignment $M_{(\bullet)}\mapsto \Phi M_{(\bullet)}$ defines an equivalence of categories from $\sS_X$ to the category of left $\cD_X$-modules. A quasi-inverse is given by $\sM\mapsto (M_{(m)})_m$. 
\end{prop}
\begin{proof}
   The statement is clear from the construction of the functors. 
\end{proof}

The above proposition simplifies several  constructions for $\cD_X$-modules. 
\begin{const}\label{construction: pullback pushforward tensor product of Dmodules}
Let $\sM$ be a $\cD_X$-module with corresponding system $M_{(\bullet)}$. 
    \begin{enumerate}
\item Let $f\colon Y\to X$ be a morphism of  Noetherian regular $F$-finite  $\F_p$-schemes. For every $m\ge0$, we set 
\[
N_{(m)}\coloneqq \cO_Y^{p^m}\otimes_{f^{-1}\cO_X^{p^m}}f^{-1}M_{(m)}. 
\]
The transition morphisms $M_{(m+1)}\to M_{(m)}$ induce morphisms $N_{(m+1)}\to N_{(m)}$. One checks that the system $(N_{(m)})_m$ defines an object of $\sS_Y$. We denote by $f^\ast \sM$ the corresponding $\cD_Y$-module. By construction, the underlying $\cO_Y$-module of $f^\ast \sM$   coincides with the usual pullback of the $\cO_X$-module $\sM$. 

        \item Let $f\colon X\to Y$ be an  \'etale morphism of Noetherian regular $F$-finite  $\F_p$-schemes. For every $m\ge0$, we set $N_{(m)}\coloneqq f_\ast( M_{(m)})$,   which is an $\cO_Y^{p^m}$-module. Since $\cO_Y^{p^m}$ is locally finite free over $\cO_Y^{p^{m+1}}$, the canonical morphism 
        \[
 \cO_Y^{p^m}\otimes_{\cO_Y^{p^{m+1}}}f_\ast (M_{(m+1)})\to f_\ast(f^{-1}\cO_Y^{p^m}\otimes_{f^{-1}\cO_Y^{p^{m+1}}} M_{(m+1)})
        \]
        is an isomorphism. Moreover, since $f$ is \'etale, the right-hand side is canonically isomorphic to 
        \[f_\ast(\cO_X^{p^m}\otimes_{ \cO_X^{p^{m+1}}} M_{(m+1)})\cong f_\ast( M_{(m)}).
        \]
      This shows that the induced morphism $\cO_Y^{p^m}\otimes_{\cO_Y^{p^{m+1}}}N_{(m+1)}\to N_{(m)}$ is an isomorphism, and hence $(N_{(m)})_m$ defines an element of $\sS_Y$. We denote by $f_\ast \sM$ the corresponding $\cD_Y$-module. By construction, the underlying $\cO_Y$-module of $f_\ast \sM$ coincides with the usual pushforward of $ \sM$. 
        
        \item Let $\sM$ and $\sN$ be two $\cD_X$-modules with corresponding systems   $(M_{(m)})_m$ and $ (N_{(m)})_m$. 
For every $m\ge0$, set 
\[L_{(m)}\coloneqq M_{(m)}\otimes_{\cO_X^{p^m}}N_{(m)}. \]
In a natural way, the modules 
$(L_{(m)})_m$ form an object of $\sS_X$. We denote by $\sM\otimes_{\cO_X}\sN$ the associated $\cD_X$-module. As $\cO_X$-modules, this coincides with the usual tensor product.  
    \end{enumerate}
\end{const}

\subsection{\texorpdfstring{Differential Operators on $\A^1$}{Differential Operators on A1}}
In this subsection, we study  $\cD$-modules over $\A^1_X$. We begin by introducing an important algebra. 

\subsubsection{The Algebra $(D_{\F_p[t]})_0$}
Set  $D_{\F_p[t]}\coloneqq\Gamma(\A^1_{\F_p},\cD_{\A^1_{\F_p}})$. 
We consider the following subalgebra of $D_{\F_p[t]}$: 
\[
(D_{\F_p[t]})_0:=\{P\in D_{\F_p[t]}\mid \forall n\ge0, \exists a_n\in\F_p\text{ such that } P(t^n)=a_nt^n\}. 
\]
In other words, this is the subring of differential operators preserving the grading $\F_p[t]=\bigoplus_{m\ge0}\F_p\cdot t^m$. 

This algebra admits the following description. 
\begin{lem}[{\cite[Proposition~III.5]{Qui21}}]\label{lem: D0 has a simple des}
Let $\mathscr{C}\coloneqq C(\Z_p,\F_p)$ denote the ring of locally constant functions $\Z_p\to\F_p$, where $\Z_p$ is equipped with the $p$-adic topology. 
Then the following hold: 
\begin{enumerate}
    \item For every $\varphi\in \sC$, the $\F_p$-linear map 
    \[
    \tilde{\varphi}\colon \F_p[t]\to\F_p[t],\quad t^m\mapsto \varphi(-m-1)\cdot t^m
    \]
    is a differential operator. It defines an element of $(D_{\F_p[t]})_0$. 
    \item The assignment $\varphi\mapsto \tilde{\varphi}$ defines a ring isomorphism 
    \[
    \sC\xrightarrow{\cong} (D_{\F_p[t]})_0. 
    \]
\end{enumerate}
\end{lem}

We adopt the normalization $\tilde{\varphi}(t^m) = \varphi(-m-1)\,t^m$, rather than $\tilde{\varphi}(t^m) = \varphi(m)\,t^m$. With this choice, the functor $N_f(-)_\alpha$ introduced in Section~\ref{section: Nearby Cycles via the Graph Embedding} is related to the $\mathcal{D}$-module defined by $f^\alpha$. See Proposition~\ref{prop: non-integral case}(2) for details. 
\begin{proof}
    We include a proof for the reader's convenience. 

    (1) Let $\varphi\in \sC$. Since $\F_p$ is equipped with the discrete topology, there exists a positive integer $m\ge1$ such that $\varphi$ factors through the projection $\Z_p\to \Z/p^m\Z$. It then follows  that $\tilde{\varphi}$ is $\F_p[t^{p^m}]$-linear, and hence defines an element of $D_{\F_p[t]}$ by Proposition~\ref{prop: D=cup De}. 

    (2) Injectivity follows from the fact that $\Z_{<0}$ is dense in $\Z_p$. Let $P\in (D_{\F_p[t]})_0$. We show that $P$ lies in the image of the map $\sC\to (D_{\F_p[t]})_0$. For each integer $m<0$, define $\varphi(m)\in\F_p$ by the condition $P(t^{-m-1})=\varphi(m)t^{-m-1}$. 
    
    It remains to show that the map $\varphi\colon \Z_{<0}\to \F_p$ extends to a locally constant function $\Z_p\to \F_p$. Since $P\in D_{\F_p[t]}$, there exists an integer $e\ge0$ such that $P$ is $\F_p[t^{p^e}]$-linear. This implies that 
    $\varphi$ factors through the quotient $\Z_{<0}\to \Z/p^e\Z$, and the claim follows. 
\end{proof}

\begin{example}\label{example: binomial}
  For an integer $n\ge0$, let $\partial^{[n]}$ denote the differential operator defined by 
  \[
  \F_p[t]\to\F_p[t],\quad t^m\mapsto \binom{m}{n}t^{m-n}, 
  \]
  where we set $t^{m-n}=0$ if $m<n$. 
  We define  
  \[
  \vartheta_n\coloneqq \partial^{[n]}\cdot t^n\in(D_{\F_p[t]})_0. 
  \]
  Under the isomorphism of Lemma~\ref{lem: D0 has a simple des}(2), the operator $\vartheta_n$ corresponds to a certain function on $\Z_p$. This function is described as follows.  Consider the binomial polynomial 
\[
\binom{x}{n}\coloneqq\frac{x(x-1)\cdots(x-n+1)}{n!}\in\Q[x]. 
\]
Note that for $\alpha\in\Z_p$, the value $\binom{\alpha}{n}$ is a $p$-adic integer. Hence its reduction modulo $p$, 
\[\binom{\alpha}{n}\bmod p\in\F_p,\]
is well-defined. 
  
\begin{lem}\label{lem: varthetan}
    For every integer $n\ge0$, the operator $\vartheta_n$ corresponds to the function \[\alpha\mapsto(-1)^n\binom{\alpha}{n}\in\F_p.\]
    Equivalently, for every $m\ge0$, we have 
    \[
    \vartheta_n( t^m)=(-1)^n\binom{-m-1}{n}\cdot t^m. 
    \]
\end{lem}
\begin{proof}
The assertion follows from the equality of polynomials 
\[
\binom{x+n}{n}=(-1)^n\binom{-x-1}{n}. 
\]
\end{proof}

\end{example}

We recall standard properties of the ring $\sC$. 
\begin{lem}\label{lem: Spec C Zp}
The following statements hold: 
	\begin{enumerate}    
    \item Let $\alpha\in\Z_p$. Then the evaluation map 
    \[
\ev_\alpha\colon \sC\to \F_p, \quad \varphi\mapsto\varphi(\alpha)
\]
is surjective and flat. Moreover, the map $\ev_\alpha$ identifies the target $\F_p$ with the localization $\sC_{\ker(\ev_\alpha)}$ at the maximal ideal $\ker(\ev_\alpha)$. 
\item 
The assignment $\alpha\mapsto \ker(\ev_{\alpha})$ induces a homeomorphism 
\[
\Z_p\cong{\rm Spec}(\sC). 
\]
\end{enumerate}
\end{lem}
   \begin{proof}
       Since $\F_p$ is equipped with the discrete topology, the canonical map 
       \[
\varinjlim_{m\ge0}C(\Z_p/p^m\Z_p,\F_p)\to C(\Z_p,\F_p)
       \]
       is an isomorphism. The assertions then follow from the corresponding results for the finite product ring $C(\Z_p/p^m\Z_p,\F_p)\cong \prod_{\Z_p/p^m\Z_p}\F_p$. 
   \end{proof}

For $\alpha\in\Z_p$, we obtain a ring homomorphism $(D_{\F_p[t]})_0\to\F_p$ given by the composition  
\[
(D_{\F_p[t]})_0\xrightarrow\cong \sC\xrightarrow{\ev_\alpha}\F_p. 
\]
To compute this homomorphism, we use Lucas' theorem, which we record here in the following form. 

\begin{thm}[Lucas]\label{thm: Lucas theorem}
    Let $n$ be a nonnegative integer and let $\alpha\in\Z_p$. Write their $p$-adic expansions as 
    \[
    n=\sum_{i\ge0}n_ip^i,\quad \alpha=\sum_{i\ge0}a_ip^i\qquad (n_i,a_i\in\{0,\dots,p-1\}). 
    \]
    Then 
    \[
    \binom{\alpha}{n}\equiv \prod_{i\ge0}\binom{a_i}{n_i}\mod p. 
    \]
    Note that the product on the right-hand side is finite, since $n_i=0$ for all sufficiently large $i$. 
\end{thm}
\begin{proof}
The classical form of Lucas’ theorem establishes the result when $\alpha$ is a nonnegative integer. This case can be proved by examining the coefficients in the expansion 
\[
(1+x)^\alpha=\prod_{i\ge0}(1+x^{p^i})^{a_i}. 
\]
The general case follows from this since the binomial function is $p$-adically continuous, and the set of nonnegative integers is dense in $\Z_p$. 
\end{proof}

\begin{lem}\label{lem: vartheta C}
The following statements hold: 
\begin{enumerate}
\item Let $\alpha\in\Z_p$, and write $\alpha=\sum_{i\ge0}a_ip^i$ with $a_i=0,\dots,p-1$. Then the composite map 
\[
(D_{\F_p[t]})_0\xrightarrow{\cong} \sC\xrightarrow{{\rm ev}_\alpha}\F_p
\]
sends $\vartheta_{p^i}$ to $-a_i$. 

    \item Under the isomorphism $\sC\cong (D_{\F_p[t]})_0$, the subalgebra $C(\Z/p^m\Z,\F_p)\subset \sC$ corresponds to the subalgebra 
\[
\F_p[\vartheta_1,\vartheta_p,\dots,\vartheta_{p^{m-1}}]\subset (D_{\F_p[t]})_0. 
\]
\end{enumerate}
\end{lem}
\begin{proof}
We first prove part~(1). 
By Lemma~\ref{lem: varthetan}, the element $\vartheta_{p^i}$ corresponds to 
\[
\alpha\mapsto-\binom{\alpha}{p^i}. 
\]
By Theorem~\ref{thm: Lucas theorem}, we have $\binom{\alpha}{p^i}\equiv a_i\bmod p$. This proves the claim. 

We now prove~(2). Set $A\coloneqq \F_p[\vartheta_1,\vartheta_p,\dots,\vartheta_{p^{m-1}}]$ and $B\coloneqq C(\Z/p^m\Z,\F_p)$. 
By~(1), the isomorphism $(D_{\F_p[t]})_0\xrightarrow{\cong} \sC$ restricts to an injection $A\hookrightarrow B$. 
To show that this is an isomorphism, it suffices to prove that the induced morphism of schemes 
\[
\varphi\colon X\coloneqq{\rm Spec}(B)\to Y\coloneqq{\rm Spec}(A)
\]
is an isomorphism. Since both $A$ and $B$ are finite $\F_p$-algebras, the schemes $X$ and $Y$ are finite and discrete. Hence, it suffices to show that $\varphi$ is bijective and that, for each point $x\in X$, the induced homomorphism of local rings $\cO_{Y,f(x)}\to \cO_{X,x}$ is an isomorphism. The latter follows since the map is injective and $\cO_{X,x}=\F_p$. 

We now prove that $\varphi$ is bijective. The surjectivity of $\varphi$ follows from the fact that  $A\to B$ is finite and injective. We prove injectivity. To this end, we identify $X$ with the discrete set $\Z/p^m\Z$ via evaluation. Let $a,b\in \Z/p^m\Z$, and suppose that $\varphi(a)=\varphi(b)$. Write 
\[
a=\sum_{i=0}^{m-1}a_ip^i,\quad b=\sum_{i=0}^{m-1}b_ip^i\qquad(a_i,b_i\in\{0,\dots,p-1\}). 
\]
Evaluating at $\vartheta_{p^i}$, we obtain $a_i=b_i$ for all $i=0,\dots,m-1$. Hence $a=b$, which proves injectivity and completes the proof. 
\end{proof}

For later use, we record another consequence of Lucas' theorem. 
\begin{lem}\label{lem: expansion of polynomial}
Let $m$ be a positive integer, and $n$ be an integer with $0\le n\le p^m-1$. 
\begin{enumerate}
    \item In the ring $\F_p[x,y]$, we have 
\[
(x-y)^{p^m-n-1}=(-1)^n\sum_{r=0}^{p^m-n-1}\binom{n+r}{n}x^ry^{p^m-n-r-1}. 
\]
    \item Write $n=\sum_{i=0}^{m-1}n_ip^i$ with $n_i=0,\dots,p-1$. Then we  have 
\[
(x-y)^{p^m-n-1}=(-1)^{n}\sum_{r_0,\dots,r_{m-1}}
\prod_{i=0}^{m-1}\binom{n_i+r_i}{n_i}\cdot x^{\sum_{i=0}^{m-1}r_ip^{i}}\cdot
y^{p^m-n-\sum_{i=0}^{m-1}r_ip^i-1}
\]
in the ring $\F_p[x,y]$, 
where the sum is taken over all integers $r_0,\dots,r_{m-1}$ such that $0\leq r_i\leq p-n_i-1$. 
\end{enumerate}
\end{lem}
\begin{proof}
(1) 
It suffices to prove the congruence 
\[
(-1)^{p^m-n-r-1}\binom{p^m-n-1}{r}\equiv (-1)^n\binom{n+r}{n}\mod p. 
\]
Since $r<p^m$, Theorem~\ref{thm: Lucas theorem}  implies 
\[
\binom{p^m-n-1}{r}\equiv \binom{-n-1}{r} \mod p. 
\]
On the other hand, we have 
\[
\binom{-n-1}{r} =(-1)^r\binom{n+r}{r}=(-1)^r\binom{n+r}{n}. 
\]
This proves the claim. 

\medskip

\noindent
(2) We compute the coefficients in the right-hand side of~(1). 
Let $r$ be an integer with $0\le r\le p^m-n-1$, and write $r=\sum_{i=0}^{m-1}r_ip^i$ for its $p$-adic expansion. 

Suppose that $n_i+r_i\le p-1$ for all  $i=0,\dots,m-1$. Then 
\[
n+r=\sum_{i}(n_i+r_i)p^i
\]
is the $p$-adic expansion of $n+r$. Hence, 
by Theorem~\ref{thm: Lucas theorem}, we have 
\[
\binom{n+r}{n}=\prod_{i=0}^{m-1}\binom{n_i+r_i}{n_i}. 
\]

On the other hand, suppose that $n_i+r_i> p-1$ for some $i$, and let $i$ be the smallest index with this property. Then the coefficient of $p^i$ in the $p$-adic expansion of $n+r$ is equal to $n_i+r_i-p$, which is strictly less than $n_i$. Hence, again by Theorem~\ref{thm: Lucas theorem} we obtain 
\[\binom{n+r}{n}\equiv0\mod p.\]
The assertion then follows from (1). 
\end{proof}

\medskip

Let $X$ be a Noetherian regular $F$-finite   $\F_p$-scheme. 
Let $\A^1_X={\rm Spec}(\cO_X[t])$, and let  ${\rm pr}$ denote the projection $\A^1_X=X\times_{\F_p}\A^1_{\F_p}\to X$. 
\begin{const}
  We construct a ring homomorphism 
\[
{\rm pr}^{-1}\cD_X\otimes_{\F_p}D_{\F_p[t]}\to \cD_{\A^1_X}
\]
as follows. By adjunction, it suffices to construct 
\begin{equation}\label{equation: Dx Dft toDa1}
    \cD_X\otimes_{\F_p}D_{\F_p[t]}\to {\rm pr}_\ast\cD_{\A^1_X}. 
\end{equation}
Let $U\subset X$ be an affine open subscheme with coordinate ring $R=\Gamma(U,\cO_X)$, and set $D_R\coloneqq\Gamma(U,\cD_X)$. 
For $P\in D_R$  and $Q\in D_{\F_p[t]}$, the tensor product $P\otimes Q$ defines a differential operator on $R[t]=R\otimes_{\F_p}\F_p[t]$. 
This construction induces a ring homomorphism 
\[
\Gamma(U,\cD_X\otimes_{\F_p}D_{\F_p[t]})=D_R\otimes_{\F_p}D_{\F_p[t]}\to \Gamma(\A^1_U,\cD_{\A^1_X}). 
\]
Since this is compatible with the restriction maps for  affine open immersions $V\hookrightarrow U$, it glues to a morphism of sheaves. This defines the morphism in \eqref{equation: Dx Dft toDa1}. 
\end{const}

Set $\mathscr{C}\coloneqq C(\Z_p,\F_p)$ and $\sC^{(m)}\coloneqq C(\Z/p^m\Z,\F_p)$.  Composing the homomorphisms 
\[
{\rm pr}^{-1}\cD_X\otimes_{\F_p}\mathscr{C}\xrightarrow[\cong]{\text{Lemma~\ref{lem: D0 has a simple des}}}{\rm pr}^{-1}\cD_X\otimes_{\F_p}(D_{\F_p[t]})_0 \hookrightarrow {\rm pr}^{-1}\cD_X\otimes_{\F_p}D_{\F_p[t]}\to \cD_{\A^1_X}, 
\]
we regard $\cD_{\A^1_X}$-modules as ${\rm pr}^{-1}\cD_X\otimes_{\F_p}\mathscr{C}$-modules. In particular, the ring $\sC$ acts on $\cD_{\A^1_X}$-modules. 

\begin{defn}
Let $\sM$ be a left $D_{\F_p[t]}$-module, or more generally, a sheaf of left $D_{\F_p[t]}$-modules on a topological space. 
\begin{enumerate}
    \item For an element $a\in \Z/p^m\Z$, set 
\[
\sM_{a}\coloneqq \F_p\otimes_{{\rm ev}_{a},\sC^{(m)}}\sM, 
\]
where $\ev_{a}$ denotes the evaluation map 
\[
\sC^{(m)}\to \F_p,\quad \varphi\mapsto \varphi(a). 
\]
\item For an element $\alpha\in \Z_p$, set 
\[
\sM_{\alpha}\coloneqq \F_p\otimes_{{\rm ev}_\alpha,\sC}\sM. 
\]
\end{enumerate}
\end{defn}

We recall some basic properties of these constructions. 
For every integer $m\ge0$, the ring homomorphism 
\[
(\ev_{a})_{a}\colon \sC^{(m)}\to\prod_{a\in\Z/p^m\Z}\F_p
\]
is an isomorphism. Consequently, we obtain a decomposition 
\begin{align}\label{align: eigenspace decomposition of M}
\sM=\bigoplus_{a\in\Z/p^m\Z}
\sM_{a}.
\end{align}
This decomposition induces a natural projection $\sM\to \sM_{a}$.

Let $a_{m+1}\in\Z/p^{m+1}\Z$, and let $a_m$ denote its image  in $\Z/p^m\Z$. Since the composite map \[\sC^{(m)}\hookrightarrow\sC^{(m+1)}\xrightarrow{\ev_{a_{m+1}}}\F_p\]
coincides with $\ev_{a_m}$, the projection $\sM\to \sM_{a_{m+1}}$ factors through $\sM_{a_m}$, and induces 
 a natural surjection  $\sM_{a_m}\to \sM_{a_{m+1}}$

Now let $\alpha\in\Z_p$, and let $\alpha_m$ denote its image in $\Z/p^m\Z$. The morphisms $\sM_{\alpha_{m}}\to \sM_{\alpha_{m+1}}$ constructed above induce an isomorphism 
\[
\sM_\alpha\cong \varinjlim_m \sM_{\alpha_m}. 
\]

\begin{lem}\label{lem: how vartheta acts}
Let $a\in\Z/p^m\Z$, and write $a=\sum_{i=0}^{m-1}a_ip^i$ with $0\le a_i\le p-1$. Then, under the decomposition~\eqref{align: eigenspace decomposition of M}, we have 
\[
\{x\in \sM\mid \forall i=0,\dots, m-1,\, \vartheta_{p^i}x=-a_ix\}=\sM_{a}.  
\]
\end{lem}
\begin{proof}
The claim follows from Lemma~\ref{lem: vartheta C}. 
\end{proof}

\section{The Graph Embedding and Bernstein--Sato roots}\label{section: Nearby Cycles via the Graph Embedding}
Let $X$ be a Noetherian regular $F$-finite  $\F_p$-scheme. Fix  an element $f\in\Gamma(X,\cO_X)$. 
\subsection{Graph Construction for Differential Modules}\label{subsection: Graph Construction for Differential Modules}
For a functoriality of $\cD$-modules, we refer to Construction~\ref{construction: pullback pushforward tensor product of Dmodules}. 

Set $\A^1_X\coloneqq{\rm Spec}(\cO_X[t])$, and consider the morphisms of schemes 
\[X\xrightarrow{\gamma}\A^1_X\xrightarrow{{\rm pr}}X, 
\]
where ${\rm pr}\colon \A^1_X=X\times_{\F_p}\A^1_{\F_p}\to X$ is the projection, and $\gamma$ is the graph embedding of $f\colon X\to \A^1_{\F_p}$. Let $j_V$ denote the open complement $V\hookrightarrow \A^1_X$ of the closed immersion $\gamma$. We consider the following $\cD_{\A^1_X}$-module: 
\[
B_f\coloneqq (j_{V\ast}\cO_V)/\cO_{\A^1_X}.  
\]

\begin{defn}
    For a $\cD_X$-module $\sM$, we define the $\cD_{\A^1_X}$-module $\gamma_+ \sM$ by 
    \[
    \gamma_+\sM\coloneqq ({\rm pr}^\ast \sM)\otimes_{\cO_{\A^1_X}}B_f. 
    \]
\end{defn}

We regard $\gamma_+$ as the $\cD$-module pushforward along the closed immersion $\gamma$. 

Note that the sheaves $B_f$ and $\gamma_+\sM$ are supported on the graph of $f$. Thus, they may be  identified with sheaves on the topological space $X$.  
As sheaves on $X$, $B_f$ and $\gamma_+\sM$ are modules over the sheaf of rings  
\[
\gamma^{-1}({\rm pr}^{-1}\cD_X\otimes_{\F_p}D_{\F_p[t]})=\cD_X\otimes_{\F_p}D_{\F_p[t]}. 
\]
In particular, we regard them as $\cO_X$-modules via the canonical morphism $\cO_X\to \cD_X$.

For every integer $m\ge0$, set 
\[
\delta_m\coloneqq (f-t)^{-m-1}\in B_f. 
\]
Then, as an $\cO_X$-module, we have a direct sum decomposition 
\[
B_f=\bigoplus_{m\ge0}\cO_X\cdot \delta_m. 
\]
Tensoring with ${\rm pr}^\ast \sM$, we obtain 
\begin{equation*}
    \gamma_+\sM=\bigoplus_{m\ge0}\sM\cdot \delta_m. 
\end{equation*}

\subsubsection{Definition of Bernstein--Sato Roots}

Let $\sM$ be a  $\cD_X$-module. For an $\cO_X$-submodule $M\subset \sM$, we define 
\[
N_{f}(M)\coloneqq
(\cD_X\otimes_{\F_p} \mathscr{C})\cdot M\delta_0 \subset \gamma_+\sM. 
\]
For each $\alpha\in\Z_p$, we set 
\[
N_{f}(M)_\alpha\coloneqq N_f(M)\otimes_{\sC,\ev_\alpha}\F_p. 
\]

\begin{defn}[{cf. \cite[Definition~1.2.4]{Bit18}}]
\label{defn:Bernstein--Sato roots for Dmodule}
We identify $\Z_p\cong{\rm Spec}(\sC)$ as in Lemma~\ref{lem: Spec C Zp}. We define 
\[
\BSR(M,f)\coloneqq\{\alpha\in\Z_p\mid N_{f}(fM)_\alpha \subsetneq N_{f}(M)_\alpha\}. 
\]
An element of this set  is called a \emph{Bernstein--Sato root of $M$ with respect to $f$}. 
\end{defn}

We also define a finite-level version. 
\begin{defn}
For each integer $m\ge1$, we define 
\[
\BSR^{(m)}(M,f)\coloneqq\{n\in\Z_{\ge0}\mid \cD_X^{(m)}f^{n+1}M\subsetneq \cD_X^{(m)}f^nM\}. 
\]
\end{defn}
To describe the  relation between $\BSR(M,f)$ and $\BSR^{(m)}(M,f)$, we introduce a finite-level version $N_{f}^{(m)}(M)$ of $N_{f}(M)$.
\subsubsection{Finite-Level Version of $N_{f}(M)$}
Let $m$ be a positive integer, and set $\sC^{(m)}\coloneqq C(\Z/p^m\Z,\F_p)$. 
 We define 
\[
N^{(m)}_{f}(M)\coloneqq (\cD_X^{(m)}\otimes \sC^{(m)})\cdot M\delta_0. 
\]
Since $\cD_X=\bigcup_{m\ge 1}\cD_X^{(m)}$ and 
$\sC=\bigcup_{m\ge 1}\sC^{(m)}$, we have 
\[
N_{f}(M)=\bigcup_{m\ge1}N_{f}^{(m)}(M). 
\]

We describe the module $N^{(m)}_{f}(M)$ following the ideas of \cite{Mus09} and \cite{Bit18}. For an integer $n$ with $0\le n\le p^m-1$, define an  element $Q_n^{(m)}\in B_f$ by 
\[
Q_n^{(m)}\coloneqq t^{p^m-n-1}(f-t)^{-p^m}, 
\]
which was introduced by Musta\c{t}\u{a}\footnote{In \cite[Section~5]{Mus09}, Musta\c{t}\u{a}  considers a tuple $a=(i_0,\dots,i_{m-1})$ of integers with $i_l=0,\dots,p-1$, and defines an element $Q_{a}^0$ of $B_f$ by using the $F$-module structure of $B_f$. This element is equal to our $Q_{\sum_{l=0}^{m-1}i_lp^l}^{(m)}$. }. 
 We define a morphism  
\[
T_{n}^{(m)}\colon \sM\to\gamma_+\sM, \quad x\mapsto (1\otimes x)\otimes Q_n^{(m)}, 
\]
where $1\otimes x$ is regarded as an element of ${\rm pr}^\ast \sM=\cO_{\A^1_X}\otimes_{\cO_X}\sM$. 
Observe that $T_n^{(m)}$ is a morphism of $\cD_X^{(m)}$-modules. Indeed, the canonical morphism ${\rm pr}^{-1}\cD_X\to \cD_{\A^1_X}$ restricts to ${\rm pr}^{-1}\cD_X^{(m)}\to \cD_{\A^1_X}^{(m)}$. This implies that any element of $\cD_X^{(m)}$ commutes with multiplication by $(f-t)^{-p^m}$. On the other hand, it is clear that $\cD_X^{(m)}$ commutes with multiplication by $t^{p^m-n-1}$. 

\begin{lem}\label{lem: eigenspace}
Let $n\in \{0,\dots,p^m-1\}$, and 
let $\bar{n}\in \Z/p^m\Z$ denote its image. 
Then 
\[
T_{n}^{(m)}(\sM)\subset (\gamma_+\sM)_{\bar{n}}. 
\]
\end{lem}
\begin{proof}
Write $n=\sum_{i=0}^{m-1}n_ip^i$ with $0\le n_i\le p-1$. 
By Lemma~\ref{lem: how vartheta acts}, it suffices to show that for every $i=0,\dots,m-1$, 
\[
\vartheta_{p^i}Q_n^{(m)}=-n_iQ_n^{(m)}. 
\]
Since $\vartheta_{p^i}$ is $\cO_{\A^1_X}^{p^m}$-linear, we compute in $\cO_{\A^1_X}[(f-t)^{-1}]$  
\begin{align*}
    \vartheta_{p^i}Q_n^{(m)}&=\partial^{[p^i]}(t^{p^m+p^i-n-1}(f-t)^{-p^m})\\
    &=(f-t)^{-p^m}\binom{p^m+p^i-n-1}{p^i}t^{p^m-n-1}=\binom{p^m+p^i-n-1}{p^i}Q_n^{(m)}. 
\end{align*}
By Theorem~\ref{thm: Lucas theorem},  this binomial coefficient is congruent to $-n_i$. The assertion follows. 
 \end{proof}

We recall the following computation, due to Musta\c{t}\u{a}. 

\begin{lem}[{\cite[Lemma~5.3]{Mus09}}]\label{lem: Mustata}
Let $n$ be an integer with $0\le n\le p^m-1$, and write 
$n=\sum_{i=0}^{m-1}n_ip^i$ for its $p$-adic expansion. Then 
\[
Q_n^{(m)}=(-1)^{n}\sum_{r_0,\dots,r_{m-1}}
\prod_{i=0}^{m-1}
\binom{n_i+r_i}{n_i}\cdot f^{\sum_{i=0}^{m-1}r_ip^{i}}\cdot \delta_{n+\sum_{i=0}^{m-1}r_ip^i}, 
\]
where the sum is taken over all integers $r_0,\dots,r_{m-1}$ such that $0\leq r_i\leq p-n_i-1$. 
\end{lem}
\begin{proof}
We include a proof for the reader's  convenience. Since 
$Q_n^{(m)}=t^{p^m-n-1}(f-t)^{-p^m}$, 
   it suffices to prove that 
    \[
    t^{p^m-n-1}=(-1)^{n}\sum_{r_0,\dots,r_{m-1}}
\prod_{i=0}^{m-1}
\binom{n_i+r_i}{n_i}\cdot f^{\sum_{i=0}^{m-1}r_ip^{i}}\cdot (f-t)^{p^m-n-\sum_{i=0}^{m-1}r_ip^i-1}. 
    \]
    The equality follows from Lemma~\ref{lem: expansion of polynomial}(2) by setting $x=f$ and $y=f-t$. 
\end{proof}

The following lemma is proved by repeated applications of the above lemma. 

\begin{lem}\label{lem: injectivity of L}
    The morphism of $\cD_X^{(m)}$-modules 
    \[
    T_\bullet^{(m)}\colon \bigoplus_{n=0}^{p^m-1}\sM\to\gamma_{+}\sM,\quad (x_n)_{n}\mapsto \sum_{n}T_n^{(m)}(x_{n})
    \]
    is injective. 
\end{lem}
\begin{proof}
For an integer $n$ with $0\le n\le p^m-1$, set 
\[s(n)\coloneqq \sum_{i=0}^{m-1}n_i,\]
where $n=\sum_{i=0}^{m-1}n_ip^i$ is the $p$-adic expansion of $n$. 

Let $y=(x_n)_{n}\in\ker T_\bullet$. We prove that $x_{n}=0$ for all $n=0,\dots,p^m-1$ by induction on $s(n)$. 
By Lemma~\ref{lem: Mustata}, the index 
\[
(n,r_0,\dots,r_{m-1})=(0,0,\dots,0)
\]
is the only one that contributes to the coefficient of $\delta_0$ in $T_\bullet(y)$. Hence, since $T_\bullet(y)=0$, we obtain $x_{0}=0$.

Now fix a nonnegative integer $N$, and 
 suppose that 
 \[\text{$x_{b}=0$ for all $b\in\{0,\dots,p^m-1\}$ with $s(b)\leq N$.} \]
 Let $n\in\{0,\dots,p^m-1\}$ be such that  $s(n)=N+1$. The coefficient of $\delta_n$ in $T^{(m)}_n(x_n)$ is equal to $(-1)^nx_n$. Therefore, to show that $x_n=0$, it suffices to prove that 
 the coefficient of $\delta_n$ in 
 \[
 T_\bullet^{(m)}(y)-T_n^{(m)}(x_n)=\sum_{b\neq n}T_b^{(m)}(x_b)
 \]
 is zero. 
 
 Let  $b\neq n$. If $s(b)\le N$, then $x_b=0$ by the induction hypothesis, and hence $T_b^{(m)}(x_b)=0$. Suppose now that $s(b)\ge s(n)$, and write $b=\sum_{i=0}^{m-1}b_ip^i$
 for its $p$-adic expansion. For integers $r_0,\dots, r_{m-1}$ with $0\le r_i\le p^m-b_i-1$, the equality 
 \[
 b+\sum_{i=0}^{m-1}r_ip^i=n
 \]
 implies that $b_i+r_i=n_i$ for all $i$, since both of $\sum_{i=0}^{m-1}(b_i+r_i)p^i$ and $\sum_{i=0}^{m-1}n_ip^i$ give the $p$-adic expansion of $n$. 

Since $s(b)\ge s(n)$, these equalities imply  that $r_i=0$ for all $i$ and $b=n$, contradicting the assumption $b\neq n$. 

Thus all $x_{n}$ vanish, and $T_\bullet^{(m)}$ is injective.
\end{proof}

\begin{prop}[{cf. \cite[Proposition~6.1]{Mus09}}]\label{prop: image of L}
Let $M$ be an $\cO_X$-submodule of a $\cD_X$-module $\sM$. Then the morphism $T_\bullet^{(m)}$ induces an isomorphism of $\cD_X^{(m)}$-modules
\[
 \bigoplus_{n=0}^{p^m-1}\cD_X^{(m)}f^{n} M\xrightarrow{\cong} N_{f}^{(m)}(M). 
\]
Moreover, the direct summand $\cD_X^{(m)}f^nM$ corresponds to $N_{f}^{(m)}(M)_{\bar{n}}$, where $\bar{n}\in\Z/p^m\Z$ denotes the class of $n$. 
\end{prop}
\begin{proof}
Since $T_\bullet^{(m)}$ is injective by Lemma~\ref{lem: injectivity of L}, it suffices to determine its image.

The identity 
\[
(f-t)^{p^m-1}=\sum_{n=0}^{p^m-1}(-1)^{n}\binom{p^m-1}{n}f^{n}t^{p^m-n-1}
\]
 shows that for any $x\in M$,
\begin{equation*}\label{xd_0}
    (1\otimes x)\otimes\delta_0=\sum_{n=0}^{p^m-1}(-1)^{n}\binom{p^m-1}{n}f^{n}T_{n}^{(m)}(x). 
\end{equation*}
    By Lemma~\ref{lem: eigenspace}, 
 each $T_n^{(m)}(x)$ is an eigenvector for the action of $\sC^{(m)}$ with eigencharacter 
\[
\ev_{\bar{n}}\colon \sC^{(m)}\to\F_p. 
\]
 Moreover, the homomorphism 
\[
(\ev_{\bar{n}})_{0\le n\le p^m-1}\colon \sC^{(m)}\to \prod_{n=0}^{p^m-1}\F_p
\]
is an isomorphism. It follows that 
\[
\sC^{(m)}\cdot((1\otimes x)\otimes\delta_0)=\bigoplus_{n=0}^{p^m-1}
\F_p\cdot f^nT_n^{(m)}(x), 
\]
 where we  use that $\binom{p^m-1}{n}\not\equiv 0 \pmod p$. 

Finally, applying $\cD_X^{(m)}$ and letting $x$ vary in $M$, we obtain the desired description of the image, which proves the claim. 
\end{proof}

As a consequence, we obtain the following. 
\begin{cor}\label{cor: finite BSR vs infinite BSR}
    Let $n$ be an integer with $0\le n\le p^m-1$, and let $\bar{n}\in \Z/p^m\Z$ denote its class. Then $n\in\BSR^{(m)}(M,f)$ if and only if 
    \[
    N_{f}^{(m)}(fM)_{\bar{n}}\subsetneq N_{f}^{(m)}(M)_{\bar{n}}. 
    \]
\end{cor}

We now describe how the isomorphism in Proposition~\ref{prop: image of L} behaves as $m$ varies. 
Let $\Psi_m$ denote the composite morphism 
\[
\bigoplus_{n=0}^{p^m-1}\cD_X^{(m)}f^{n} M\xrightarrow{\cong} N_{f}^{(m)}(M)\hookrightarrow N_{f}^{(m+1)}(M)\xrightarrow{\cong} \bigoplus_{n=0}^{p^{m+1}-1}\cD_X^{(m+1)}f^{n} M. 
\]

\begin{prop}\label{prop: how L varies}
Let $\sM$ be a $\cD_X$-module, and let $M\subset \sM$ be an $\cO_X$-submodule. Let $n_1$ and $n_2$ be integers with $0\le n_1\le p^m-1$ and $0\le n_2\le p^{m+1}-1$. Then the composite morphism 
\[
\cD_X^{(m)}f^{n_1}M\hookrightarrow \bigoplus_{n=0}^{p^m-1}\cD_X^{(m)}f^{n} M\xrightarrow{\Psi_m}\bigoplus_{n=0}^{p^{m+1}-1}\cD_X^{(m+1)}f^{n}M\twoheadrightarrow \cD_X^{(m+1)}f^{n_2}M
\]
is zero unless $n_2-n_1\in p^m\Z$. 

If $n_2-n_1\in p^m\Z$, write $n_2=n_1+ip^m$ with $i=0,\dots,p-1$. Then the composite morphism is given by multiplication by $(-1)^i\binom{p-1}{i}f^{ip^m}$. 
\end{prop}
\begin{proof}
Let $n$ be an integer with $0\le n\le p^m-1$. We claim that (cf.~\cite[Remark~5.7]{Mus09})
\[
Q_{n}^{(m)}=\sum_{i=0}^{p-1}(-1)^i\binom{p-1}{i}f^{ip^m}Q^{(m+1)}_{n+ip^m}. 
\]
Indeed, this is equivalent to the identity 
\[
t^{p^m-n-1}(f-t)^{p^{m+1}-p^m}=\sum_{i=0}^{p-1}(-1)^i\binom{p-1}{i}f^{ip^m}t^{p^{m+1}-n-ip^m-1}, 
\]
which follows by expanding $(f-t)^{p^{m+1}-p^m}=(f^{p^m}-t^{p^m})^{p-1}$.  

The assertion then follows from the construction of $T_\bullet^{(m)}$.    
\end{proof}

The following result should be compared with Lemma~\ref{BSRnu}. For an integer $m\ge1$ and   $\alpha\in\Z_p$, let $\alpha_m$ denote the unique integer in $\{0,\dots,p^m-1\}$ such that $\alpha-\alpha_m\in p^m\Z_p$. 

\begin{prop}\label{prop: infinite BSR vs finite}
Let $\sM$ be a  $\cD_X$-module, and let $M\subset \sM$ be an $\cO_X$-submodule. 
Then the following statements hold:  
\begin{enumerate}
    \item The assignment $a\mapsto a_m$ induces a well-defined map 
    \[
    \BSR^{(m+1)}(M,f)\cap[0,p^{m+1})\to \BSR^{(m)}(M,f)\cap[0,p^{m}). 
    \]
    \item The assignment $\alpha\mapsto (\alpha_m)_m$ induces a well-defined injective map 
    \[
    \BSR(M,f)\to\varprojlim_{m\ge1}\Big(\BSR^{(m)}(M,f)\cap [0,p^m)\Big). 
    \]
    Further if $M$ is locally finitely generated as an $\cO_X$-module, then this map is bijective. 
\end{enumerate}
\end{prop}
\begin{proof}
    (1) Write $a=a_m+ip^m$ with $i\ge0$. If 
    \[
    \cD_X^{(m)}f^{a_m+1}M= \cD_X^{(m)}f^{a_m}M, 
    \]
    then $\cD_X^{(m)}f^{a+1}M= \cD_X^{(m)}f^{a}M$ since $\cD_X^{(m)}$ consists of $\cO_X^{p^m}$-linear maps. Applying $\cD_X^{(m+1)}$, we obtain $a\notin \BSR^{(m+1)}(M,f)$.

    (2) The well-definedness follows from Corollary~\ref{cor: finite BSR vs infinite BSR}. Injectivity follows from the uniqueness of the $p$-adic expansion. 

It remains to prove that the map is surjective, provided that $ M$ is locally finitely generated. Let $(a_m)_m$ be an element of the target, and let $\alpha$ denote its $p$-adic limit. It suffices to show that $\alpha\in \BSR(M,f)$. Suppose, to the contrary, that $\alpha\notin\BSR(M,f)$. Write 
\[
\Big(N_{f}(M)/N_{f}(fM)\Big)_\alpha\cong \varinjlim_{m\ge1}\Big(N_{f}^{(m)}(M)/N_{f}^{(m)}(fM)\Big)_{\overline{a_m}}, 
\]
    where $\overline{a_m}\in\Z/p^m\Z$ denotes the class of $a_m$. By Propositions~\ref{prop: image of L} and~\ref{prop: how L varies}, the colimit on the right-hand side can be identified with 
    \[
    \varinjlim_{m\ge1}(\cD_X^{(m)}f^{a_m}M)/(\cD_X^{(m)}f^{a_m+1}M), 
    \]
    where the transition morphisms are given by multiplication by $u_m\cdot f^{a_{m+1}-a_m}$ for some unit $u_m\in \F_p^\times$. 

    Since $\alpha\notin \BSR(M
    ,f)$, this colimit vanishes. As $M$ is locally finitely generated, it follows that for a sufficiently large $m$, the morphism 
    \[
(\cD_X^{(1)}f^{\alpha_1}M)/(\cD_X^{(1)}f^{\alpha_1+1}M)\to (\cD_X^{(m)}f^{\alpha_m}M)/(\cD_X^{(m)}f^{\alpha_m+1}M)
    \]
    is zero. Equivalently, 
    \[
    f^{a_m-a_1}\cdot\cD_X^{(1)}f^{a_1}M=\cD_X^{(1)}f^{a_m}M\subset \cD_X^{(m)}f^{a_m+1}M. 
    \]
    Applying $\cD_X^{(m)}$, we obtain $\cD_X^{(m)}f^{a_m}M= \cD_X^{(m)}f^{a_m+1}M$, which contradicts the assumption $a_m\in \BSR^{(m)}(M,f)$. 
    This  proves the surjectivity of the map. 
\end{proof}

\subsection{\texorpdfstring{The Functors $N_{f,\alpha}$}{The Functors Nfalpha}}\label{subsection: The functors Nfalpha}

Let $\cP$ be the category defined as follows.  An object of $\cP$ is a pair $(\sM,M)$ consisting of a  $\cD_X$-module $\sM$ and an $\cO_X$-submodule $M\subset \sM$. A morphism 
\[(\sM,M)\to (\sM',M')\]
is a morphism of $\cD_X$-modules $\sM\to \sM'$ that sends $M$ into $M'$.

Let $\alpha\in\Z_p$. In Subsection~\ref{subsection: Graph Construction for Differential Modules}, a $\cD_X$-module $N_{f}(M)_\alpha$ is constructed from an object  $(\sM,M)$ of $\cP$. The assignment
\[
(\sM,M)\mapsto N_{f}(M)_\alpha
\]
defines a functor from the category $\cP$ to the category of $\cD_X$-modules. In this subsection, we study this functor. 

Recall that  $t\sC=\sC t$ 
inside $D_{\F_p[t]}$ \cite[Lemma~1.2.2]{Bit18}. 
Moreover, we have $t\delta_0=f\delta_0$. Hence multiplication by $t$ induces a $\cD_X$-linear endomorphism 
\[t\colon N_f(M)\to N_f(M).\]
 We will use the following computation. 
\begin{prop}[{\cite[Proposition~4.5]{Mus09}}]\label{prop: action of t}
Let $E$ be a $D_{\F_p[t]}$-module. 
For $a\in \Z/p^m\Z$, we have 
\[
t\cdot E_{a}\subset E_{a-1}. 
\]
\end{prop}
\begin{proof}
Write $a=\sum_{i=0}^{m-1}a_ip^i$ with $a_i=0,\dots,p-1$. 
In loc.~cit, the eigenspace $E_a$ is denoted by  
\[
E_{a_0,\cdots,a_{m-1}}. 
\]
The assertion is therefore a restatement of  \cite[Proposition~4.5]{Mus09}. 
\end{proof}

\begin{prop}\label{prop: t sends a to a-1}
    Let $(\sM,M)$ be an object of $\cP$, and let $\alpha\in\Z_p$. Then the following hold: 
    \begin{enumerate}
        \item The operator $t$ induces a commutative diagram 
        \[
        \xymatrix{
        N_f(M)\ar[d]\ar[r]^-t&N_f(M)\ar[d]\\
         N_f(M)_{\alpha+1}\ar[r]^-t& N_f(M)_{\alpha}, 
        }
        \]
        where the vertical arrows are the canonical surjections. 
        \item  The image of the induced morphism 
        \[
        t\colon N_f(M)_{\alpha+1}\to N_f(M)_{\alpha}
        \]
        coincides with the submodule $N_f(fM)_{\alpha}\subset N_f(M)_{\alpha}$. Moreover, if  $\alpha\neq-1$, this morphism is injective. 
    
        \item The morphism in~(2) is surjective if and only if $\alpha\notin\BSR(M,f)$. 
    \end{enumerate}
\end{prop}
\begin{proof}
Part~(1) follows from Proposition~\ref{prop: action of t}. 

    We now prove~(2). 
    For each integer $m\ge1$, let $\beta_m$ be  the unique integer in $\{0,\dots,p^m-1\}$ such that $\alpha+1-\beta_m\in p^m\Z_p$. We denote by $\overline{\beta_m}$ the image of $\beta_m$ in $\Z/p^m\Z$. 
    It suffices to identify the morphism  
    \begin{equation}\label{equation: Nfb to Nfb-1}
         t\colon N_f^{(m)}(M)_{\overline{\beta_m}}\to N_f^{(m)}(M)_{\overline{\beta_m}-1}
    \end{equation}
     for all sufficiently large $m$. 

First assume that $\alpha+1\neq0$. Then  $\beta_m>0$ for all but finitely many $m$. For such $m$, we have $tQ_{\beta_m}^{(m)}=Q_{\beta_m-1}^{(m)}$. 
    Therefore, by Proposition~\ref{prop: image of L}, the above morphism is injective, and its image coincides with $N_f^{(m)}(fM)_{\overline{\beta_m}-1}$. This proves the case $\alpha\neq-1$. 

    Now assume that $\alpha=-1$. In this case, $\beta_m=0$ for all $m$. A direct computation shows that 
    \[
    tQ_0^{(m)}=f^{p^m}Q_{p^m-1}^{(m)}. 
    \]
    Hence, it follows that 
    the morphism \eqref{equation: Nfb to Nfb-1} can be identified with  
    \[
    f^{p^m}\colon \cD_X^{(m)}M\to \cD_X^{(m)}f^{p^m-1}M. 
    \]
    Since $\cD_X^{(m)}$ consists of $\cO_X^{p^m}$-linear morphisms, the image of this morphism is equal to $\cD_X^{(m)}f^{p^m}M=N_f^{(m)}(fM)_{p^m-1}$, which proves the case $\alpha=-1$.

    Part~(3) follows from~(2) and the definition of $\BSR(M,f)$. 
\end{proof}

In what follows, we compute the morphism $t\colon N_f(M)_{\alpha+1}\to N_f(M)_{\alpha}$ in detail. We first treat the case where $\alpha$ is an integer. 

\subsubsection{The Integral Case}
\begin{prop}\label{prop: computation of t integral case}
    Let $(\sM,M)$ be an object of $\cP$, and let $n$ be an integer. Let $\sM_f\coloneqq \sM\otimes_{\cO_X}\cO_X[f^{-1}]$, and 
    let $\xi\colon \sM\to\sM_f$ denote the natural morphism. 
    Then the morphism $t\colon N_f(M)_{n+1}\to N_f(M)_{n}$ is described as follows: 
    \begin{enumerate}
\item When $n\ge0$, it can be identified with the natural inclusion $\cD_Xf^{n+1}M\hookrightarrow \cD_Xf^{n}M$. 

\item When $n=-1$, it can be identified with $\cD_XM\to \cD_Xf^{-1}\cdot\xi(M)$ induced by $\xi$. 

\item When $n\le -2$, it can be identified with the natural inclusion $\cD_Xf^{n+1}\cdot\xi(M)\hookrightarrow \cD_Xf^{n}\cdot\xi(M)$. 
    \end{enumerate}
\end{prop}
\begin{proof}
We describe the $\cD_X$-module $N_f(M)_{n}$ as follows. 

    First assume that $n\ge0$. By Proposition~\ref{prop: image of L}, for each positive integer $m$ with $p^m-n>0$, we have 
    \[
    N_{f}^{(m)}(M)_{n}\cong \cD_X^{(m)}f^{n}M. 
    \]
    Moreover, by Proposition~\ref{prop: how L varies}, the transition morphism  $N_{f}^{(m)}(M)_{n}\to  N_{f}^{(m+1)}(M)_{n}$ corresponds to the inclusion $\cD_X^{(m)}f^{n}M\hookrightarrow \cD_X^{(m+1)}f^nM$. Passing to the colimit over $m$, we obtain $N_f(M)_n\cong \cD_Xf^nM$. 

    Next assume that $n<0$, and write $n=-a$. Then, for every $m>0$, we have 
    \[n\equiv p^m-(a-1)-1\mod p^m.\]
     If $p^m>a-1$, Proposition~\ref{prop: image of L} yields 
    \[
    N_f^{(m)}(M)_n\cong \cD_X^{(m)}f^{p^m-(a-1)-1}M. 
    \]
 Moreover, 
by Proposition~\ref{prop: how L varies}, the transition morphism  $N_{f}^{(m)}(M)_{n}\to N_{f}^{(m+1)}(M)_{n}$ corresponds to 
\begin{equation}\label{equation: f qm+1}
    f^{p^{m+1}-p^m}\colon \cD_X^{(m)}f^{p^m-(a-1)-1}M\to \cD_X^{(m+1)}f^{p^{m+1}-(a-1)-1}M. 
\end{equation}
Since these transition morphisms are given by multiplication by positive powers of $f$, the colimit is $f$-torsion free. Thus we may replace $\sM$ by $\sM_f$. In this case, we have a commutative diagram 
\[
\xymatrix{
 \cD_X^{(m)}f^{p^m-(a-1)-1}M\ar[r]^-{\eqref{equation: f qm+1}} \ar[d]_-{f^{-p^m}}&\cD_X^{(m+1)}f^{p^{m+1}-(a-1)-1}M\ar[d]^-{f^{-p^{m+1}}}\\
 \cD_X^{(m)}f^{-a}M\ar[r]&\cD_X^{(m+1)}f^{-a}M, 
}
\]
where the bottom horizontal arrow is the natural inclusion. 
Passing to the colimit over $m$, we obtain $N_f(M)_n\cong \cD_Xf^{-a}\cdot\xi(M)$. 

The identifications in~(1) and~(3) follow from the above description of $N_f(M)_n$. For~(2), we use 
 the identity $tQ_0^{(m)}=f^{p^m}Q_{p^m-1}^{(m)}$. This completes the proof. 
\end{proof}

 \subsubsection{The Non-Integral Case}\label{subsubsection: Non-Integral Case}
Let $\alpha\in\Z_p\setminus\Z$. 
Inspired by \cite{BMS09} and \cite{BQG25}, we consider the following $\cD_X$-module: 
\[
\sL_{f,\alpha}\coloneqq \cO_X[f^{-1}]\cdot f^\alpha,
\]
where $f^\alpha$ is a formal symbol such that the assignment $x\mapsto x\cdot f^\alpha$ defines an isomorphism  $\cO_X[f^{-1}]\xrightarrow{\cong}\mathscr{L}_{f,\alpha}$. 

The algebra $\cD_X$ acts on $\sL_{f,\alpha}$ as follows. 
For each $P\in\cD_X$, choose an integer $m\ge 0$ such that $P\in\cD_X^{(m)}$, 
and an integer $n$ such that $\alpha-n\in p^m\Z_p$. We then define
\[
P(x\cdot f^\alpha)=f^{-n}P(xf^{n})\cdot f^\alpha.
\]
This definition is independent of the choices of $m$ and $n$. 

Let $n$ be an integer. Then the morphism 
\begin{equation}\label{equation: a a+n}
    \mathscr{L}_{f,\alpha}\to \sL_{f,\alpha+n},\quad a\cdot f^\alpha\mapsto af^{-n}\cdot f^{\alpha+n}
\end{equation}
defines an isomorphism of $\cD_X$-modules.

Write 
\[
\alpha=\sum_{i\ge0}a_ip^i\quad(a_i=0,\dots,p-1). 
\]
For each integer $m\ge1$, set $\alpha_m=\sum_{i=0}^{m-1}a_ip^i$, and set $\alpha_0=0$. 
\begin{lem}\label{lem: system correponding to K}
Under the equivalence in Proposition~\ref{prop: Frobenius descent for Dmodule}, the $\cD_X$-module $\mathscr{L}_{f,\alpha}$ corresponds to the system $(K_{(m)})_{m\ge0}$ defined by 
\[
K_{(m)}=f^{-\alpha_m}\cO_X[f^{-1}]^{p^m}\cdot f^\alpha. 
\]
The transition morphisms are given by the natural inclusions $K_{(m+1)}\hookrightarrow K_{(m)}$. 
\end{lem}
\begin{proof}
For each $m\ge0$, set 
\[
K_{(m)}'\coloneqq \{y\in \mathscr{L}_{f,\alpha}\mid P(ay)=P(a)y\quad(\forall a\in \cO_X,\forall P\in\cD_X^{(m)})\}. 
\]
By Proposition~\ref{prop: Frobenius descent for Dmodule}, $\mathscr{L}_{f,\alpha}$ corresponds to the system $(K_{(m)}')_m$. Thus, it suffices to show that $K_{(m)}=K_{(m)}'$. 

The inclusion $K_{(m)}\subset K_{(m)}'$ follows from a direct computation. Hence, it remains to show that the induced inclusion $K_{(m)}\hookrightarrow K_{(m)}'$ is an isomorphism. Since $\cO_X$ is faithfully flat over $\cO_X^{p^m}$, it suffices to check this after base change to $\cO_X$. By Lemma~\ref{lem: Frobenius descendent}, 
we have $\cO_X\otimes_{\cO_X^{p^m}}K_{(m)}'=\mathscr{L}_{f,\alpha}$. 
On the other hand, we have 
\[\cO_X\otimes_{\cO_X^{p^m}}K_{(m)}=\mathscr{L}_{f,\alpha},\] 
since 
\[
\cO_X\otimes_{\cO_X^{p^m}}\cO_X[f^{-1}]^{p^m}=\cO_X[f^{-1}]. 
\]
This proves the claim. 
\end{proof}

For a $\cD_X$-module $\sM$, we set  
\[\sM_{f,\alpha}\coloneqq \sM\otimes_{\cO_X}\mathscr{L}_{f,\alpha}, 
\]
equipped with the $\cD_X$-module structure given by Construction~\ref{construction: pullback pushforward tensor product of Dmodules}. For notational simplicity, we  write $x\otimes (a\cdot f^\alpha)$ as $ax\cdot f^\alpha$. Then there is a natural isomorphism of $\cO_X$-modules 
\[
\sM_{f,\alpha}\cong \sM_f\cdot f^\alpha, \quad x\otimes (a\cdot f^\alpha)\mapsto ax\cdot f^\alpha, 
\]
where $\sM_f=\sM[f^{-1}]$. 
\begin{lem}\label{lem: Lfalpha}
    Let $\sM$ be a $\cD_X$-module, and let $\xi\colon\sM\to\sM_f$ denote the natural morphism. 
    Let $M\subset \sM$ be an $\cO_X$-submodule. Then the following hold: 
    \begin{enumerate}
        \item Let $m\ge 0$ be an integer, and let $n\in\Z$ be such that $\alpha-n\in p^m\Z_p$. Then 
         the morphism 
        \[
         \sM\to \sM_f\cdot f^\alpha,\quad x\mapsto f^{-n}x\cdot f^\alpha
        \]
        is $\cD_X^{(m)}$-linear. 
\item Consider the inductive system \[(L_m)_{m\ge0}\coloneqq(\cD_X^{(m)}f^{\alpha_m}M)_{m\ge 0}\] with transition morphisms $L_m\to L_{m+1}$ given by  multiplication by $f^{a_mp^m}$. Then the morphisms 
\[
f^{-\alpha_m}\colon L_m\to \sM_f\cdot f^\alpha
\]
are $\cD_X^{(m)}$-linear and compatible with the transition morphisms. The induced morphism 
\[
\varphi\colon \varinjlim_mL_m\to \sM_f\cdot f^\alpha 
\]
is injective, and its image is equal to  $\cD_X(\xi(M)\cdot f^\alpha)$. 
    \end{enumerate}
\end{lem}
\begin{proof}
    (1) Let $M_{(m)}\subset \sM$ denote the $\cO_X^{p^m}$-submodule defined in Lemma~\ref{lem: Frobenius descendent}. 
Since $\alpha-n\in p^m\Z_p$, the morphism 
\[
M_{(m)}\to M_{(m)}\otimes_{\cO_X^{p^m}}(f^{-\alpha_m}\cO_X[f^{-1}]^{p^m}\cdot f^\alpha),\quad x\mapsto x\otimes f^{-n}\cdot f^\alpha
\]
is well-defined. 
Extending scalars to $\cO_X$, and using Lemmas~\ref{lem: Frobenius descendent} and~\ref{lem: system correponding to K}, we obtain a $\cD_X^{(m)}$-linear morphism 
\[\sM\to \sM_f\cdot f^\alpha.\]
One  checks that this morphism is given by multiplication by $f^{-n}$. 

\medskip 

\noindent
(2) The first assertion follows from~(1). Observe that for every $m\ge0$, the image of $L_0\to L_m$ generates $L_m$ as a $\cD_X^{(m)}$-module. It follows that the image of $\varphi$ is equal to $\cD_X(\xi(M)\cdot f^\alpha)$. 

We show that $\varphi$ is injective. Since the transition morphisms of $(L_m)_{m\ge0}$ are given by multiplication by positive powers of $f$, the colimit is $f$-torsion free. Hence, we may replace $\sM$ by $\sM_f$ and $M$ by $\xi(M)$. In this case, the injectivity is clear. 
\end{proof}

%For every $\beta\in \Z_p$ and every integer $m\ge0$, we write $\beta_m$ for the unique integer in $\{0,\dots, p^m-1\}$ such that $\beta-\beta_m\in p^m\Z_p$. 
\begin{prop}\label{prop: non-integral case}
Let $(\sM,M)$ be an object of $\cP$. Let 
$\xi\colon\sM\to\sM_f$ denote the natural morphism. 
 Let $\alpha\in\Z_p\setminus\Z$. 
 Then the following hold: 
\begin{enumerate}
    
\item The $\cD_X$-module $N_f(M)_{\alpha}$ is isomorphic to the colimit of the inductive system $(\cD_X^{(m)}f^{\alpha_m}M)_{m\ge0}$ whose  transition morphism $\cD_X^{(m)}f^{\alpha_m}M\to \cD_X^{(m+1)}f^{\alpha_{m+1}}M$ is given by  multiplication by $f^{a_mp^m}$. 

\item The morphism $t\colon N_f(M)_{\alpha+1}\to N_f(M)_{\alpha}$ is identified with the natural inclusion 
\[\cD_X(f\xi(M)\cdot f^\alpha)\hookrightarrow \cD_X(\xi(M)\cdot f^\alpha). 
\]
\end{enumerate}
\end{prop}
\begin{proof}
We first prove~(1). 
Since $\alpha\neq0$, there exists an integer $m_0\ge0$ such that $\alpha_{m_0}>0$. 
Fix such $m_0$, and consider the inductive system
\[
(\cD_X^{(m)}f^{\alpha_m}M)_{m\ge m_0}
\]
where the transition morphisms are given by   multiplication by $f^{a_mp^m}$. 
For every $m\ge m_0$, we define an element $u_m\in\F_p^\times$ by 
\[u_m\coloneqq \prod_{i=m_0}^{m-1}(-1)^{a_i}\binom{p-1}{a_i}, 
\]
with the convention that $u_{m_0}=1$.  Then, by Proposition~\ref{prop: how L varies},  for every $m\ge m_0$ the diagram 
\[
\xymatrix{
\cD_X^{(m)}f^{\alpha_m}M\ar[d]_-{f^{a_mp^m}}\ar[rr]^-{u_mT^{(m)}_{\alpha_m}}&&N_f^{(m)}(M)_{\overline{\alpha_m}}\ar[d]\\
\cD_X^{(m+1)}f^{\alpha_{m+1}}M\ar[rr]^-{u_{m+1}T^{(m+1)}_{\alpha_{m+1}}}&&N_f^{(m+1)}(M)_{\overline{\alpha_{m+1}}}
}
\]
is commutative, where $\overline{\alpha_m}$ and $\overline{\alpha_{m+1}}$ denote the classes of $\alpha$ in $\Z/p^m\Z$ and in $\Z/p^{m+1}\Z$, respectively. Since the horizontal arrows are isomorphisms by Proposition~\ref{prop: image of L}, assertion~(1) follows by taking  colimits over $m$. 

We now prove~(2). By~(1) and Lemma~\ref{lem: Lfalpha}(2), $N_f(M)_{\alpha}$ is naturally isomorphic to $\cD_X(\xi(M)\cdot f^\alpha)$. Then the claim follows from Proposition~\ref{prop: t sends a to a-1}(2).

\end{proof}

\section{The Case of Unit Frobenius Modules}\label{section: The Case of Unit Frobenius Modules}

It is well known that a unit $F$-module naturally carries both  a Cartier module structure and a $\cD$-module structure. Consequently, it is associated with two types of invariants: 
the $\nu$-invariants and the Bernstein--Sato roots. In this section, we study the relationship between these invariants. 

Let $X$ be a Noetherian regular $F$-finite  $\F_p$-scheme. Fix a positive integer $e$ and set $q=p^e$.

\subsection{Frobenius Modules}
In this subsection, we review  unit Frobenius modules. 
Let $\cO_X[F]$ be the $\cO_X$-algebra defined by 
\[
\cO_X[F]\coloneqq \bigoplus_{e'\ge0}\cO_X\cdot F^{e'}. 
\]
As a left $\cO_X$-module, it is free with basis $\{F^{e'}\}_{e'\ge0}$, and its multiplication is determined by the relation 
\[
F\cdot a=a^pF\qquad(a\in\cO_X). 
\]
We denote by $\cO_X[F^e]\subset\cO_X[F]$ the $\cO_X$-subalgebra generated by $F^e$. 

In what follows, an $\cO_X[F^e]$-module always means a left module. An $\cO_X[F^e]$-module is also referred to as an $F^e$-module on $X$. 

Let $\sM$ be an $\cO_X$-module. Then giving $\sM$ an $\cO_X[F^e]$-module structure that is compatible with its $\cO_X$-module structure is equivalent to giving an $\cO_X$-linear morphism $F^\ast_{q} \sM\to \sM$.

\begin{defn}\label{defn: lfgu Femodule}
    Let $\sM$ be an $\cO_X$-quasi-coherent $\cO_X[F^e]$-module. 
    \begin{enumerate}
        \item We say that $\sM$ is \emph{unit} if the induced morphism  
        \[
        F^\ast_{q} \sM\to \sM
        \] is an isomorphism. 
        \item We say that $\sM$ is \emph{locally finitely generated unit} (lfgu for short) if it is unit and is locally finitely generated as an $\cO_X[F^e]$-module. 
    \end{enumerate}
\end{defn}

Let $\mathscr{M}$ be a unit $\cO_X[F^e]$-module. 
We write 
\[\theta_{\sM}\colon \mathscr{M}\xrightarrow{\cong} F_q^\ast  \mathscr{M}
\]
for the inverse of the isomorphism $F_q^\ast \mathscr{M}\to \mathscr{M}$. 
The morphism $\theta_\sM$ induces a sequence of isomorphisms
\[
\mathscr{M}\xrightarrow{\theta_\sM} F_q^\ast  \mathscr{M}\xrightarrow{F_q^\ast \theta_\sM} F_q^{2\ast}  \mathscr{M}\xrightarrow{F_q^{2\ast}\theta_\sM} \cdots. 
\]
In the sequel, we usually identify $\sM$ with the colimit of this system. For example, a submodule $N\subset F_q^\ast\sM$ is identified with the submodule $\theta_\sM^{-1}(N)\subset\sM$. 

We recall the following definition. 
\begin{defn}
    Let $\sM$ be a unit $\cO_X[F^e]$-module. A \emph{root} of $\sM$ is a coherent $\cO_X$-submodule $M\subset\sM$ satisfying the following properties: 
    \begin{enumerate}
        \item   $\theta_\sM(M)\subset F_q^\ast M. $
        \item   $\sM=\bigcup_{m\ge0}F_q^{m\ast} M.$
    \end{enumerate}
\end{defn}

It is known that a unit $F^e$-module admits a root if and only if it is lfgu. The following proposition provides a more precise form of the ``if'' direction. 
\begin{prop}\label{prop: existence of minimal root}
    Let $\sM$ be an lfgu $\cO_X[F^e]$-module. Then   $\sM$ admits a root. Moreover, it admits a minimal root. Here, a root is said to be \emph{minimal} if it is contained in every root  of $\mathscr{M}$. 
\end{prop}
\begin{proof}
It suffices to prove the existence of a minimal root, which is established in \cite[Corollary~2.26]{Blickle08}. 
\end{proof}

We now recall a more precise statement for the ``only if '' direction.  Consider a quasi-coherent $\cO_X$-module $M$ equipped with an $\cO_X$-linear morphism $\theta\colon M\to F_q^\ast M$. 
Let $\sM$ denote the colimit of the diagram 
\[
M\xrightarrow{\theta}F_q^\ast M\xrightarrow{F_q^\ast\theta}F_q^{2\ast}M\xrightarrow{F_q^{2\ast}\theta}\cdots
\]
Since $F_q^\ast$ commutes with filtered colimits, $\theta$ induces an $\cO_X$-linear isomorphism  
\[
\theta_\sM\colon \sM\xrightarrow{\cong} F_q^\ast\sM. 
\]
Then the inverse of $\theta_\sM$ endows $\sM$ with the structure of an $F^e$-module. By construction, this $F^e$=module is unit. Following \cite{BL19}, we refer to this unit $F^e$-module as the \emph{unitalization of $(M,\theta)$.}
\begin{prop}\label{prop: a root generates lfgu Fmodule}
Let $M$ be a coherent $\cO_X$-module  equipped with an $\cO_X$-linear morphism $\theta\colon M\to F_q^\ast M$. 
Then the unitalization $\sM$ of $(M,\theta)$ is lfgu.  
\end{prop}
\begin{proof}
It suffices to show that $\sM$ is locally finitely generated as an $\cO_X[F^e]$-module. Since the assertion is local, we may assume that $X$ is affine. In this case, 
 the claim is proved in \cite[Corollary~11.2.6]{BL19}. 
\end{proof}

Suppose that $X={\rm Spec}(R)$ is an affine scheme. Since $\cO_X[F^e]$ is a quasi-coherent $\cO_X$-algebra, the global section functor induces an equivalence of categories between the category of $\cO_X$-quasi-coherent $\cO_X[F^e]$-modules and that of $R[F^e]$-modules. One readily checks that, under this equivalence, unit $\cO_X[F^e]$-modules correspond to unit $R[F^e]$-modules. In the next lemma, we show that lfgu $\cO_X[F^e]$-modules correspond to fgu $R[F^e]$-modules. 
\begin{lem}\label{lem: lfgu is fgu}
Let $\sM$ be a unit $\cO_X[F^e]$-module with associated unit $R[F^e]$-module $M$. Then the following conditions are equivalent: 
\begin{enumerate}
    \item The $\cO_X[F^e]$-module $\sM$ is lfgu. 
    \item The $R[F^e]$-module $M$ is fgu, where ``fgu" stands for ``finitely generated unit" in the sense of \cite[Definition~11.1.3]{BL19}.   
\end{enumerate}
\end{lem}
\begin{proof}
    The implication (2)$\Rightarrow$(1) is clear. 
    
    For the converse, assume that $\sM$ is lfgu. 
    Then $M$ is a unit $R[F^e]$-module. It remains to show that $M$ is finitely generated over $R[F^e]$. Take a finite affine open covering 
    \[
    X=\bigcup_{i=1}^nU_i=\bigcup_{i=1}^n{\rm Spec}(R_{f_i})\qquad (f_i\in R)
    \]
    such that for each $i$, $\sM|_{U_i}$ is generated by finitely many sections $x_{i1},\dots,x_{im_i}\in\Gamma(U_i,\sM)$. Replacing $x_{ij}$ by $f_i^Nx_{ij}$ for a sufficiently large integer $N$ if necessary, we may assume that every $x_{ij}$ lifts to a section $y_{ij}\in \Gamma(X,\sM)=M$. Then $M$ is generated by the elements $y_{ij}$ as an $R[F^e]$-module. 
\end{proof}

%Note that $\cO_X[F^e]$ is a quasi-coherent $\cO_X$-algebra. Therefore, when $X$ is an affine scheme ${\rm Spec}(R)$,  the global section functor induces an equivalence of categories between the category of $\cO_X$-quasi-coherent $\cO_X[F^e]$-modules and that of $R[F^e]$-modules. Similarly to the scheme case, we define the notion of unit $R[F^e]$-modules. 

%In the affine case, the notion of lfgu can be replaced by ``fgu''. 
%\begin{lem}\label{lem: lfgu is fgu}
%Assume that $X={\rm Spec}(R)$. Let $\sM$ be a unit $\cO_X[F^e]$-module with associated unit $R[F^e]$-module $M$. Then the following conditions are equivalent: 
%\begin{enumerate}
 %   \item The module $\sM$ is lfgu. 
 %   \item The module $M$ is finitely generated as an $R[F^e]$-module. 
%\end{enumerate}
%\end{lem}
%\begin{proof}
 %   The implication (2)$\Rightarrow$(1) is clear. 
    
%    We prove the reverse implication. Assume that $\sM$ is lfgu. Then, by Proposition~\ref{prop: existence of minimal root}(1), it admits a root $\sM'$.  Then $M'\coloneqq\Gamma(X,\sM')$ is a finitely generated $R$-submodule of $M$ such that the unitalization 
%    \[
%    M'^u=\colim(M'\to F_q^\ast M'\to\cdots)
%    \]
%    is isomorphic to $M$. By \cite[Corollary~11.2.6]{BL19}, condition (2) is satisfied. 
%\end{proof}

We recall some functorialities of $F^e$-modules.

\begin{lem}\label{lem: functoriality for Fmodule}
Let $Y$ and $Z$ be  Noetherian regular $F$-finite  $\F_p$-schemes. 
\begin{enumerate}
    \item Let $f\colon Y\to Z$ be a morphism of schemes. If $\mathscr{N}$ is a unit   $\cO_Z[F^e]$-module, then $f^\ast\mathscr{N}$ naturally becomes a unit $\cO_Y[F^e]$-module. If $\sN$ is lfgu, then so is $f^\ast\sN$. 
    \item Let 
    $j\colon V\to Y$ be an affine open immersion. If $\sM$ is a unit $\cO_V[F^e]$-module, then 
     $j_\ast\sM$ naturally becomes a unit $\cO_Y[F^e]$-module. If $\sM$ is lfgu, then so is $j_\ast\sM$. 
     \item Let $\sM$ and $\mathscr{N}$ be unit $\cO_Y[F^e]$-modules. Then $\sM\otimes_{\cO_Y}\mathscr{N}$ naturally becomes a unit $\cO_Y[F^e]$-module. If $\sM$ and $\sN$ are lfgu, then so is $\sM\otimes_{\cO_Y}\mathscr{N}$. 
\end{enumerate}
\end{lem}
In the situation of~(2), we denote by $j_+\sM$ the unit $F^e$-module $j_\ast\sM$. 
\begin{proof}
(1) 
The isomorphism $ F_{Z,q}^\ast\mathscr{N}\xrightarrow{\cong}  \mathscr{N}$ induces an isomorphism 
\[
f^\ast F_{Z,q}^\ast\mathscr{N}\xrightarrow{\cong} f^\ast \mathscr{N}. 
\]
Applying the canonical isomorphism $f^\ast F_{Z,q}^\ast\cong F_{Y,q}^\ast f^\ast$ to the source, we obtain an isomorphism 
\[
F_{Y,q}^\ast f^\ast\mathscr{N}\xrightarrow{\cong} f^\ast \mathscr{N}, 
\]
which endows $f^\ast\mathscr{N}$ with a unit $F^e$-module structure. 

To show that $f^\ast \mathscr{N}$ is lfgu when $\mathscr{N}$ is lfgu, we may reduce to the case where $Y$ and $Z$ are affine.  
This case is proved in \cite[Proposition~2.9(a)]{Lyubeznik-Fmod}.\footnote{The notion of $F^e$-finite modules in \cite{Lyubeznik-Fmod} is equivalent to that of fgu $F^e$-modules by \cite[Corollary~11.2.12]{BL19}. 
Moreover, the functor $f^\ast$ in our notation is denoted by $f_\ast$ in  loc.~cit. (see \cite[Definition-Proposition 1.3 (a)]{Lyubeznik-Fmod}).}  

(2) Note that the canonical morphism $F_{Y,q}^\ast j_\ast\to j_\ast F_{V,q}^\ast$ is an isomorphism. Hence, the isomorphism $F_{V,q}^\ast\sM\xrightarrow{\cong}\sM$ induces an isomorphism 
\[
F_{Y,q}^\ast j_\ast\sM\xrightarrow{\cong} j_\ast\sM. 
\]
Thus, $j_\ast\sM$ is a unit $\cO_Y[F^e]$-module. To show that it is lfgu when $\sM$ is lfgu, we may assume that $Y$ is affine. In this case the assertion is proved in \cite[Proposition~2.9(b)]{Lyubeznik-Fmod}. 

(3) The unit $F^e$-module structure on $\sM\otimes_{\cO_X}\sN$ is given by 
\[
F_q^\ast(\sM\otimes_{\cO_Y}\mathscr{N})\cong F_q^\ast\sM\otimes_{\cO_Y}F_q^\ast\mathscr{N}\xrightarrow{\cong}\sM\otimes\mathscr{N}, 
\]
where the first isomorphism is the canonical one, and the second is the tensor product of the respective isomorphisms $F_q^\ast\sM\to\sM$ and $F_q^\ast\mathscr{N}\to\mathscr{N}$. 

Suppose now that $\sM$ and $\sN$ are lfgu. Let $M\subset\sM$ and $N\subset \mathscr{N}$ be respective roots. Then the morphisms $M\to F_q^\ast M$ and $N\to F_q^\ast N$ induce a morphism 
\[\theta\colon M\otimes_{\cO_X}N\to F_q^\ast( M\otimes_{\cO_X}N).\]
One easily checks that $\sM\otimes_{\cO_Y}\mathscr{N}$ is the unitalization of $(M\otimes_{\cO_X}N,\theta)$. 
 Hence, by Proposition~\ref{prop: a root generates lfgu Fmodule}, $\sM\otimes_{\cO_Y}\mathscr{N}$ is lfgu. 
\end{proof}

\begin{lem}\label{lem: minimal root restrict to minimal root}
Let $\sM$ be an lfgu $F^e$-module on $X$, and let $M$ denote its minimal root. Let $U\subset X$ be an open subscheme. Then the restriction $M|_U$ is the minimal root of $\sM|_U$. 
\end{lem}
\begin{proof}
Let $j\colon U\hookrightarrow X$ denote the inclusion. Clearly, $j^\ast M$ is a root of $j^\ast\sM$. We show that this is minimal. Let $M'$ denote the minimal root of $j^\ast\sM$. Let $\xi\colon \sM\to j_+j^\ast\sM$ be the canonical morphism. Since $j_\ast$ commutes with filtered colimits and preserves quasi-coherence, the intersection 
\[M\cap \xi^{-1}(j_\ast M')\]
is a root of $\sM$. By the minimality of $M$, it follows that $\xi(M)\subset j_\ast M'$. Restricting to $U$, we obtain 
\[
j^\ast M\subset M'. 
\]
Since $M'$ is minimal, it follows that $j^\ast M=M'$. The assertion follows.  
\end{proof}

We recall that a unit $\cO_X[F^e]$-module naturally carries a $\cD_X$-module structure. 
\begin{const}\label{construction: Dmodule str}
    Let $\sM$ be a unit $\cO_X[F^e]$-module. We construct a ring homomorphism 
\begin{equation}\label{equation: map D to End in Cartier}
    \cD_X\to\mathcal{E}nd_{\F_p}(\sM)
\end{equation}
as follows. Let $P\in \cD_X$. By Proposition~\ref{prop: D=cup De}, there exists an integer $n\ge0$ such that $P$ is $\cO_X^{q^n}$-linear. On the other hand, the $F^{e}$-structure on $\sM$ induces an isomorphism 
\[
\varphi\colon F_q^{n\ast}\sM\xrightarrow{\cong} \sM. 
\]
 We define $\tilde{P}$ to be the composite 
\[
\sM\xrightarrow{\varphi^{-1}}F_q^{n\ast}\sM\xrightarrow{P\otimes {\rm id}_\sM}F_q^{n\ast}\sM\xrightarrow{\varphi}\sM. 
\]
One checks easily that $\tilde{P}$  does not depend on the choice of $n$, and that $P\mapsto \tilde{P}$ defines a ring homomorphism. 

This completes the construction of \eqref{equation: map D to End in Cartier}, and endows $\sM$ with a $\cD_X$-module structure. We refer to this as the \emph{induced $\cD_X$-module structure}. 
\end{const}

\subsection{Unit Frobenius Modules and Cartier Modules}\label{subsection: Unit Frobenius Modules and Cartier Modules}

 Here, we recall the fact that a unit $F^e$-module naturally admits the structure of a $q$-Cartier module. 
For every integer $m\ge0$, set 
\[
\mathcal{C}_X^{(m)}\coloneqq\mathcal{H}om_{\cO_X}((F_q^m)_\ast\cO_X,\cO_X). 
\]
We regard $\mathcal{C}_X^{(m)}$ as an $\mathcal{O}_X$-module via 
\begin{equation}\label{equation: module structure of CXm}
    a\cdot\phi\coloneqq \phi(a-), 
\end{equation}
where we use the identification $\cO_X\cong (F_q^m)_{\ast}\cO_X$ induced by the fact that $F_q^m$ is the identity on the underlying topological space. 
 
We first recall the following result, which is a special case of \cite[III, Theorem~6.7]{Har66}. 
\begin{prop}\label{prop: C to theta first version}For $\cO_X$-modules $M$ and $N$, there exists a natural bijection  
\[
{\rm Hom}_{\cO_X}((F_q^m)_\ast M,N)\xrightarrow{\cong} {\rm Hom}_{\cO_X}( M,\mathcal{C}_X^{(m)}\otimes_{\cO_X}F_q^{m\ast}N). 
\] 
\end{prop}
\begin{proof}

For the reader's convenience, we recall the construction of the bijection. Let 
$C\colon (F_q^m)_\ast M\to N$ be an $\cO_X$-linear morphism. 
The morphism $C$ induces a morphism 
\[\tilde{C}\colon M\to \mathcal{H}om_{\cO_X}((F_q^m)_\ast\cO_X, N),\quad \tilde{C}(x)(a):=C(ax). 
\] 
Here, we identify $(F_q^m)_\ast M$ with $M$ by using the fact that $F_q^m$ induces the identity on the underlying space. 

Since $X$ is $F$-finite and regular, $(F_q^m)_\ast\cO_X$ is a locally free $\cO_X$-module of finite rank. 
 Hence, the canonical morphism  
\begin{equation*}
    \mathcal{C}_X^{(m)}\otimes_{\cO_X}F_q^{m\ast} N\to 
\mathcal{H}om_{\cO_X}((F_q^m)_\ast\cO_X, N),\quad \phi\otimes(1\otimes n)\mapsto [a\mapsto \phi(a)n]
\end{equation*}
is an isomorphism, where the $\cO_X$-module structure of $\mathcal{C}_X^{(m)}$ is given by~\eqref{equation: module structure of CXm}. 
 Composing its inverse with $\tilde{C}$, we obtain a morphism of abelian sheaves 
\[
\theta_C'\colon M\to \mathcal{C}_X^{(m)}\otimes_{\cO_X} F_q^{m\ast}N. 
\]
By unwinding the definitions, one checks that $\theta_C$ is $\cO_X$-linear. 

The assignment $C\mapsto \theta'_C$ is a bijection, since it coincides with the construction in \cite[III, Theorem~6.7]{Har66}. 
\end{proof}

We next identify the $\cO_X$-module $\mathcal{C}_X^{(m)}$.  
We begin with the following general lemma. 
\begin{lem}\label{lem: general Y to Z}
Let $f\colon Y\to Z$ be a finite flat morphism of  Noetherian schemes. Assume that the $f_\ast\cO_Y$-module 
\[
\mathcal{H}\coloneqq \mathcal{H}om_{\cO_Z}(f_\ast\cO_Y,\cO_Z)
\]
    is invertible. Let 
    \[
    \tilde{\eta}\colon f_\ast\cO_Y\to \mathcal{H}
    \]
    be an $f_\ast\cO_Y$-linear morphism, and set $\eta\coloneqq\tilde{\eta}(1)\colon f_\ast\cO_Y\to \cO_Z$. For any point $z\in Z$ with residue field $k(z)$, let 
    \[
    \eta_z\colon f_{z\ast}\cO_{Y_z}\to k(z)
    \]
    denote the base change of $\eta$ to $k(z)$. 
    Then the following are equivalent: 
    \begin{enumerate}
        \item The morphism $\tilde{\eta}$ is an isomorphism. 
        \item For every point $z\in Z$, there exists no nonzero coherent $\cO_{Y_z}$-submodule $M\subset \cO_{Y_z}$ such that $\eta_z(f_{z\ast}M)=0$. 
    \end{enumerate}
\end{lem}
\begin{proof}
Since $\mathcal{H}$ is invertible by assumption, condition~(1) is equivalent to  $\tilde{\eta}$ being surjective. By Nakayama's lemma, this is further equivalent to requiring that, for every point $z\in Z$, the induced map 
\[
\tilde{\eta}_z\colon f_{z\ast}\cO_{Y_z}\to \mathcal{H}\otimes_{\cO_Z}k(z)
\]
is surjective. Hence we may reduce to the case where $Z={\rm Spec}(k)$ for a field $k$. 

Assume that $Z={\rm Spec}(k)$. In this case, $\Gamma(Y,\cO_Y)$ and $\Gamma(Z,\mathcal{H})$ are finite-dimensional $k$-vector spaces of the same dimension. Therefore, $\tilde{\eta}$ is an isomorphism if and only if it is injective. Note that, by adjunction, $\ker\tilde{\eta}$ is the largest $f_\ast\cO_Y$-submodule that is annihilated by $\eta$. Hence, the injectivity of $\tilde{\eta}$ 
is equivalent to~(2). The assertion follows. 
\end{proof}

\begin{lem}\label{lem: generator}
The following statements hold: 
\begin{enumerate}
    \item The $\mathcal{O}_X$-module $\mathcal{C}_X^{(m)}$ is invertible. 
    \item Let $\eta\in \mathcal{C}_X^{(1)}$ be a generator as an $\cO_X$-module. Then the composite $\eta^m\in \mathcal{C}_X^{(m)}$ is a generator 
    as an $\mathcal{O}_X$-module. 
\end{enumerate}
\end{lem}
\begin{proof}
(1) Since $X$ is regular and $F$-finite, $\cO_X$ is a dualizing sheaf. Hence, $F_q^{m!}\cO_X$ is a dualizing complex on $X$. In particular, it is an inverible sheaf up to a shift in degree. Moreover, since $F_q^m$ is finite flat, $F_q^{m!}\cO_X$ is isomorphic to $\mathcal{C}_X^{(m)}$.  The assertion follows. 

\medskip

\noindent
(2) We argue by induction on $m$. The case $m=1$ is given by the assumption. Now assume that $m>1$. 

By Lemma~\ref{lem: general Y to Z} and part~(1), it suffices to show that, for every point $x\in X$ with residue field $k(x)$ and for every nonzero coherent submodule $M\subset \cO_{X_{x,m}}$, where we set $X_{x,m}\coloneqq X\times_{F_q^m,X}x$, we have 
\[
\eta^m_x((F^m_{q,x})_\ast M)\neq0. 
\]
Here, $\eta_x^m$ and $F_{q,x}^m$ denote the base changes of $\eta^m$ and $F_q^m$ over the point $x$, respectively. Let 
\[f\colon X_{x,m}\to X_{x,m}\]
denote the base change of $F_q^{m-1}$ to the fibers over $x\in X$. By the induction hypothesis and  Lemma~\ref{lem: general Y to Z}, the image 
\[
N\coloneqq \eta^{m-1}_x((F_{q,x}^m)_\ast M)\subset (F_{q,x})_\ast\cO_{X_{x,1}}
\]
is nonzero. Since $\eta$ is a generator, again by Lemma~\ref{lem: general Y to Z}, the image $\eta_x(N)$ is nonzero. Hence 
\[
\eta_x^m((F^m_{q,x})_\ast M) = \eta_x(N) \neq 0.
\]
The assertion then follows from Lemma~\ref{lem: general Y to Z}. 
\end{proof}

By Lemma~\ref{lem: generator}(1), the $\cO_X$-module $\mathcal{C}^{(1)}_X$ is invertible. We assume that it is free, and choose a generator $\eta\in \mathcal{C}^{(1)}_X$. Then, by Lemma~\ref{lem: generator}(2), the composite $\eta^m\in\mathcal{C}_X^{(m)}$ is a generator for all $m\ge1$.

   \begin{prop}\label{prop: C to theta}
Fix a generator $\eta\in\mathcal{C}_X^{(1)}$. Then, for any $\cO_X$-modules $M$ and $N$, there exists a natural bijection 
\[
{\rm Hom}_{\cO_X}((F_q^m)_{\ast}M,N)\to {\rm Hom}_{\cO_X}(M,F_{q}^{m\ast} N). 
\]     
This bijection is characterized as follows. 
Let $C \in {\rm Hom}_{\cO_X}((F_q^m)_{\ast}M,N)$, and let $\theta$ denote its image. Then $\theta$ is the unique $\cO_X$-linear morphism $M\to F_q^{m\ast}N$ making the diagram 
\[
\xymatrix{
M\ar[r]^-\theta\ar[rd]_-C&F_q^{m\ast}N\ar[d]^-{\eta^m\otimes{\rm id}_N}\\
&N
}
\]
 commutative. 
   \end{prop} 
\begin{proof}
By Lemma~\ref{lem: generator}(2), the composite $C^m$ induces an isomorphism 
\[
\cO_X\cong\mathcal{C}_X^{(m)}, \quad 1\mapsto \eta^m. 
\]
Hence, the composite map 
\[
{\rm Hom}_{\cO_X}((F_q^m)_\ast M,N)\xrightarrow{\text{Proposition~\ref{prop: C to theta first version}}} {\rm Hom}_{\cO_X}( M,\mathcal{C}_X^{(m)}\otimes_{\cO_X}F_q^{m\ast}N)\cong {\rm Hom}_{\cO_X}( M,F_q^{m\ast}N)
\]
gives the desired bijection. The remaining assertion is immediate from the construction. 
\end{proof}

\begin{remark}
    In Proposition~\ref{prop: C to theta}, we choose a generator of $\mathcal{C}_X^{(1)}$. This choice can in fact be avoided by twisting $M$ and $N$ by a suitable invertible sheaf, so that the resulting  bijection becomes canonical, in the sense that it no longer depends on any auxiliary choice. We describe this construction in the appendix. Although this may be well known to experts, it answers a question raised in \cite[Remark~2.22]{BB11}. When $X$ is of finite type over an $F$-finite field, the construction is classical. 
\end{remark}

\begin{lem}\label{lem: C theta eta}
Let $(M,C)$ be a $q$-Cartier module on $X$, and let $\theta\colon M\to F_q^\ast M$ denote the corresponding $\cO_X$-linear morphism constructed in Proposition~\ref{prop: C to theta}. Let $m$ be a positive integer, and let $\theta_m$ denote the composite morphism 
\[
M\xrightarrow{\theta} F_q^\ast M\xrightarrow{F_q^\ast\theta}F_q^{2\ast}M\xrightarrow{F_q^{2\ast}\theta}\cdots \xrightarrow{(F_q^{m-1 })^\ast\theta} F_q^{m\ast}M. 
\]
\begin{enumerate}
    \item The diagram 
\[
\xymatrix{
M\ar[r]^-{\theta_m}\ar[rd]_-{C^m}&F_q^{m\ast}M\ar[d]^-{\eta^m\otimes{\rm id}_M}\\
&M
}
\]
is commutative. 
\item Consider the $\cD_X^{(em)}$-module structure on $F_q^{m\ast} M$ defined in \S\ref{subsubsection: Frobenius descent}. Then, for any $\cO_X$-submodule $N\subset M$, we have 
\[
\cD_X^{(em)}\cdot \theta_m(N)=F_q^{m\ast}(C^m(N)). 
\]
\end{enumerate}
\end{lem}
\begin{proof}
(1) The case $m=1$ follows from the construction of $\theta$. The general case follows by induction on $m$. 
 
(2) By Theorem~\ref{thm: Frobenius descent}, one can write 
\[
\cD_X^{(em)}\cdot\theta_m(N)=F_q^{m\ast}(N')
\]
for some submodule $N'\subset M$. Since $\eta^m\colon (F_q^m)_{\ast}\cO_X\to\cO_X$ is surjective, we have 
\[
N'=(\eta^m\otimes1)\cdot F_q^{m\ast}(N')=(\eta^m\otimes1)\cdot \cD_X^{(em)}\cdot\theta_m(N). 
\]
Since $\eta^m$ is a generator of $\mathcal{C}_X^{(m)}$ as an $\cO_X$-module, it is also a generator as a right $\cD_X^{(em)}$-module. Hence, 
\[
\eta^m\cdot \cD_X^{(em)}=\mathcal{C}_X^{(m)}=\eta^m\cdot  \cO_X. 
\]
Therefore, we obtain 
\[
(\eta^m\otimes1)\cdot \cD_X^{(em)}\cdot\theta_m(N)=(\eta^m\otimes1)\cdot\theta_m(N). 
\]
By~(1), the right-hand side equals $C^m(N)$. The assertion follows.  
\end{proof}

Let $\sM$ be a unit $F^e$-module on $X$. Via Proposition~\ref{prop: C to theta}, the isomorphism $\theta_{\sM}$ corresponds to an $\cO_X$-linear  morphism 
\begin{equation}\label{equation: CM}
    C_\sM\colon F_{q\ast}\sM\to \sM. 
\end{equation}
We refer to $(\sM,C_\sM)$ as the \emph{associated $q$-Cartier module}.

\begin{lem}\label{lem: minimal root is Fpure}
Let $\sM$ be a unit $F^e$-module on $X$, and let $M\subset\sM$ be an $\cO_X$-submodule. Then the following hold: 
\begin{enumerate}
    \item The submodule $M$ satisfies $\theta_\sM(M)\subset F_q^\ast M$ if and only if $C_\sM(M)\subset M$. 
    \item Suppose that $\sM$ is lfgu and that $M$ is a root of $\sM$. Then $M$ is the minimal root if and only if $(M,C_\sM|_M)$ is an $F$-pure $q$-Cartier module.
\end{enumerate}
\end{lem}
\begin{proof}
(1) Since $F_q^\ast M$ is a $\cD_X^{(e)}$-submodule, the condition $\theta_\sM(M)\subset F_q^\ast M$ is equivalent to 
\[
\cD_X^{(e)}\cdot\theta_\sM(M)\subset F_q^\ast M. 
\]
By Lemma~\ref{lem: C theta eta}(2) and the faithful flatness of $F_q$, this is equivalent to $C_\sM(M)\subset M$. 

\medskip

\noindent
(2) Suppose that $M$ is a root. Let $\underline{M}\subset M$ be the $F$-pure submodule defined in~\eqref{Mbar}. Consider the exact sequence of $q$-Cartier modules 
\[
0\to \underline{M}\to M\to N\to0. 
\]
Note that $C$ is nilpotent on $N$. By Proposition~\ref{prop: C to theta} together with Lemma~\ref{lem: C theta eta}(1), it follows that,  for some integer $m\ge 1$, the induced morphism  $N\to F_q^{m\ast}N$ is zero. This implies that the unitalization of $N$ is zero. 
Since the unitalization functor is exact, the unitalization of $\underline{M}$ is $\sM$, so that  $\underline{M}$ is a root of $\sM$. Consequently, if $M$ is minimal, then $M=\underline{M}$, and hence $M$ is $F$-pure. 

Conversely, assume that $M$ is $F$-pure. Let $M_{\rm min}\subset M$ denote the minimal root. Since $M$ is coherent, the equality 
\[
\sM=\bigcup_{m\ge0}F_q^{m\ast}M_{\rm min}
\]
implies that there exists $m\ge0$ such that $M\subset F_q^{m\ast}M_{\rm min}$. By Lemma~\ref{lem: C theta eta}(2) and the faithful flatness of $F_q$, we then have  
\[
C^m_\sM(M)\subset M_{\rm min}. 
\]
On the other hand, since $M$ is $F$-pure, we have $C^m_\sM(M)=M$. It follows that $M=M_{\rm min}$, and hence $M$ is minimal. 
\end{proof}

The following proposition is a refinement of  \cite[Corollary~4.4]{AMBL05}. 
\begin{prop}\label{prop: refinement of generator}
Let $\sM$ be an lfgu $F^e$-module on $X$, and let $M\subset\sM$ denote the minimal root. Then for every integer $m\ge1$, 
\[
\cD_X^{(em)}M=F_q^{m\ast}M. 
\]
Consequently, for any root $N$ of $\sM$, we have $\cD_XN=\sM$. 
\end{prop}
\begin{proof}
    Since the statement is local, we may assume that $\sC_X^{(1)}$ is trivial. Fix a generator and let $C$ denote the associated $q$-Cartier operator on $\sM$. Then, by Lemma~\ref{lem: C theta eta}(2), 
    \[
    \cD_X^{(em)}M=F_q^{m\ast}(C^m(M)). 
    \]
    On the other hand, by 
    Lemma~\ref{lem: minimal root is Fpure}(2), $M$ is $F$-pure. Therefore, the right-hand side equals $F_q^{m\ast}M$, which proves the displayed equality. 

To prove the last assertion, it suffices to treat  the case where $N$ is the minimal root $M$.  Since $M$ is a root, we have $\sM=\bigcup_{m\ge0}F_q^{m\ast}M$. The claim then follows from the displayed equality.  
\end{proof}

\begin{thm}\label{thm: nu invariants vs BSR Dmodule}
Let $\sM$ be a unit $F^e$-module, and let $M\subset \sM$ be an $\cO_X$-submodule. Let $f\in \Gamma(X,\cO_X)$. We  regard $\sM$ as a $q$-Cartier module by \eqref{equation: CM} and a $\cD_X$-module by Construction~\ref{construction: Dmodule str}. 
Then the following hold: 
\begin{align*}
    &\nu_f(q^m;M)=\BSR^{(em)}(M,f)\qquad(\forall m\ge1),\\
    &\nu_f(q^\infty;M)\supset \BSR(M,f). 
\end{align*}
 Moreover, if $M$ is locally finitely generated as an $\cO_X$-module, then  
 \[
  \nu_f(q^\infty;M)=\BSR(M,f).
 \]
\end{thm}
\begin{proof}
The first equality  follows from Lemma~\ref{lem: C theta eta}(2) together with the faithful flatness of $F_q^m$. The remaining assertions follow from the first equality together with Lemma~\ref{BSRnu}(2) and Proposition~\ref{prop: infinite BSR vs finite}(2). 
\end{proof}

Let $\sM$ be an lfgu $F^e$-module on $X$, and let $M\subset\sM$ be a coherent $\cO_X$-submodule. Motivated by the above theorem, we consider the following assumption on $(\sM,M)$ as an analogue of Assumption~\ref{assumption: nu invariant is bound}. 
\begin{assumption}\label{assumption: BSR is bound}
There exists a constant $L\geq0$ such that 
    \[
    \#\Big(\BSR^{(em)}(M,f)\cap [0,q^m)\Big)\leq L
    \]
    for all $m\geq1$. 
\end{assumption}

The following theorem follows from the combination of results presented above. 
\begin{thm}\label{thm: finiteness and rationality of BSR}
Let $X$ be a  Noetherian regular $F$-finite  $\F_p$-scheme. 
Let $\sM$ be an lfgu $F^e$-module on $X$, and let $M\subset\sM$ be  a coherent $\cO_X$-submodule. 
Then the following statements hold: 
\begin{enumerate}
\item The pair $(\sM,M)$ satisfies Assumption~\ref{assumption: BSR is bound} if $X$ admits an open covering $X=\bigcup_iU_i$ where each $U_i$ is of finite type over an $F$-finite field. 
    \item Assume that $(\sM,M)$ satisfies Assumption~\ref{assumption: BSR is bound}. 
    Then the set  
$\BSR(M,f)$ is finite. Moreover, if $M$ is a root, then it is contained in 
 $\Z_{(p)}=\Z_p\cap\Q$. Finally, 
 if $M$ is the minimal root of $\sM$, then we have  
\[
\BSR(M,f)\subset\Z_{(p)}\cap[-1,0]. 
\]
\end{enumerate}
\end{thm}
\begin{proof}
Let $N$ be a root  of $\sM$. Since $\sM=\bigcup_{m\ge0}F^{m\ast}_qN$ and $M$ is coherent, it follows that $M\subset F_q^{m\ast}N$ for a sufficiently large $m$. As $F_q^{m\ast}N$ is a root for all $m$, after replacing $N$ by $F_q^{m\ast}N$, we may assume that $M\subset N$. 

Since the statements can be checked locally, we may assume that $\mathcal{C}_X^{(1)}$ is trivial. Fix a generator, and let $C\colon F_{q\ast}\sM\to \sM$ denote the corresponding $\cO_X$-linear morphism.

By Lemma~\ref{lem: minimal root is Fpure}(1), 
$N$ is a coherent $q$-Cartier submodule of $\sM$. 
By Theorem~\ref{thm: nu invariants vs BSR Dmodule}, Assumption~\ref{assumption: BSR is bound} for the pair $(\sM,M)$ is equivalent to  Assumption~\ref{assumption: nu invariant is bound} for the pair $((N,C),M)$. Moreover, by Lemma~\ref{lem: minimal root is Fpure}(2), the condition that $M$ is the minimal root implies that $M$ is an $F$-pure $q$-Cartier submodule. 

Therefore, the assertions follow from the corresponding statements for $q$-Cartier modules, namely, Proposition~\ref{consBl}, Proposition~\ref{BSRfin}, and  Theorem~\ref{BSRrat}. 
\end{proof}

\subsection{\texorpdfstring{The Module $N_f(M)_\alpha$}{The Module Nf(M)alpha}}\label{subsection: The Nearby Cycle}

Let $f\in \Gamma(X,\cO_X)$. Set $U={\rm Spec}(\cO_X[f^{-1}])$, and let $j\colon U\hookrightarrow X$ denote the open immersion. 

Let $\sM$ be an lfgu $F^e$-module on $X$, and let $M$ denote its minimal root. 
For such a pair $(\sM,M)$, the 
$\cD_X$-module $N_f(M)_\alpha$ studied in Section~\ref{section: Nearby Cycles via the Graph Embedding} has a simple description, which we explain in this subsection. We will freely use the notation in \S\ref{subsection: The functors Nfalpha}. 

%(the morphism $t\colon N_f(M)_{\alpha+1}\to N_f(M)_\alpha$ has  already been determined in Section~\ref{section: Nearby Cycles via the Graph Embedding}) 

 We begin with the case $\alpha\notin \Z_{(p)}$. In this case, 
 we require Assumption~\ref{assumption: BSR is bound}. 
\begin{prop}
Let $\sM$ be an lfgu $F^e$-module on $X$, and let $M$ be a root of $\sM$. Assume that the pair $(\sM,M)$ satisfies Assumption~\ref{assumption: BSR is bound}. 
Let $\alpha\in \Z_p\setminus\Z_{(p)}$. Then 
\[
N_f(M)_\alpha\cong \sM\otimes_{\cO_X}\sL_{f,\alpha}. 
\]
\end{prop}
\begin{proof}
By Proposition~\ref{prop: t sends a to a-1} and Theorem~\ref{thm: finiteness and rationality of BSR}, we have 
\[
N_f(M)_\alpha\cong N_f(M)_{\alpha-n}\qquad\text{for all $n\in\Z_{\ge0}$. }
\]
On the other hand, by Proposition~\ref{prop: non-integral case}(2), we have  
\[N_f(M)_{\alpha-n}\cong\cD_X(\xi(M)\cdot f^{\alpha-n})=\cD_X(f^{-n}\xi(M)\cdot f^{\alpha}),\] 
where $\xi$ denotes the canonical morphism $\sM\to \sM_f$. Passing to the colimit over $n$, we obtain 
\[
N_f(M)_\alpha=\bigcup_{n\ge0}\cD_X(f^{-n}\xi(M)\cdot f^{\alpha})=\cD_X(M\otimes_{\cO_X}\sL_{f,\alpha}). \]
It thus suffices to show that 
\begin{equation}\label{equation: MLfa=DML} \sM\otimes_{\cO_X}\sL_{f,\alpha}=\cD_X(M\otimes_{\cO_X}\sL_{f,\alpha}). 
\end{equation}
Since both sides are quasi-coherent $\cO_X[f^{-1}]$-modules, it suffices to verify the equality after restricting to $U$. Thus, we may assume that $X=U$. 

In this case,  $\sL_{f,\alpha}$ is an invertible $\cO_X$-module. Consequently, the functor 
\[\sN\mapsto \sN\otimes_{\cO_X}\sL_{f,\alpha}\]
defines an autoequivalence of the category of $\cD_X$-modules (see Construction~\ref{construction: pullback pushforward tensor product of Dmodules}(3)). A quasi-inverse is given by $\sN\mapsto\sN\otimes_{\cO_X}\sL_{f,-\alpha}$. 

Under this equivalence, the submodule 
\[
\cD_X(M\otimes_{\cO_X}\sL_{f,\alpha})\subset \sM\otimes_{\cO_X}\sL_{f,\alpha}
\]
corresponds to a $\cD_X$-submodule of $\sM$ containing $M$. By Proposition~\ref{prop: refinement of generator}, $\sM$ is the only submodule with this property. This proves \eqref{equation: MLfa=DML}. 

\end{proof}

\medskip

From now on, we consider the case $\alpha\in\Z_{(p)}$. 
We assume that the invertible $\cO_X$-module $\mathcal{C}_X^{(1)}$ is trivial, and fix a generator. We impose this assumption only to simplify the exposition; it can be removed by using 
\eqref{equation: FMN to MFN} instead. 

Let $C\colon F_{q\ast}\sM\to\sM $ 
denote the $\cO_X$-linear morphism corresponding to $\theta_\sM$ via Proposition~\ref{prop: C to theta}. 
The $F^e$-module $\sM_f\coloneqq j_+ j^\ast \sM$
 is an lfgu $F^e$-module on $X$ by Lemma~\ref{lem: functoriality for Fmodule}. Let $\xi\colon \sM\to \sM_f$ denote the canonical morphism.

\subsubsection{The Integral Case}
Here, we compute $N_f(f^nM)_m$ for $n\in\Z_{\ge0}$ and $m\in\Z$. 
By Proposition~\ref{prop: computation of t integral case}, we have the following. 
\begin{lem}\label{lem: NffiM=}
    We have 
    \[
N_f(f^iM)_n\cong\begin{cases}\cD_Xf^{i+n}M&\text{if $n\ge0$,}\\
    \cD_Xf^{i+n}\xi(M)&\text{if $n<0$.}  
    \end{cases}
    \]
\end{lem}
Therefore, it suffices to compute $\cD_Xf^iM$ for $i\ge0$ and $\cD_Xf^i\xi(M)$ for all $i\in \Z$. 
To this end, we introduce the following subobject of $\sM$. 
\begin{prop}\label{prop: minimal subobject for F module}
There exists a smallest subobject 
\[
\sM_{f!}\subset\sM
\]
    in the category of lfgu $F^e$-modules on $X$ such that $j^\ast \sM_{f!}=j^\ast\sM$. 
\end{prop}
\begin{proof}
    By \cite[Theorem~5.13]{BB11}, an lfgu $F^e$-module has finite length. This proves the assertion. 
\end{proof}

Note that $\sM_{f!}$ is the \emph{largest} subobject admitting no nonzero morphism 
\[
\sM_{f!}\to \sN
\]
to an lfgu $F^e$-module supported on $X\setminus U$. This motivates the notation $\sM_{f!}$ (extension by zero).

By Lemma~\ref{lem: minimal root is Fpure}(2), the $q$-Cartier module $(M,C)$ is $F$-pure. Let $M_{f!}\subset M$ denote the submodule defined in Definition~\ref{defn: zero extension of Cartier module}. By Lemma~\ref{lem: right reduced}, this is an $F$-pure $q$-Cartier module, and we have 
\[
M_{f!}=C^m(f^iM), 
\]
where $i$ is any positive integer, and $m$ is a sufficiently large positive integer (depending on $i$). 

\begin{lem}\label{lem: roots for Mf}
Let $\sM$ be an lfgu $F^e$-module, and let $M$ denote its minimal root. 
The following statements hold:  
\begin{enumerate}
    \item The submodule $f^{-1}\cdot \xi(M)\subset \sM_f$ 
    is a root. 
    \item The module $M_{f!}$ is the minimal root of $\sM_{f!}$.   
\end{enumerate}
\end{lem}
\begin{proof}
(1) The claim follows from the computation 
\[
F_q^{m\ast}(f^{-1}\cdot \xi(M))=f^{-q^m}F^{m\ast}_q(\xi(M)), 
\]
and the fact $\sM=\bigcup_{m\ge1}F^{m\ast}_qM$. 

(2) Since $M_{f!}$ is a $q$-Cartier submodule of $\sM$, the morphism $\theta_\sM$ induces an $\cO_X$-linear morphism $M_{f!}\to F_q^\ast M_{f!}$. We denote by $\sM'$ the unitalization of $M_{f!}$. By construction, $\sM'$ is an lfgu $F^e$-submodule of $\sM$. Moreover, 
by Lemma~\ref{lem: minimal root is Fpure}(2), $M_{f!}$ is the minimal root of $\sM'$. 

    It remains to show $\sM'=\sM_{f!}$. Since taking the restriction $j^\ast$ commutes with taking the image of $C^m$, we have 
    \[
    j^\ast M_{f!}=C^m(j^\ast M)=j^\ast M. 
    \]
Hence, $j^\ast \sM'=j^\ast\sM$, which shows that $\sM_{f!}\subset \sM'$. 

We prove the reverse inclusion. Let $\phi\colon \sM'\to \sN$ be a morphism to 
 an lfgu $F^e$-module $\sN$ on $X$ supported on $X\setminus U$. Since $\sN$ is $\cO_X$-quasi-coherent and $M_{f!}$ is coherent, there exists a positive integer $n\ge1$ such that 
 \[
 f^nM_{f!}\subset\ker\phi. 
 \]
Since $\phi$ is a morphism of unit $F^e$-modules, it commutes with the $\cD_X$-action. Thus we have
\[
\cD_X^{(em')}f^nM_{f!}\subset\ker\phi. 
\]
 On the other hand, 
applying Lemma~\ref{lem: right reduced}(1) to $M=M_{f!}$ and $\cI=f^n\cO_X$, we obtain 
\[
M_{f!} =C^{m'}(f^nM_{f!})
\]
for all sufficiently large $m'$. Hence, by Lemma~\ref{lem: C theta eta}(2), 
\[
F^{m'\ast}_q M_{f!}=\cD_X^{(em')}f^nM_{f!}. 
\]
Consequently, 
\[
F^{m'\ast}_q M_{f!}\subset \ker\phi. 
\]
    Since this holds for all sufficiently large $m'$, and since $\sM'=\bigcup F_q^{m'\ast}M_{f!}$, we obtain $\sM'\subset\ker \phi$. This proves the inclusion $\sM'\subset \sM_{f!}$. The proof is completed. 
\end{proof}
\begin{prop}\label{prop: description of DXM}
We have 
\[
\cD_Xf^iM=\begin{cases}
    \sM&\text{if $i=0$,}\\
     \sM_{f!}&\text{if $i>0$,}
\end{cases}
\qquad
\cD_Xf^i\xi(M)=\begin{cases}
\sM_f&\text{if $i<0$,}\\
\xi(\sM)&\text{if $i=0$,}\\
\xi(\sM_{f!})&\text{if $i>0$.}
\end{cases}
\]
\end{prop}
\begin{proof}
It suffices to show that 
\[
\cD_Xf^{-1}\cdot\xi(M)=\sM_f,\qquad \cD_XM=\sM,\qquad \cD_Xf^iM=\sM_{f!}\quad(\forall i\ge1). 
\]
Since $M$ is a root of $\sM$ and $f^{-1}\cdot\xi(M)$ is a root of $\sM_f$, the equalities for $\sM$ and $\sM_f$ follow from Proposition~\ref{prop: refinement of generator}. 

    We prove $\cD_Xf^iM=\sM_{f!}$. 
By Lemma~\ref{lem: C theta eta}(2), for every $m\ge1$, we have 
\[
\cD_X^{(em)}f^iM=F_q^{m\ast}(C^m(f^iM)). 
\]
Since $i\ge1$, the module $C^m(f^iM)$ is equal to $M_{f!}$ for all sufficiently large $m$. Thus, 
\[
\cD_Xf^iM=\bigcup_{m\gg0}F_q^{m\ast}M_{f!}=\sM_{f!}. 
\]
The proof is completed. 
\end{proof}

\subsubsection{The Rational Case}
Let $\alpha\in\Z_{(p)}\cap(-1,0)$. In this subsection, we compute $N_f(f^iM)_{\alpha+n}$ for $i\in\Z_{\ge0}$ and $n\in\Z$. By Propositions~\ref{prop: t sends a to a-1}(2) and~\ref{prop: non-integral case}(2), we have the following. 
\begin{lem}
    We have 
    \[
    N_f(f^iM)_{\alpha+n}\cong \cD_X(f^{i+n}\xi(M)\cdot f^\alpha). 
    \]
\end{lem}
Therefore, it suffices to compute $\cD_X(f^{i}\xi(M)\cdot f^\alpha)$ for all $i\in\Z$. 

\medskip

Since the denominator of $\alpha$ is prime to $p$, there  exists a positive power $q'$ of $q$  with $(q'-1)\alpha\in\Z$. Fix such a $q'$ and write $q'=p^{e'}$.  Set $a=(1-q')\alpha$. Then $a\in\{1,\dots,q'-2\}$ and 
\[
\alpha=\frac{a}{1-q'}=\sum_{i\ge0}aq'^{i}\qquad\text{in $\Z_p$.}
\]
We also set $\alpha_m\coloneqq \sum_{i=0}^{m-1}aq'^i$, with the convention that $\alpha_0=0$. 

The $\cO_X$-module $\sL_{f,\alpha}$ considered in \S\ref{subsubsection: Non-Integral Case} admits an $F^{e'}$-module structure by letting $F^{e'}$ act via $F^{e'}(f^\alpha)=f^{-a}f^\alpha$. The $F^{e'}$-module $\mathscr{L}_{f,\alpha}$ is lfgu \cite[Lemma~2.3 and Remark~2.10]{BMS09}. 
Moreover, the induced $\cD_X$-module structure coincides with the $\cD_X$-module structure considered in \S\ref{subsubsection: Non-Integral Case}.  
We also note that $ \cO_X\cdot f^\alpha\subset \mathscr{L}_{f,\alpha}$ is a root.

We will freely use the following fact. 
\begin{lem}
    Let $\sM$ be an lfgu $F^e$-module on $X$.Via the canonical injection $\cO_X[F^{e'}]\hookrightarrow \cO_X[F^e]$, we regard $\sM$ as an $\cO_X[F^{e'}]$-module. 
    Then the resulting module is again lfgu. 
\end{lem}
\begin{proof}
    Since $\cO_X[F^e]$ is finitely generated as a left $\cO_X[F^{e'}]$-module, $\sM$ is also finitely generated as a left $\cO_X[F^{e'}]$-module. 

Note that the $\cO_X$-linear morphism $F^{e'\ast}\sM\to \sM$ induced by the $\cO_X[F^{e'}]$-module structure agrees with the composite of the iterates of the structural map $F^{e\ast}\sM\to \sM$. Since the latter is an isomorphism, it follows that $\sM$ is unit as an $\cO_X[F^{e'}]$-module. 
\end{proof}

 \begin{defn}\label{defn: Tame nearby cycles for Frobenius modules}
Let $\sM$ be an lfgu $F^e$-module on $X$. 
\begin{enumerate}
    \item Set $\sM_{f,\alpha}\coloneqq \sM\otimes_{\cO_X}\mathscr{L}_{f,\alpha}$, which is an lfgu $F^{e'}$-module on $X$. 
    \item Let  $\sM_{f,\alpha,!}\subset\sM_{f,\alpha}$ denote the smallest subobject in the category of lfgu $F^{e'}$-modules such that $j^\ast\sM_{f,\alpha,!}= j^\ast\sM_{f,\alpha}$. 
\end{enumerate}
 \end{defn}

\begin{prop}\label{prop: tame case for Fmodule}
Let $\sM$ be an lfgu $F^e$-module on $X$, and let $M\subset\sM$ denote its minimal root. Let $\xi$ denote the canonical morphism $\sM\to \sM_f$.   Then 
\[
\cD_X(f^i\xi(M)\cdot f^\alpha)=\begin{cases}
    \sM_{f,\alpha}&\text{if $i\le0$,}\\
    \sM_{f,\alpha,!}&\text{if $i>0$.}
\end{cases}
\]
\end{prop}
\begin{proof}
Replacing $q$ by $q'$, we may assume that $q'=q$ and $e'=e$. 

  The equality 
\[\cD_X(\xi(M)\cdot f^\alpha)=\sM_{f,\alpha} 
\]
 follows from Proposition~\ref{prop: refinement of generator}, since $\xi(M)\cdot f^\alpha$ is a root of $\sM_{f,\alpha}$. This also proves the case $i<0$. 

\medskip 

It remains to show that 
\[
\cD_X(f^i\xi(M)\cdot f^\alpha)=\sM_{f,\alpha,!} \qquad\text{for all $i\ge1$. }
\]
We first prove that this equality holds for all sufficiently large $i$. 
Let $M'$ denote the minimal root of $\sM_{f,\alpha,!}$. Since $\xi(M)\cdot f^\alpha$ is a root of $\sM_{f,\alpha}$, we have 
\[
M'\subset \xi(M)\cdot f^\alpha. 
\]
Since $j^\ast\sM_{f,\alpha}=j^\ast\sM_{f,\alpha,!}$, it follows from Lemma~\ref{lem: minimal root restrict to minimal root} that $j^\ast M'$ is the minimal root of $j^\ast \sM_{f,\alpha}$. On the other hand, since $j^\ast\mathscr{L}_{f,\alpha}$ is an invertible $\cO_U$-module, the assignment $N\mapsto N\otimes_{\cO_U}j^\ast\mathscr{L}_{f,\alpha}$ induces a bijection between the set of roots of $j^\ast \sM_{f}$ and that of $j^\ast\sM_{f,\alpha}$. 
This implies that $j^\ast(\xi(M)\cdot f^\alpha)$ 
is also the minimal root of $j^\ast\sM_{f,\alpha}$. 
Therefore, 
\[
j^\ast M'=j^\ast(\xi(M)\cdot f^\alpha). 
\]
Since $\xi(M)\cdot f^\alpha$ is $\cO_X$-coherent, it follows that for some $i_0\ge1$, we have 
\[
f^{i_0}\xi(M)\cdot f^\alpha\subset M'. 
\]
Thus, for every $m\ge0$, we obtain  
\[
\cD_X(f^{i_0+m}M')\subset \cD_X(f^{i_0+m}\xi(M)\cdot f^\alpha)\subset \cD_XM'. 
\]
Applying Proposition~\ref{prop: description of DXM} to $\sM=\sM_{f,\alpha,!}$ yields 
\[
\cD_X(f^{i_0+m}M')=\cD_XM'=\sM_{f,\alpha,!}. 
\]
Hence, 
\[\cD_X(f^i\xi(M)\cdot f^\alpha)=\sM_{f,\alpha,!}\qquad\text{for all $i\ge i_0$. }
\]

We are now reduced to showing that  $N_f(fM)_{\alpha}=N_f(f^iM)_\alpha$ for all $i\ge1$.  By Proposition~\ref{prop: t sends a to a-1}, it suffices to prove that 
\[
\alpha+i\notin\BSR(M,f) \qquad\text{for all $i\ge1$. }
\]
Suppose, for contradiction, that there exists $i\ge1$ with $\alpha+i\in\BSR(M,f)$. By  Lemma~\ref{lem: nu invariants for Fpure} and Theorem~\ref{thm: nu invariants vs BSR Dmodule}, for every integer $n\ge1$, there exists an integer $m_n\in[0,q^n-1]$ such that 
\[
q^n(\alpha+i)+m_n\in\BSR(M,f). 
\]
Since $\alpha+i>0$, the rational number $q^n(\alpha+i)+m_n$ tends to $\infty$ as $n\to\infty$. On the other hand, using the relation $q\alpha=\alpha-a$, one checks by induction that $q^n(\alpha+i)+m_n-\alpha
$ is an integer. 
Hence, it follows that there exists a strictly increasing sequence of integers 
\[
m_1'<m_2'<\cdots
\]
such that $\alpha+m_i'\in\BSR(M,f)$ for all $i\ge1$. 
This contradicts the fact that $\alpha+j\notin \BSR(M,f)$ for all sufficiently large $j\ge1$. 
\end{proof}

\appendix
\section{Global Interpretation of Cartier Operators}\label{section: Global Interpretation of Cartier Operators}
Let $X$ be a  Noetherian regular $F$-finite  $\F_p$-scheme. 
Let $e$ be a positive integer and let $q=p^e$. 
In this appendix, for any $\cO_X$-modules $M$ and $N$, we construct a functorial bijection 
\begin{equation}\label{equation: FMN to MFN}
{\rm Hom}_{\cO_X}(F_{q\ast}(M\otimes_{\cO_X}\omega_X),N\otimes_{\cO_X}\omega_X)\xrightarrow{\cong}{\rm Hom}_{\cO_X}(M,F_q^\ast N),   
\end{equation}
where $\omega_X$ is a suitably chosen dualizing sheaf. We begin by defining $\omega_X$. 
\begin{lem}\label{lem: OmegaX is free}
    The sheaf of K\"ahler differentials $\Omega_X^1=\Omega^1_{X/\Z}$ 
    is a locally free $\cO_X$-module of finite rank. 
\end{lem}
\begin{proof}
Since the statement is local, we may assume that $X$ is affine. Let $R=\Gamma(X,\cO_X)$. 
Note that 
\[
\Omega_R^1=\Omega_{R/R^p}^1. 
\]
 This is a finite $R$-module, since $R$ is finite over $R^p$. 
    
    It remains to show that $\Omega^1_R$ is flat. 
By Popescu's theorem, $R$ can be written as a filtered colimit of smooth $\F_p$-algebras $R_i$. For each $i$, the $R_i$-module $\Omega^1_{R_i}$ is finite projective, and hence flat. It follows that  the filtered colimit $\Omega^1_R$ is flat. The assertion follows. 
\end{proof}

\begin{defn}
    We define 
    \[
    \omega_X\coloneqq \bigwedge^{\mathop{\rm rk}\Omega_X^1}\Omega_X^1, 
    \]
   where $\mathop{\rm rk}\Omega_X^1$ denotes the locally constant function on $X$ that assigns to each point the rank of $\Omega_X^1$. 
\end{defn}
The construction in \eqref{equation: FMN to MFN} is reduced to the following result.  
\begin{prop}\label{prop: CX and omegaX}
    There exists a canonical isomorphism of $\cO_X$-modules 
    \[
    \omega_X^{\otimes(1-q)}\cong \mathcal{C}_X^{(1)}. 
    \]
\end{prop}
\begin{cor}
For any $\cO_X$-modules $M$ and $N$, there exists a functorial bijection as in \eqref{equation: FMN to MFN}. 
\end{cor}
\begin{proof}
By Propositions~\ref{prop: C to theta first version} and~\ref{prop: CX and omegaX}, we obtain canonical identifications  
\begin{align*}
  {\rm Hom}_{\cO_X}(F_{q\ast}(M\otimes_{\cO_X}\omega_X),N\otimes_{\cO_X}\omega_X)&\cong{\rm Hom}_{\cO_X}(M\otimes_{\cO_X}\omega_X, \omega_X^{\otimes(1-q)}\otimes_{\cO_X}F_q^\ast(N\otimes_{\cO_X}\omega_X))\\
  &\cong {\rm Hom}_{\cO_X}(M, \omega_X^{\otimes(-q)}\otimes_{\cO_X}F_q^\ast(N\otimes_{\cO_X}\omega_X)).   
\end{align*}
Since $\omega_X$ is invertible, $F_q^\ast\omega_X$ is canonically isomorphic to $\omega_X^{\otimes q}$. Hence, 
\[{\rm Hom}_{\cO_X}(M, \omega_X^{\otimes(-q)}\otimes_{\cO_X}F_q^\ast(N\otimes_{\cO_X}\omega_X))\cong{\rm Hom}_{\cO_X}(M, F_q^\ast N), 
\]
and therefore, 
\[
{\rm Hom}_{\cO_X}(F_{q\ast}(M\otimes_{\cO_X}\omega_X),N\otimes_{\cO_X}\omega_X)\cong {\rm Hom}_{\cO_X}(M, F_q^\ast N). 
\]
This proves the assertion. 
\end{proof}

The remainder of this appendix is devoted to the proof of Proposition~\ref{prop: CX and omegaX}. 
Let $J_q$ denote the kernel of the multiplication morphism   
\[
\cO_X\otimes_{\cO_X^q}\cO_X\to \cO_X,\quad a\otimes b\mapsto ab. 
\]
Define 
\[
S(J_q)\coloneqq\{x\in J_q\mid J_q\cdot x=0\}. 
\]
This is a coherent $\cO_X$-module via the identification $(\cO_X \otimes_{\cO_X^q} \cO_X)/J_q \cong \cO_X$. 

\begin{lem}\label{lem: S(J)=omega}
    There exists a canonical isomorphism of $\cO_X$-modules 
    \[
    \omega_X^{\otimes(q-1)}\xrightarrow{\cong} S(J_q). 
    \]
\end{lem}
\begin{proof}
We first treat the case where $X$ is affine and the ring $R\coloneqq\Gamma(X,\cO_X)$ admits a $p$-basis $f_1,\dots,f_n$ over $R^p$. Set $S\coloneqq R\otimes_{R^q}R$, and regard it as an $R$-algebra via the first factor. For an element $x\in R$, we write $\xi(x)\coloneqq1\otimes x-x\otimes1$.

Since the elements $f_1,\dots,f_n$ form a $p$-basis, $R$ is a free $R^q$-module with basis $f_1^{e_1}\cdots f_n^{e_n}\quad (0\le e_i\le q-1)$. Therefore, the $R$-module $S$ is  free with basis $1\otimes f_1^{e_1}\cdots f_n^{e_n}$. 
Hence, the $R$-algebra homomorphism 
\[
\rho\colon R[x_1,\dots,x_n]\to S,\quad x_i\mapsto \xi(f_i)
\]
is surjective. Since $\rho$ annihilates the ideal  $(x_1^q,\dots,x_n^q)$, it induces a surjective $R$-algebra homomorphism 
\[
\overline{\rho}\colon R[x_1,\cdots,x_n]/(x_1^q,\dots,x_n^q)\to S. 
\]
By comparing ranks as $R$-modules, it follows that $\overline{\rho}$ is an isomorphism.

Under this isomorphism, the ideal $J_q$ corresponds to the ideal $(x_1,\dots,x_n)$, and $S(J_q)$ corresponds to $(x_1^{q-1}\cdots x_n^{q-1})$. Therefore, $\overline{\rho}$ induces an isomorphism 
\[
J_q/J_q^2\cong (x_1,\dots,x_n)/(x_1,\dots,x_n)^2. 
\]
On the other hand, there is a canonical isomorphism $
J_q/J_q^2\xrightarrow{\cong}\Omega_{R/R^q}^1=\Omega_R^1$ that sends $\xi(x)$ to $dx$. It follows that  $\Omega_R^1$ is a free $R$-module with basis $df_1,\dots,df_n$. Therefore, we obtain an isomorphism of invertible $R$-modules 
\[
\eta_{f_\bullet}\colon\omega_R^{\otimes(q-1)}\to S(J_q),\qquad (df_1\wedge\cdots\wedge df_n)^{\otimes(q-1)}\mapsto \bigl(\xi(f_1)\cdots \xi(f_n)\bigr)^{q-1}. 
\]
We claim that $\eta_{f_\bullet}$ does not depend on the choice of $p$-basis. Indeed, suppose that we are given another $p$-basis $g_1,\dots,g_m\in R$. Since $dg_1,\dots,dg_m$ also form a basis of $\Omega_R^1$, it follows that $m=n$ and that there exists a  invertible matrix $A=(a_{ij})\in GL_n(R)$ such that 
\begin{equation}\label{equation: dg=Adf}
    dg_i=\sum_{j=1}^na_{ij}\, df_j\qquad(\forall i=1,\dots,n).
\end{equation}
Hence, 
\[
dg_1\wedge\cdots\wedge dg_n=\det(A)\cdot df_1\wedge\cdots\wedge df_n. 
\]
  Therefore, to prove that $\eta_{f_\bullet} = \eta_{g_\bullet}$, it suffices to show that 
  \begin{equation}\label{equation: g=det Af}
      \big(\xi(g_1)\cdots \xi(g_n)\big)^{q-1}=\det(A)^{q-1}\big(\xi(f_1)\cdots \xi(f_n)\big)^{q-1}
  \end{equation}
  in $S$. 

  We now prove  \eqref{equation: g=det Af}. The relation \eqref{equation: dg=Adf} implies that 
  \[
  \xi(g_i)\equiv \sum_ja_{ij}\otimes1\cdot\xi(f_j)\mod J_q^2. 
  \]
  Since $J_q^{n(q-1)+1}=0$, it follows that 
  \[
  \big(\xi(g_1)\cdots \xi(g_n)\big)^{q-1}=\prod_i\big(\sum_ja_{ij}\otimes1\cdot\xi(f_j)\big)^{q-1}. 
  \]
  Then \eqref{equation: g=det Af} follows from the lemma below. This proves that the isomorphism $\eta_{f_\bullet}$ is independent of the choice of $p$-basis. 

  \medskip

  \noindent
  We now treat the general case. Since $X$ is regular and $F$-finite, it admits a $p$-basis locally. By the independence of $\eta_{f_\bullet}$ from the choice of $p$-basis, these local isomorphisms glue to yield a global isomorphism. 
\end{proof}

\begin{lem}
Let $R$ be an $\F_p$-algebra, and let $A=(a_{ij})\in GL_n(R)$. Then we have 
\[
\prod_{i=1}^n\big(\sum_{j=1}^na_{ij}\cdot x_j\big)^{q-1}=\det(A)^{q-1}\prod_{i=1}^nx_i^{q-1}
\]
in $R[x_1,\dots,x_n]/(x_1^q,\dots,x_n^q)$.   
\end{lem}
\begin{proof}
    Clearly, the left-hand side lies in $(x_1^{q-1}\cdots x_n^{q-1})$. Thus we may write 
    \[
    \prod_i\big(\sum_ja_{ij}\cdot x_j\big)^{q-1}=c(A)\prod_ix_i^{q-1} 
    \]
    for some $c(A)\in R$. We claim that the assignment $A\mapsto c(A)$ is multiplicative. 
    
    Set $y_i\coloneqq\sum_ja_{ij}x_j$. We show that for any $B=(b_{ij})\in GL_n(R)$, we have 
    \begin{equation}\label{equation: AB=C}
        c(A)c(B)\prod_ix_i^{q-1}=c(B)\prod_iy_i^{q-1}=\prod_i\big(\sum_jb_{ij}y_j\big)^{q-1}=\prod_i\big(\sum_jc_{ij}x_j\big)^{q-1},
    \end{equation}
     where $(c_{ij})=BA$. Since the rightmost term is equal to $c(BA)\prod_ix_i^{q-1}$, it follows that $c(B)c(A)=c(BA)$. 
     
     We prove \eqref{equation: AB=C}. 
     The first and third equalities are clear from the definitions, and it remains to prove the second. The assignment $y_i\mapsto\sum_{j=1}^na_{ij}x_j$ induces an isomorphism of $R$-algebras 
    \[f\colon R[y_1,\dots,y_n]\to R[x_1,\dots,x_n], \]
     which sends the ideal $(y_1,\dots,y_n)$ onto $(x_1,\dots,x_n)$. On the other hand, the ideal $(y_1^q,\dots,y_n^q)$ coincides with 
\[F_q((y_1,\dots,y_n))\cdot R[y_1,\dots,y_n],\]
    where $F_q$ denote the $q$-th Frobenius endomorphism on $R[y_1,\dots,y_n]$. Since the Frobenius endomorphism is functorial with respect to $\F_p$-algebra homomorphisms, it follows that 
    \[
f((y_1^q,\dots,y_n^q))=(x_1^q,\dots, x_n^q). 
    \]
Hence, $f$ induces $R[y_1,\dots,y_n]/(y_1^q,\dots,y_n^q)\xrightarrow{\cong}R[x_1,\dots,x_n]/(x_1^q,\dots,x_n^q)$. Therefore, we obtain the equality 
    \[
    c(B)\prod_iy_i^{q-1}=\prod_i\big(\sum_jb_{ij}y_j\big)^{q-1}
    \]
    in the ring $R[x_1,\dots,x_n]/(x_1^q,\dots,x_n^q)$. 
    
\medskip

\noindent
Since $c(A)$ is multiplicative and $c(1)=1$, the assignment $A \mapsto c(A)$ defines a group homomorphism 
\[
c\colon GL_n(R)\to R^\times. 
\]
Since $R^\times$ is commutative, $c$ is trivial on the commutator subgroup $[GL_n(R),GL_n(R)]=SL_n(R)$. Hence, the  homomorphism $c$ factors as 
\[
GL_n(R)\xrightarrow{\det}R^\times\xrightarrow{\overline{c}}R^\times. 
\]
Since the map $\overline{c}$ is functorial in the $\F_p$-algebra $R$, it is induced by a homomorphism of algebraic groups 
\[
\G_{m,\F_p}\to \G_{m,\F_p}. 
\]
This shows that there exists an integer $m$ such that 
\[
c(A)=\det(A)^m. 
\]
Evaluating at $A=\mathrm{diag}(a,1,\dots,1)$, we obtain $m=q-1$. 
\end{proof}

\begin{proof}[Proof of Proposition~\ref{prop: CX and omegaX}]
By scalar extension, we have a morphism 
\[
\mathcal{C}_X^{(1)}\to\mathcal{H}om_{\cO_X}(\cO_X\otimes_{\cO_X^q}\cO_X,\cO_X),\quad C\mapsto [x\otimes y\mapsto xC(y)]. 
\]
The inclusion $S(J_q)\hookrightarrow \cO_X\otimes_{\cO_X^q}\cO_X$ then induces an $\cO_X$-linear morphism  
\[
\mathcal{C}_X^{(1)}\xrightarrow{\varphi} S(J_q)^\vee\xrightarrow{\cong} \omega_X^{\otimes(1-q)}, 
\]
where the last isomorphism is given by  Lemma~\ref{lem: S(J)=omega}. 

To prove Proposition~\ref{prop: CX and omegaX}, it suffices to show that $\varphi$ is an isomorphism. 
Since the assertion is local, we may assume that $X$ is affine. Set $R\coloneqq\Gamma(X,\cO_X)$. Moreover, we may assume that $R$ admits a $p$-basis $f_1,\dots,f_n$ over $R^p$. 
Let 
\[
C\colon F_{q\ast}\cO_X\to \cO_X
\]
be the $\cO_X$-linear morphism defined by 
\[
C(\prod_if_i^{e_i})=\begin{cases}0&\text{if $0\le e_i\le q-1$ and  $(e_1,\dots,e_n)\neq(q-1,\dots,q-1)$},\\
1&\text{if }(e_1,\dots,e_n)=(q-1,\dots,q-1). 
\end{cases}
\]
Since $\bigl(\prod_if_i\bigr)^{q-1}$ generates  $S(J_q)$, the image $\varphi(C)$ is a generator of $S(J_q)^\vee$. Therefore, $\varphi$ is a surjection between invertible $\cO_X$-modules, and hence an isomorphism. This proves the claim. 
\end{proof}

\newpage
\bibliographystyle{alpha}
\bibliography{sample}

\end{document}

%% file: header.tex
\usepackage[english]{babel}
\usepackage[letterpaper,top=2cm,bottom=2cm,left=3cm,right=3cm,marginparwidth=1.75cm]{geometry}
\usepackage{graphicx}
\usepackage[colorlinks=true, allcolors=blue]{hyperref}
\usepackage{amsmath, amssymb, amsfonts, amsthm,mathrsfs,comment,color,cleveref,mathtools}
\usepackage{hyperref}
\usepackage[mathcal]{euscript}
\usepackage[all]{xy}
\usepackage{tikz-cd}
\usepackage{enumitem}
\usepackage{colonequals}
\usepackage[new]{old-arrows}

\newcommand{\A}{{\bf A}}
\newcommand{\C}{{\bf C}}
\newcommand{\F}{{\bf F}}
\newcommand{\G}{{\bf G}}

\newcommand{\R}{{\bf R}}
\newcommand{\Q}{{\bf Q}}
\newcommand{\Z}{{\bf Z}}

\newcommand{\cO}{\mathcal{O}}
\newcommand{\cP}{\mathcal{P}}
\newcommand{\cD}{\mathcal{D}}

\newcommand{\sM}{\mathscr{M}}
\newcommand{\sN}{\mathscr{N}}
\newcommand{\sC}{\mathscr{C}}

\newcommand{\sS}{\mathscr{S}}
\newcommand{\sL}{\mathscr{L}}

\DeclareMathOperator{\BSR}{BSR}

\renewcommand{\to}{\longrightarrow}
\renewcommand{\hookrightarrow}{\longhookrightarrow}

\newcommand*{\Surjrightarrow}{%
    \mathrel{\ooalign{$\longrightarrow$\cr$\mkern 8.5mu\rightarrow$}}%
}
\renewcommand{\twoheadrightarrow}{\Surjrightarrow}

\newcommand{\cI}{\mathcal{I}}

\newcommand{\ev}{\mr{ev}}

\newcommand{\FJN}{{\rm FJN}}

\newcommand{\mr}{\mathrm}

\newcommand{\cal}{\mathcal}

\newtheorem{thm}{Theorem}[section]

\newtheorem{cor}[thm]{Corollary}
\newtheorem{lem}[thm]{Lemma}

\newtheorem{prop}[thm]{Proposition}
\newtheorem{ques}[thm]{Question}

\theoremstyle{definition}
\newtheorem{assumption}[thm]{Assumption}
\newtheorem{const}[thm]{Construction}

\newtheorem*{convention*}{Convention and Notation}
\newtheorem*{acknowledgments*}{Acknowledgments}
\newtheorem{defn}[thm]{Definition}

\newtheorem{example}[thm]{Example}

\newtheorem{rem}[thm]{Remark}
\newtheorem{remark}[thm]{Remark}

\makeatletter
  
  \@addtoreset{equation}{section}
\makeatother

\newcommand{\be}{\begin{eqnarray*}}
\newcommand{\ee}{\end{eqnarray*}}
\newcommand{\ben}{\begin{eqnarray}}
\newcommand{\een}{\end{eqnarray}}
\newcommand{\bi}{\begin{itemize}}
\newcommand{\ei}{\end{itemize}}
\newcommand{\bn}{\begin{enumerate}}
\newcommand{\en}{\end{enumerate}}
\newcommand{\bx}{\begin{eqnarray*}\begin{xy}\xymatrix}
\newcommand{\ex}{\end{xy}\end{eqnarray*}}
\newcommand{\bxy}{\begin{xy}\xymatrix}
\newcommand{\exy}{\end{xy}}
\newcommand{\bq}{\begin{ques}}
\newcommand{\eq}{\end{ques}}

\newcommand{\cC}{\mathcal{C}}

\makeatletter
\renewcommand*\env@matrix[1][*\c@MaxMatrixCols c]{%
	\hskip -\arraycolsep
	\let\@ifnextchar\new@ifnextchar
	\array{#1}}
\makeatother